\documentclass[12pt]{article}
\usepackage{setspace}
\usepackage{natbib}
\usepackage{amsmath,hyperref}
\usepackage{amsfonts}
\usepackage{amsthm}
\usepackage{amsmath}
\usepackage{amscd}
\usepackage[latin2]{inputenc}
\usepackage{t1enc}
\usepackage[mathscr]{eucal}
\usepackage{indentfirst}
\usepackage{graphicx}
\usepackage{graphics}
\usepackage{pict2e}
\numberwithin{equation}{section}
\usepackage[margin=2.9cm]{geometry}
\usepackage{epstopdf}
\usepackage[T1]{fontenc}
\usepackage{amssymb}
\usepackage[normalem]{ulem}
\usepackage{comment}
\usepackage{xcolor}
\useunder{\uline}{\ul}{}

 \theoremstyle{plain}
\newtheorem{Th}{Theorem}[section]
\newtheorem{Lemma}[Th]{Lemma}

\theoremstyle{definition}

\newtheorem{Rem}{Remark}
\newtheorem{?}[Th]{Problem}

\newcommand\relphantom[1]{\mathrel{\phantom{#1}}}

\newcommand{\be}{\mathbf{e}}

\newcommand{\bq}{\mathbf{q}}
\newcommand{\bs}{\mathbf{s}}
\newcommand{\bu}{\mathbf{u}}
\newcommand{\bw}{\mathbf{w}}
\newcommand{\bx}{\mathbf{x}}
\newcommand{\by}{\mathbf{y}}

\newcommand{\bA}{\mathbf{A}}
\newcommand{\bB}{\mathbf{B}}
\newcommand{\bC}{\mathbf{C}}
\newcommand{\bD}{\mathbf{D}}
\newcommand{\bG}{\mathbf{G}}

\newcommand{\bI}{\mathbf{I}}
\newcommand{\bM}{\mathbf{M}}
\newcommand{\bO}{\mathbf{O}}

\newcommand{\bQ}{\mathbf{Q}}

\newcommand{\bS}{\mathbf{S}}
\newcommand{\bU}{\mathbf{U}}
\newcommand{\bV}{\mathbf{V}}

\newcommand{\bT}{\mathbf{T}}
\newcommand{\bX}{\mathbf{X}}
\newcommand{\bY}{\mathbf{Y}}
\newcommand{\bZ}{\mathbf{Z}}
\newcommand{\bSigma}{\mathbf{\Sigma}}
\newcommand{\btSigma}{\tilde{\mathbf{\Sigma}}}
\newcommand{\bGamma}{\mathbf{\Gamma}}
\newcommand{\bLambda}{\mathbf{\Lambda}}

\newcommand{\diag}{\mbox{diag}}

\allowdisplaybreaks

\begin{document}

\title{Asymptotic independence of spiked eigenvalues and linear spectral statistics for large sample covariance matrices}
\author{Zhixiang Zhang\vspace{0.5cm}\\
 Nanyang Technological University \vspace{0.5cm} \\
Shurong Zheng \vspace{0.5cm}\\
Northeast Normal University \vspace{0.5cm} \\
Guangming Pan \vspace{0.5cm} \\
  Nanyang Technological University \vspace{0.5cm} \\
Pingshou  Zhong \vspace{0.5cm} \\
     University of Illinois at Chicago
}
\date{}

\maketitle

\begin{abstract}
We consider general high-dimensional spiked sample covariance models and show that their leading sample spiked eigenvalues and their linear spectral statistics are asymptotically
independent when the sample size and dimension are proportional to each other.
As a byproduct, we also establish the central limit theorem of the leading sample spiked eigenvalues by removing the block diagonal assumption on the
population covariance matrix, which is commonly needed in the literature. Moreover, we propose consistent estimators of  the $L_4$ norm of the spiked population eigenvectors.
Based on these results, we  develop a  new statistic to test  the equality of two spiked population covariance matrices. Numerical studies show that the new test procedure
is more powerful than some existing methods.



\end{abstract}

\begin{quote}
\noindent\baselineskip=24pt
\end{quote}


\section{Introduction}
Sample covariance matrices play a fundamental role in traditional multivariate statistics (see \cite{anderson1962introduction}). There has also been a significant interest in studying the
asymptotic properties of the eigenvalues and eigenvectors of the sample covariance matrix in the high-dimensional setting where the data dimension $p$ grows with the sample size $n$.
These asymptotic properties have been used to make statistical inference, such as hypothesis testing or parameter estimation, about the population covariance matrices of high-dimensionaldata.
Random matrix theory turns out to be a powerful tool for studying such asymptotic properties. One can refer to the monograph of \cite{bai2010} or the review paper of
 \cite{paul2014} for a comprehensive review.

 \indent The spiked covariance model appears naturally in the areas of wireless communication,
speech recognition,  genomics and genetics, finance, etc. It refers to a phenomenon that if the largest population eigenvalue is greater than some critical value then the largest sample eigenvalue will jump outside the bulk spectrum of the corresponding sample covariance matrix. Such a phenomenon has received much attention recently.
 The pioneer work of \cite{johnstone2001distribution} considered a special spiked model  with a $p\times p$ diagonal population covariance matrix
\begin{equation}
\bSigma = \text{diag} ( \alpha_1, \cdots, \alpha_K, 1, \cdots, 1),
\end{equation}
where  $\alpha_1 \geq \cdots \geq \alpha_K > 1$ are referred to as spikes and $K<\infty$ is the number of spikes.
\cite{baik2006eigenvalues} investigated the almost sure limits of the largest eigenvalues which  depend on the critical value $1+\sqrt{\gamma}$  when $p/n \rightarrow \gamma>0$. \cite{paul2007asymptotics} established the central limit theorem for the spiked sample eigenvalues under the Gaussian assumption on the data. 
\cite{bai2008central} extended Paul's results by removing the Gaussian assumption, but assumed a block diagonal structure on the population covariance matrix. \cite{bai2012} further generalized the spiked model by considering an arbitrary non-spiked part of the population covariance matrix instead of identity (but still a block diagonal structure). They classified the spikes into distant spikes and close spikes,  and discussed the almost sure limits for two types of spikes and established a central limit theorem for distant spikes.
\cite{jiang2019} obtained the CLT for the spiked eigenvalues without assuming a block diagonal structure on the population covariance matrix in addition to the finite 4th moment on the entries of the data matrix is relaxed to a tail probability decay.  However they had to pay a price that the largest entries of the population eigenvectors corresponding to the spikes tend to zero or the fourth moment of the underlying random variables must match with that of Gaussian distribution due to their moment matching method. Their condition about the largest entries of the population eigenvectors is hard to verify in practice and excludes all the diagonal population covariance matrices. In addition to  the above literature about the bounded spikes we would also like to mention that there is some literature about the unbounded spikes
and one may see \cite{cai2019} and the reference therein.

\indent The study of linear spectral statistics (LSS) of sample covariance matrices is another important topic in statistics and random matrix theory. The most influential work is \cite{bai2004}. They showed that the LSS of sample covariance matrices converge to normal distribution under some moment restrictions. Further refinements were carried out under different relaxed settings. \cite{pan2008central} improved Theorem 1.1 in \cite{bai2004} by removing the constraint on the fourth moment of the underlying random variables. \cite{pan2014comparison} showed the CLT of LSS for non-centered sample covariance matrices and discussed the difference between the centered and non-centered sample covariance matrices. \cite{zheng2015substitution} provided similar results for centralized and noncentralized sample covariance matrices in a unified framework. Furthermore \cite{najim2016gaussian} also provided CLT  in terms of vanishing Levy-Prohorov distance between the  LSS distribution and a Gaussian probability distribution.

\indent However, even though a lot of effort has been devoted to these two topics separately the relationship between the extreme eigenvalues and linear spectral statistics has not been well understood. \cite{baik2018ferromagnetic} obtained the joint normal distribution of the largest eigenvalue and LSS for a spiked Wigner matrix. They showed that the asymptotic joint distribution of the largest eigenvalue and LSS converge to a bivariate normal distribution with the covariance dependent on the the third moment of entries of the Wigner matrix. In this case the spiked eigenvalues and LSS are asymptotically independent if the third moment is zero.  Recently, \cite{li2019asymptotic} established that the extreme eigenvalues and the trace of sample covariance matrices are jointly asymptotically normal and independent for a block diagonal population covariance matrix.
 \\
\indent This paper focuses on more general spiked covariance matrices instead of block diagonal population covariance matrices. Specifically speaking, consider a population covariance matrix
\begin{equation}\label{050601}\bf \Sigma = \bV \begin{pmatrix} \bLambda_S& 0\\
0 & \bLambda_P \end{pmatrix} \bV^{\intercal}\end{equation}
where $\bV$ is an orthogonal matrix, $\bLambda_S$ is a diagonal matrix consisting of the bounded and descending spiked eigenvalues, and $\bLambda_P$ is the diagonal matrix of non-spiked eigenvalues.

 Our main contributions are summarized as follows. For the first time we establish CLT for the leading spiked eigenvalues of the sample covariance matrices with the general spiked covariance matrices $\bf\Sigma$ in (\ref{050601}). We need neither the block diagonal structure unlike Bai and Yao (2008, 2012) nor the maximum absolute value
of the eigenvector of the corresponding spikes tending to zero nor requiring the match of the 4th moment with  the standard Gaussian distribution (i.e.,
 the $4$-th moment is 3) unlike Jiang and Bai (2019). We also show that the extreme eigenvalues and  LSS of large sample covariance matrices are asymptotically independent.  Moreover consistent estimators of the $L_4$ norm of population eigenvectors associated with the leading sample spikes are proposed.


The remaining sections are organized as follows. Section 2 presents the main results about the asymptotic distribution of the largest sample spikes, the asymptotic independence between the largest sample spikes and the linear spectral statistics and the estimator of the population eigenvectors corresponding the largest spikes. We also explore an application of our main results in the two sample hypothesis testing
about  covariances in Section 2. The simulation is reported in Section 3.

Throughout the paper, we say that an event $\Omega_n$ holds with high probability if $P(\Omega_n) \geq 1-O(n^{-l})$ for some large constant  $l>0$. We use $I(\mathcal{A})$ to denote an indicator function of an event $\mathcal{A}$. The intersection of events $\mathcal{A}$ and $\mathcal{B}$ is denoted by $\mathcal{A}\cap \mathcal{B}$, or abbreviated by $\mathcal{A}\mathcal{B}$. The spectral norm of a matrix $M$ is denoted by $\| M \|$.


\section{The main results}

 Consider the data matrix $\bGamma \bX$,  where $\bGamma$ is a $p\times p$ deterministic matrix with  $\bGamma\bGamma^{\intercal}=\bf\Sigma$ and ${\bX} =(x_{ij})$ is  a $p\times n$ random matrix with entries $x_{ij}= n^{-1/2} q_{ij}$  where $q_{ij}$ are independent random variables satisfying Assumption 1 below.
 The sample covariance matrix has the form  $$ \bS_n  = \bGamma \bX\bX^\intercal \bGamma^\intercal.$$
  Order the eigenvalues of $\bS_n$ as $\lambda_1 \geq \lambda_2 \geq \cdots \geq\lambda_p$. The Denote the singular value decomposition of $\bGamma$ by
 \begin{equation}\label{32nan}\bf
 \Gamma = \bV \begin{pmatrix} \bLambda_S^{1/2} & 0\\
0 & \bLambda_P^{1/2} \end{pmatrix} \bU^{\intercal}
\end{equation} where $\bU$ and $\bV$ are orthogonal matrices, $\bLambda_S$ is a diagonal matrix consisting of the spiked eigenvalues in descending order and $\bLambda_P$ is the diagonal matrix of the non-spiked eigenvalues. To be more specific, denote the eigenvalues of the spiked part $\bLambda_S$ as  $\alpha_1 \geq \alpha_2 \geq\cdots \geq \alpha_K$, and eigenvalues of the non-spiked part as
$\alpha_{K+1}\geq\alpha_{K+2}\cdots \geq \alpha_{p}$.
Partition $\bU$ as $\bU =( \bU_1 , \bU_2 )$, where $\bU_1$ is a $p\times K$ submatrix of $\bU$. Let $\bu_i = (u_{i1},\cdots,u_{ip})^\intercal$ be the $i$-th column of $\bU_1$.  Define
\begin{equation}
\label{2o3jb}\bf \Sigma_{1P} = U_2\Lambda_p U_2^{\intercal}.
\end{equation}

\subsection{Limiting laws for spiked eigenvalues}
We first specify the assumptions for establishing CLTs of the leading sample spiked eigenvalues.

\noindent\textbf{Assumption 1.} The double array $\{q_{ij}:i = 1,\cdots,p, j = 1,\cdots,n\}$ consists of independent and identically distributed random variables, with $Eq_{11} =0, E|q_{11}|^2=1$ and $E|q_{11}|^4=\gamma_4$.\\
\textbf{Assumption 2.} $p/n = c_n \rightarrow c>0$ as $n\rightarrow \infty$.\\
\textbf{Assumption 3.} The $p \times p$ matrix $\bSigma =\bGamma\bGamma^{\intercal}$ has a bounded spectral norm. Furthermore, denote the empirical spectral distribution (ESD) of $\bSigma$ by $H_n$, which tends to a nonrandom limiting distribution $H$ as $p \rightarrow \infty$.\\

For the next assumption, we denote by $\Gamma_\mu$ the support for any measure $\mu$ on $\mathbb{R}$.  For $\alpha \notin \Gamma_H $ and $\alpha \neq 0$, define
  \begin{equation}\label{43afb}
  \psi(\alpha) := \alpha + c \alpha \int \frac{t}{\alpha -t}dH(t).
  \end{equation}
Its derivative is \begin{equation}
\psi^\prime(\alpha) = 1 - c\int \frac{t^2}{(\alpha-t)^2}dH(t).
\end{equation}
Define $\psi_n(\alpha)$ from (\ref{43afb}) with $H, c$  replaced by $H_n,c_n$.

\noindent\textbf{Assumption 4.}  Suppose that the population covariance matrix $\bSigma$ has $K$ spiked eigenvalues:  $\alpha_1 > \cdots >  \alpha_K$, lying outside the support of $H$, and satisfying $\psi^{\prime} (\alpha_k)>0$ for $1 \leq k\leq K$.

\cite{bai2012} provided a complete characterization of sample spikes according to the sign of $\psi^\prime(\alpha)$.
 If $\psi^\prime(\alpha) > 0$ then the corresponding sample spiked eigenvalues have limits outside the support of $F^{c,H}$, the limit of the empirical spectral distribution of $\bS_n$. They called them distant spikes in this case. Here we need to clarify that although \cite{bai2012} assumed that the population covariance matrices are block diagonal but this assumption is not essential and can be removed. This is because their method of deriving almost sure convergence relies on their Propositions 3.1 and 3.2 and these two results regarding the exact separation were first appeared in \cite{bai1998no, bai1999exact} without a block diagonal structure of the population covariance matrices.  



\noindent \textbf{Assumption 5.} Assume that for $i=1,\cdots,K$ the following limits exist:
$$\sigma_i^2 = \lim_{p\rightarrow \infty} (\gamma_4-3)\frac{\alpha_i^2 \{\psi^\prime(\alpha_i)\}^2}{\psi^2  (\alpha_i)}\sum_{j=1}^p u_{ij}^4 + 2\frac{\alpha_i^2 \psi^\prime(\alpha_i)}{\psi^2  (\alpha_i)} \;\mbox{and}\;$$
$$\sigma_{ij} = \lim_{p\rightarrow \infty}(\gamma_4-3)\frac{\alpha_i\alpha_j \psi^\prime(\alpha_i)\psi^\prime(\alpha_j)}{\psi(\alpha_i)\psi(\alpha_j)} \sum_{k=1}^p u_{ik}^2 u_{jk}^2. $$

We will show that the sample spiked eigenvalues $\lambda_i$ ($i=1,\cdots, K$) of ${\bf S}_n$ are associated with a random quadratic form given
by the following equation (see the details given in the proof of Theorem \ref{Th22}):
\begin{equation}
\det \{\bLambda_S^{-1}-\bU_1\bX(\lambda_i \bI-\bX^\intercal \bSigma_{1P} \bX)^{-1}\bX^\intercal \bU_1^\intercal\} = 0.
\end{equation}
Thus, our results rely on 
a new technique tool, a CLT for a type of random quadratic forms.  The result in Theorem \ref{Th21} is crucial to removing the block diagonal structure of the population covariance matrices
(hence the proof of Theorems \ref{Th22} and \ref{Th23} below). It can be of independent interest.

\begin{Th}\label{Th21}
Suppose that Assumptions 1-3 hold. 
Moreover, suppose that the non-random orthogonal unit vectors $\bw_1$ and $\bw_2$ satisfy $\bw_1^{\intercal}\bU_2 = \bw_2^{\intercal}\bU_2 = 0$ and $\bw_1^{\intercal}\bw_2=0$, and $\alpha$ satisfies $\psi^{\prime}(\alpha) > 0$.  Then
\begin{equation}\label{324hi} \frac{\sqrt{n}}{\tilde{\sigma}_1}\Bigg (\bw_1^{\intercal}\bX(\bI-\frac{1}{\psi_n(\alpha)}\bX^{\intercal}\bSigma_{1P}\bX)^{-1}\bX^{\intercal}\bw_1-\frac{\psi_n(\alpha)}{\alpha}\Bigg)\overset{D}\rightarrow N(0,1)\end{equation}
and
\begin{equation}\label{ah72h} \frac{\sqrt{n}}{\tilde{\sigma}_{12}} \bw_1^{\intercal}\bX(\bI-\frac{1}{\psi_n(\alpha)}\bX^{\intercal}\bSigma_{1P}\bX)^{-1}\bX^{\intercal}\bw_2\overset{D}\rightarrow N(0,1)\end{equation}
where $\tilde{\sigma}_1^2 =\psi^2(\alpha)\{ (\gamma_4-3) \sum_{i=1}^p w_{1i}^4 + {2}/{\psi^\prime(\alpha)}\}/\alpha^2$,
$\tilde{\sigma}_{12}^2=\{(\gamma_4-3)\psi^2(\alpha)\}\sum_{i=1}^p w_{1i}^2w_{2i}^2/{\alpha^2}$  and  $w_{ij}$ is the $j$-th element of $\bw_i, i =1, 2$.
\end{Th}

We are ready to provide the central limit theorem for the sample spiked eigenvalues. We consider the case when the eigenvalues of $\bLambda_S$ are all simple first. 
\begin{Th}\label{Th22} Let $\theta_i = \psi_n(\alpha_i),i=1,\cdots,K$. Suppose that Assumptions 1-5 hold. Then for all $i = 1,2,\cdots, K$,
\begin{equation}\label{21thm}\sqrt{n}\frac{\lambda_i-\theta_i}{\theta_i}\stackrel{D}\rightarrow N ( 0, \sigma_i^2).\end{equation}
 Moreover, for any fixed $1\leq r \leq K$, \begin{equation}\label{22thm}
\Big( \sqrt{n}\frac{\lambda_1-\theta_1}{\theta_1} , \cdots,  \sqrt{n}\frac{\lambda_r-\theta_r}{\theta_r} \Big) \stackrel{D}\rightarrow N ( 0,  \bSigma_{\lambda r}),
\end{equation}
where $ \bSigma_{\lambda r} =(\bSigma_{\lambda r, ij}) $ with \begin{equation*} \bSigma_{\lambda r, ij}=
\left \{
\begin{array}{rrrrl}
 \sigma_i^2, \; i=j \\
\sigma_{ij}, \. i\neq j \\
\end{array}
\right..
\end{equation*}
\end{Th}
\begin{Rem} Since the convergence rate of $c_n\rightarrow c$ and $ H_n \rightarrow H$ can be arbitrarily slow, $\theta_i = \psi_n(\alpha_i)$ is used in the CLT,  rather than $\psi(\alpha_i)$, which is the almost sure limit of $\lambda_i$.
\end{Rem}
\begin{Rem}
If we only consider the asymptotic distribution for an individual sample spiked eigenvalue,  Assumption 5 is not needed, since $\sqrt{n}(\lambda_i - \theta_i)/\theta_i$ can be normalized by
$ \big[(\gamma_4-3){\alpha_i^2 \{\psi^\prime(\alpha_i)\}^2}\sum_{j=1}^p u_{ij}^4+ 2{\alpha_i^2 \psi^\prime(\alpha_i)}\big]/{\psi^2  (\alpha_i)}$.
\end{Rem}
\begin{Rem}
Compared with earlier asymptotic results on spiked eigenvalues of sample covariance matrices obtained by \cite{bai2008central,bai2012} and \cite{li2019asymptotic}, we do not assume a block diagonal structure on population covariance matrices. Moreover \cite{bai2008central,bai2012} and \cite{jiang2019} did not consider the joint distribution of the different leading sample spiked eigenvalues corresponding to the different population eigenvalues. Instead they considered the joint distribution of the different leading sample spiked eigenvalues corresponding to the same population eigenvalues.
\end{Rem}


 We next consider the case when the multiplicity of the spiked eigenvalues of $\bLambda_S$ are more than one.

\noindent\textbf{Assumption 6.} Suppose that the population covariance matrix $\bSigma$ has  $K$ spiked eigenvalues:
$\alpha_1 > \cdots > \alpha_\mathcal{L}$ with respective multiplicities $m_1, \cdots, m_\mathcal{L}$, laying outside the support of  $H$, and satisfying $\psi^{\prime} (\alpha_k)>0$ for $1 \leq k\leq \mathcal{L}$. Furthermore, we assume that the following limits exist for $i=1, \cdots, \mathcal{L}$:
\begin{eqnarray}\label{cocar}\begin{aligned}&g(r_i, k_1,l_1,k_2,l_2) =\lim_{p\rightarrow \infty}  (\gamma_4-3)\frac{\alpha_i^2 \{\psi^\prime(\alpha_i)\}^2}{\psi^2  (\alpha_i)} \sum_{j=1}^p u_{r_i+k_1,j} u_{r_i+l_1,j} u_{r_i+k_2,j} u_{r_i+l_2,j}\\&\relphantom{EE}+ \frac{\alpha_i^2 \psi^\prime(\alpha_i)}{\psi^2  (\alpha_i)}\{(\bu_{r_i+k_1}^\intercal \bu_{r_i+k_2})(\bu_{r_i+l_1}^\intercal \bu_{r_i+l_2})+(\bu_{r_i+k_1}^\intercal \bu_{r_i+l_2})(\bu_{r_i+k_2}^\intercal \bu_{r_i+l_1})\},\end{aligned}\end{eqnarray}
where $r_i:=\sum_{j=0}^{i-1} m_j$,  $m_0=0$ and $1\leq k_1, l_1, k_2, l_2 \leq m_i$.

\begin{Th}\label{Th23}  Suppose that  Assumptions 1, 2, 3 and 6 hold.
Then
\begin{equation}
\Big(\sqrt{n}\frac{\lambda_{r_i+1}-\theta_i}{\theta_i}, \cdots, \sqrt{n}\frac{\lambda_{r_i+m_i}-\theta_i}{\theta_i} \Big)
\end{equation}
converges weakly to the joint distribution of the eigenvalues of $m_i\times m_i$ Gaussian random matrix $\mathcal{\bG}_i$ with $E\mathcal{\bG}_i=0$ and covariance of $(\mathcal{\bG}_i)_{k_1,l_1}$ and $(\mathcal{\bG}_i)_{k_2,l_2}$  being $g(r_i,k_1,l_1,k_2,l_2)$ defined in \eqref{cocar}.
\end{Th}

\begin{Rem} This result is similar to those in \cite{bai2008central,bai2012}, Theorem 3.1 and Corollary 3.1 in \cite{jiang2019}. However we neither need a block diagonal population covariance structure as in \cite{bai2008central,bai2012} nor the maximum absolute
value of the eigenvector of the corresponding spikes tending to zero (i.e., $\max\limits_{1\leq i\leq K,1\leq j\leq K}|u_{ij}|\rightarrow 0$) nor requiring the match of the 4th moment with Gaussian distribution (i.e.,$\gamma_4=3$) as in  \cite{jiang2019}. The assumption [D] about the population eigenvectors in \cite{jiang2019} excludes all the diagonal population covariance matrices when $\max\limits_{1\leq i\leq K,1\leq j\leq K}|u_{ij}|\rightarrow 0$. Under their assumption [D] we have 
\begin{equation}
g(r_i, k_1,l_1,k_2,l_2) =
\begin{cases}
  {2\alpha_i^2 \psi^\prime(\alpha_i)}/{\psi^2  (\alpha_i)}  \hspace{0.5cm}  k_1=k_2=l_1=l_2\\
 {\alpha_i^2 \psi^\prime(\alpha_i)}/{\psi^2  (\alpha_i)} \hspace{0.5cm} k_1=k_2 \mbox{ and } l_1=l_2 \mbox{ or } k_1=l_2 \mbox{ and } l_1=k_2\\
0 \,  \hspace{0.5cm} \mbox{otherwise},
\end{cases}
\end{equation}
which is consistent with theirs.  
\end{Rem}

\subsection{Asymptotic joint distribution of sample spiked eigenvalues and linear spectral statistics}

We now turn to the asymptotic joint distribution of sample spiked eigenvalues and linear spectral statistics of sample covariance matrices. To this end, define
\begin{eqnarray}L_p(\varphi) = \sum_{i=1}^p \varphi(\lambda_i)- p\int \varphi(x)dF^{c_n,H_n}(x),
\end{eqnarray}
where $\varphi(x)$ is an analytic function on an open interval containing
$$[\liminf_n \lambda_{ \min}^\bSigma I_{(0,1)}(c)(1-\sqrt{c})^2,\limsup_n \lambda_{ \max}^\bSigma(1+\sqrt{c})^2].$$
Here $F^{c_n,H_n}$ is obtained from $F^{c,H}$ with $c$ and $H$ being replaced by $c_n$ and $H_n$ respectively. 


Let $F^{\bS_n}$ be the empirical spectral distribution (ESD) of the sample covariance matrix $\bS_n$. It is well known that $F^{\bS_n}$ under some mild assumptions converges weakly to a non-random distribution $F^{c,H}$ with probability one, whose Stieltjes transform is the unique solution in $\mathbb{C}^+$ to the equation  \begin{equation}\label{88am9}
m = \int \frac{1}{t(1-c-czm)-z}dH(t),  \mbox{for z} \in \mathbb{C}^+.
\end{equation}
 Also, the ESD of $\underline{\bS}_n = \bX^\intercal \bGamma^\intercal \bGamma \bX$ has an almost sure limit whose Stieltjes transform satisfies \begin{equation}\label{23ato}  z = -\frac{1}{\underline{m}}+c\int \frac{t}{1+t \underline{m}}dH(t).\end{equation}
\noindent\textbf{Assumption 7.} Suppose that \begin{equation}\frac{1}{p}\sum_{i=1}^p \be_i^\intercal \bGamma^\intercal(\underline{m}(z_1)\bGamma\bGamma^\intercal + \bI)^{-1}\bGamma \be_i \be_i^\intercal \bGamma^\intercal(\underline{m}(z_2)\bGamma\bGamma^\intercal + \bI)^{-1}\bGamma \be_i \rightarrow h_1(z_1, z_2)
\end{equation}
and\begin{equation}\frac{1}{p}\sum_{i=1}^p  \be_i^\intercal \bGamma^\intercal(\underline{m}(z)\bGamma\bGamma^\intercal + \bI)^{-2}\bGamma \be_i \be_i^\intercal \bGamma^\intercal(\underline{m}(z)\bGamma\bGamma^\intercal + \bI)^{-1}\bGamma \be_i \rightarrow h_2(z).
\end{equation}




\begin{Th}\label{Th2.4} Suppose that  Assumptions 1-5 and 7 hold, then for $1\leq i \leq K$,  $ \sqrt n(\lambda_i-\theta_i)/{\theta_i}$ and $L_p(\varphi)$ are asymptotically independent.
\end{Th}
\begin{Rem}
This theorem implies that the joint distribution of  $ \sqrt n(\lambda_i-\theta_i)/{\theta_i}$ and $L_p(\varphi)$ is bivariate normal with asymptotic independent marginal distribution. The result can be generalized to the joint distribution of the multiple spikes and LSS easily with the spiked eigenvalues part and the LSS part still being independent. For the results regarding to  marginal distribution of LSS, one can refer to \cite{bai2004} and \cite{pan2008central}.
\end{Rem}
\begin{Rem} Compared with Theorem 3.1 in  \cite{li2019asymptotic}, we have two advantages. Firstly, we don't need the block diagonal assumption on the population covariance matrices. Secondly, our
LSS is not restricted to the trace of sample covariance matrices.
\end{Rem}
\vspace{0.5cm}

\subsection{Estimating the population eigenvectors associated with the spiked eigenvalues}

This section is to explore the estimation of the population spiked eigenvectors associated with the simple spiked eigenvalues $\alpha_1,  \cdots   \alpha_K$ involved in \eqref{21thm}.
Although many studies of the spiked eigenvectors have been carried out, most of them have not provided consistent estimators for the population eigenvectors in terms of certain norm. 
 For example, \cite{paul2007asymptotics} established the almost sure limit of $\bu_i^{\intercal} \hat{\bu}_i$ and a CLT for $\hat{\bu}_i $ for any $1\leq i \leq K$ under the assumption that $\bX$ is Gaussian and $\bGamma$ is diagonal with the nonspiked covariance being identity. 
 \cite{ding2017} further characterized the limit of $\bu_i^{\intercal} \hat{\bu}_i$ for a general spiked model. However these results are not helpful for estimating the population eigenvectors in terms of certain norm. Our following theorem provides a consistent estimator of $\sum_{k=1}^p u_{ik}^4$ inspired by the results in \cite{mestre2008}, which
 considered an estimator of $\bs^\intercal \bu_i$ where $\bs$ is any fixed vector with a bounded norm in $R^p$ when the underlying random variables are continuous  with finite eighth order moments. 


\begin{Th}\label{Th2.5} Suppose that the assumptions of Theorem \ref{Th22} hold and $\bGamma$ is symmetric, i.e. the left orthogonal matrix $\bV$ in \eqref{32nan} equals $\bU$. Let $\hat{\bu}_i$ be eigenvectors of  $\bS_n$ associated with eigenvalue $\lambda_i$ and $\hat{u}_{ik}$ be the k-th coordinate of $\hat{\bu}_i$. For $1 \leq i \leq K$, $\sum_{k=1}^p u_{ik}^4$ is consistently estimated by $\sum_{j=1}^p
\{\sum_{k=1}^p \theta_i(k)\hat{u}_{kj}^2\}^2$,
where \begin{equation}\label{ceeg1}\begin{aligned}
&\theta_i(k) =
\left \{
\begin{array}{rrrrl}
-\phi_i(k), \hspace{1cm} k \neq i \\
1+ \varrho_i(k) , \hspace{1cm}  k=i \\
\end{array}
\right.,\\
&\phi_i(k) =\frac{\lambda_i}{\lambda_k-\lambda_i} - \frac{\nu_i}{\lambda_k-\nu_i},\\
&\varrho_i(k)= \sum_{j\neq i}^p \big(\frac{\lambda_j}{\lambda_k-\lambda_j} - \frac{\nu_j}{\lambda_k-\nu_j}\big),
\end{aligned}\end{equation}
and where $\nu_1 \geq \nu_2 \geq \cdots \geq \nu_p$ are the real valued solutions to the equation in x:
\begin{equation}\label{an9k0}\frac{1}{p}\sum_{i=1}^p \frac{\lambda_i}{\lambda_i-x} = \frac{1}{c}.\end{equation}
 When $c > 1$, take $\nu_n = \cdots =\nu_p =0$. In the expressions of $\phi_i(k)$ and $\varrho_i(k)$, we use the convention that any term of form $\frac{0}{0}$ is 0.
\end{Th}

\begin{Rem}
Table 5 below shows that such an estimator of $\sum_{k=1}^p u_{ik}^4$ is quite accurate.
\end{Rem}

\subsection{Testing  the equality of two spiked covariance matrices}

\indent This subsection is to explore an application of our results.  Consider the problem of testing the equality of  two spiked covariance matrices $\bSigma_1$ and $\bSigma_2$.
 Let  $\{\by_{1i}=\bSigma_1^{1/2}{\bq_{1i}}, i = 1,\cdots, n_1\}$ be  i.i.d $p$ variate random samples from the population $F_1$ with mean zero and covariance matrix $\bSigma_1$, and
  $\{\by_{2i}=\bSigma_2^{1/2}{\bq_{2i}}, i = 1,\cdots, n_2\}$ be  i.i.d  $p$ variate random samples from the population $F_2$ with mean zero and covariance matrix $\bSigma_2$. Suppose $F_1$ and $F_2$ are independent. Several tests on the hypothesis:
\begin{equation}\label{06251}
H_0 : \bSigma_1 = \bSigma_2  \quad \text{versus} \quad  H_1:\bSigma_1\neq \bSigma_2
\end{equation}
have been proposed under high-dimensional settings. To name a few,  \cite{li2012two} suggested a test based on an unbiased estimator for ${tr}[(\bSigma_1-\bSigma_2)^2]$. The test in \cite{cai2013} is motivated by studying the maximum of standardized differences between entries of two sample covariance matrices to deal with sparse alternatives. \cite{yang2017weighted} proposed a weighted statistic that is powerful for dense or sparse alternatives.


 Let $\bY_1=(\by_{11},\cdots, \by_{1n_1})$ and $\bY_2=(\by_{21},\cdots, \by_{2n_2})$. Denote $\bx_{1i}=n_1^{-1/2}\by_{1i}, i=1, \cdots, n_1$ and $\bx_{2i}=n_2^{-1/2}\by_{2i}, i=1,\cdots, n_2$. Let $\bX_1=(\bx_{11},\cdots, \bx_{1n_1})$ and $\bX_2=(\bx_{21},\cdots, \bx_{2n_2})$. Denote two sample covariance matrices by
 $$\bS_1=\frac{1}{n_1}\bY_1\bY_1^\intercal=\bSigma_1^{\frac{1}{2}} \bX_1\bX_1^\intercal \bSigma_1^{\frac{1}{2}}\quad\mbox{and}\quad
\bS_2=\frac{1}{n_2}\bY_2\bY_2^\intercal =\bSigma_2^{\frac{1}{2}} \bX_2\bX_2^\intercal \bSigma_2^{\frac{1}{2}}.$$ 
We also assume that the respective largest spike eigenvalues of $\bSigma_1$ and $\bSigma_2$ are simple for simplicity.
 Denote the largest eigenvalues of $\bS_1$ and $\bS_2$ as $\lambda_1(\bS_1)$ and $\lambda_1(\bS_2)$ respectively.  Denote the largest spiked eigenvalues of $\bSigma_k$ by $\alpha_{k1}, k=1,2$,  and the corresponding eigenvector by $\bu_{1,k}=(u_{11,k},\cdots,u_{1p,k})^\intercal, k=1,2$.  Let $\gamma_{4k}, k=1,2$ be the fourth moment of
 $\{q_{1ij},j=1,\cdots,p, i=1, \cdots, n_1\}$ and $\{q_{2ij},j=1,\cdots,p, i=1,\cdots, n_2\}$ respectively.
 A natural test statistic for (\ref{06251}) by using the largest eigenvalues and the linear spectral statistics is
 \begin{equation}\label{tes11}
\Big\{\sqrt{n}\frac{\lambda_1(\bS_1)-\lambda_1(\bS_2)}{\sigma_{spi}}\Big\}^2+\Big\{\frac{ tr(\bS_1)+ tr(\bS_1^2)- tr(\bS_2) - tr(\bS_2^2)}{\sigma_{lin}}\Big\}^2
\end{equation}
where
$$\sigma_{spi}^2 = \sigma_{spi1}^2+\sigma_{spi2}^2,  \quad \sigma_{lin}^2=\sigma_{lin1}^2+\sigma_{lin2}^2,
$$
$$\sigma_{spik}^2 = (\gamma_{4k}-3)\alpha_{k1}^2 (\psi^\prime(\alpha_{k1}))^2 \sum_{j=1}^p u_{1j,k}^4+ 2 \alpha_{k1}^2 \psi^\prime(\alpha_{k1}) ,\quad k=1,2$$ and
\begin{equation}\label{9anmj} \begin{aligned}&\sigma_{link}^2 = 8c_k r_{k4}+16c_k^2 r_{k3} r_{k1} + 8c_k r_{k3} +8c_k^3 r_{k2}(r_{k1})^2+8c_k^2 r_{k2} r_{k1}+4c_k^2(r_{k2})^2+2c_k r_{k2}\\&+
(\gamma_{4k}-3)[4c_k r_{k4}+8c_k^2 r_{k3} r_{k1}+4c_k r_{k3}+4c_k^3 r_{k2} (r_{k1})^2+4c_k^2 r_{k2} r_{k1}+c_k r_{k2}],\quad k=1,2 \end{aligned}\end{equation}
with $r_{km}= tr(\bSigma_k^m)/{p}$ and $c_k={p}/{n_k}$.  The expression \eqref{9anmj} is obtained by calculating the contour integral in (1.20) in \cite{pan2008central}. This statistic is modified further below.


 The statistic in (\ref{tes11}) is asymptotic $\chi_2^2$ under the null hypothesis by Theorem \ref{Th2.4}. We next develop the estimators of unknown parameters
 $\alpha_{1k}, \psi^\prime(\alpha_{1k}), \sum_{j=1}^p u_{1j,k}^4$, $\gamma_{4k}$ and $r_{km}$ for practical implementation. For notational simplicity, the population index $k$ is omitted and we aim to find estimators of
 $\alpha_{1}, \psi^\prime(\alpha_{1}), \sum_{j=1}^p u_{1j}^4, r_m =tr \bSigma_1^m/{p}$ and $\gamma_{4}$ associated with the population $F_1$.
 The similar estimators are applicable to $F_2$ as well.  In the following, we use $n$ to denote the sample size and recall $c_n=p/n.$

 The estimator of $\sum_{j=1}^p u_{1j}^4$ is given in Theorem \ref{Th2.5}.   For the estimation of $\alpha_1$, we use the result in \cite{bai2012estimation}
. Note that  $$\underline{m}_n^*(z) : = -\frac{1-c_n}{z} + \frac{1}{n}\sum_{j\geq 2} \frac{1}{\lambda_j-z}\stackrel{a.s.}\rightarrow \underline{m}(z), $$ and  $$-\frac{1}{\underline{m}_n^*(\lambda_1)}\stackrel{a.s.}\rightarrow \alpha_1.$$ Therefore, as proposed by \cite{bai2012estimation}, $\alpha_1$ is estimated by
\begin{equation}\label{a7g0g}
\left(\frac{1-c_n}{\lambda_1} + \frac{1}{n}\sum_{j\geq 2} \frac{1}{\lambda_1-\lambda_j}\right)^{-1}.
\end{equation}
 Consider an estimator of $\psi^\prime(\alpha_1)$ now. Since $\psi(\cdot)$ is the inverse of the function $\alpha: x \mapsto -1/\underline{m}(x)$, we obtain
 \begin{equation}\label{qow23}
 \psi^\prime(\alpha_1) = \frac{1}{\alpha_1^2 \underline{m}^\prime\{\psi(\alpha_1)\}}.
 \end{equation}
  Thus we can estimate  $\underline{m}^\prime\{\psi(\alpha_1)\}$ by taking $z = \lambda_1$ in the expression of $ {d\underline{m}_n^*(z)}/{dz}$, which is
  \begin{equation}
   \label{pawh4}\frac{1-c_n}{\lambda_1^2} + \frac{1}{n}\sum_{j\geq 2} \frac{1}{(\lambda_j-\lambda_1)^2}.
  \end{equation}
  An estimator of $\psi^\prime(\alpha_1)$ follows by replacing $\alpha_1$  with (\ref{a7g0g}) and $ \underline{m}^\prime\{\psi(\alpha)\}$ with \eqref{pawh4} in (\ref{qow23}).

Let  $s_m =  tr (\bS_1^m)/p$.  According to Lemma 2.16 in \cite{yao2015sample} and Theorem 1.4 in \cite{pan2008central}, we have the following consistent estimators $A_m$ for $r_m, m=1,2,3,4,$
\begin{equation}\label{22joq}
\begin{aligned}
&A_1 = s_1,\quad
A_2 = s_2 - c_n (A_1)^2,\quad
A_3= s_3-3c_n A_1 A_2-c_n^2 (A_1)^3,\\
&A_4=s_4-2c_n(A_2)^2-4c_n A_1 A_3-6c_n^2(A_1)^2 A_2-c_n^3(A_1)^4.
\end{aligned}
\end{equation}

To estimate $\gamma_4$, notice that
\begin{equation}\label{326hn} \mathfrak{M}:=\frac{1}{p}{ E(\by_{11}^\intercal \by_{11}- tr\bSigma_1)^2}=
\frac{\gamma_4-3}{p} \sum_{i=1}^p (\bSigma_{1ii})^2+ 2 r_2,
\end{equation}
where $\bSigma_{1ii}$ and $\bS_{1ii}$  are, respectively, the $i$-th diagonal entry of $\bSigma_{1}$ and $\bS_1$.
Since $r_2$ can be estimated by $A_2$ above, we just need to find estimators of $\mathfrak{M}$ and $\sum_{i=1}^p (\bSigma_{1ii})^2/p$.
The following Lemma specifies their consistent estimators.

\begin{Lemma} \label{lem31} {\color{blue}Under Assumptions 1 and 2, and assume that $\bSigma_1$ has bounded spectral norm,}
we have \begin{equation} \label{aah92}\frac{1}{p}\sum_{i=1}^p \bS_{1ii}^2 -\frac{1}{p}\sum_{i=1}^p (\bSigma_{1ii})^2 \stackrel{i.p.} \rightarrow 0,
\end{equation}
and \begin{equation}\label{aah93}
\frac{1}{pn}\sum_{i=1}^n ( \by_{1i}^\intercal \by_{1i} -  tr \bS_1)^2 - \mathfrak{M} \stackrel{i.p.} \rightarrow 0,
\end{equation}
\noindent where  $\by_{1i}$ denotes the $i$-th observation from the first population.
\end{Lemma}

We assume that $ \sum_{i=1}^p (\bSigma_{1ii})^2/p$ does not  converge to 0, which is a  mild assumption for a population covariance matrix (otherwise the variances of the majority of the underlying random variables tend to zero).
From (\ref{326hn}) and Lemma \ref{lem31}, we propose a consistent estimator for $\gamma_4$ as follows
\begin{eqnarray}
\hat{\gamma}_4=\frac{{n}^{-1}\sum_{i=1}^n  (\by_{1i}^\intercal \by_{1i})^2-(1- {2}/{n})({ tr \bS_1})^2-2  tr\bS_1^2}{\sum_{i=1}^p \bS_{1ii}^2} + 3.
\end{eqnarray}

\indent Through our simulations, we find that the largest blue sample spiked eigenvalue and the full linear spectral statistics have large correlations although they are asymptotic uncorrelated in theory. 
This is due to the fact that $cov (\sqrt{n}\lambda_1/\psi(\alpha_1), \lambda_1) = O\{\psi(\alpha)/\sqrt{n}\}$ by Theorem \ref{Th22}, which is theoretically negligible. However, in practice, it may happen that $\psi(\alpha_1)$ is comparable to $\sqrt{n}$ (for example $\psi(8)=8+7c/8$ for model 1 in the simulation part) which results in significant covariance. Therefore, we correct the statistic in (\ref{tes11}) by removing the largest sample eigenvalue from the linear spectral statistics part. Actually, by using Slutsky's theorem and the fact that the single sample  spiked eigenvalues converge to a constant, our proof of Theorem \ref{Th2.4} also applies to the case when linear spectral statistics do not include the sample  spiked eigenvalues.

\indent It then suffices to recalculate the variance of LSS part without the largest sample spiked eigenvalue and estimate it. By taking contour $z$ enclosing all the sample eigenvalues except the largest spiked one,
and after analyzing the contour integral, we find that the more accurate variance is just to replace $r_m= {tr\bSigma^m}/{p}$ with $r_m- {\alpha_1^m}/{p}$, $ m = 1,2,3,4$ in (\ref{9anmj}), and
denote it by $(\sigma_{lin}^{(1)})^2$.  The corresponding estimator for $r_m-{\alpha_1^m}/{p}$ is obtained by replacing $s_m =   {tr\bS_1^m}/{n}$ with
$s_m- {\lambda_1^m}/{n}$ in (\ref{22joq}). Thus we find an estimator for $(\sigma_{lin}^{(1)})^2$, and denote it by $(\widehat{\sigma_{lin}^{(1)}})^2$. \\

Let\begin{equation}\begin{aligned}&M_n = \sqrt{n}\frac{\lambda_1(\bS_1)-\lambda_1(\bS_2)}{\widehat{\sigma_{spi}}};\\
&L_n=\frac{\sum_{i=1}^p  f\{\lambda_i(\bS_1)\}-f\{\lambda_i(\bS_2)\}}{\widehat{\sigma_{lin}}}, \text{where} \;  f(x) = x+x^2;\\
 &L_n^{(1)}=\frac{\sum_{i=2}^p  f\{\lambda_i(\bS_1)\}-f\{\lambda_i(\bS_2)\}}{\widehat{\sigma_{lin}^{(1)}}}, \text{where}\; f(x) = x+x^2;\\
&T_n = M_n^2 + (L_n^{(1)})^2.\end{aligned}\end{equation}
We then propose the above statistic $T_n$ for (\ref{06251}). As discussed before, under $H_0$, $T_n$ is asymptotically $\chi_2^2$, and $M_n,L_n, L_n^{(1)}$ are all asymptotically $N(0,1)$ as well.

\section{Simulations}
This section is to conduct the simulations to verify the performance of the earlier proposed statistics and the accuracy of the estimator of the population eigenvectors corresponding to the spikes.

\indent We introduce five covariance models to be used in simulations.
\begin{itemize}
\item Model 1: $\bSigma^{(1)}= \diag(8,1,\cdots,1)_{p\times p} $.
\item Model 2: $\bSigma^{(2)}= \diag(6,2,\cdots,2, 1, \cdots,1)_{p\times p} $ where the number of 2 is 10.
\item Model 3:  $\bSigma^{(3)}=\bO_p  \diag(12,d_2,\cdots,d_p)_{p\times p} \bO_p^\intercal $ where $d_i=3-1.5(i-1)/p$, and
$$\bO_p = \begin{bmatrix}
\bO_1 & 0 \\
0 & \bI_{p-3}
\end{bmatrix}$$ where $\bO_1$ is a  $3\times3$ orthogonal matrix.
\item Model 4: $\bSigma^{(4)}= \bO_p \diag(15,d_2,\cdots,d_p)_{p\times p}\bO_p^\intercal  $ where $d_i=3-2(i-1)/p$ and $\bO_p$ is the same as Model 3.
\item Model 5: $\bSigma^{(5)}= \diag(12,2,\cdots,2, 1, \cdots,1)_{p\times p} $ where the number of 2 is 10.
\end{itemize}
We consider two types of distribution for entries of $\bX_1$ and $\bX_2$: standard normal distribution, and $t_{10}/\sqrt{(5/4)}$.
We investigate the performance of $T_n$, and compare it with the tests in \cite{li2012two} and \cite{cai2013}, respectively, denoted as Chen's test and Cai's test.
The performance of $M_n$ and $L_n$ are also reported.

\subsection{Approximation accuracy}
 In Tables 1 and 2, we report the empirical sizes of testing $H_0:\bSigma_1=\bSigma_2=\bSigma^{(i)}$ for $\bSigma^{(i)}$ given by the above Model 1-5. The results listed in Table 1 are for standard normal distributed entries while Table 2 is for normalized $t_{10}$ distributed entries. We run 500 simulation replications for each test of population covariance matrices. The nominal test size is 0.05.  From the tables, we can see that the empirical sizes are around 0.05, which indicates that the $\chi_2^2$ approximation is accurate. We would like to point out that although $\lambda_1 \stackrel{a.s.}\rightarrow \psi(\alpha_1)$ as $n$ goes to infinity, the approximation is not accurate enough when $n =100$.  The estimating errors in (\ref{pawh4}) and (\ref{22joq}) are slightly amplified if $\lambda_1^2$ is involved. This accounts for the slightly smaller size for statistics $T_n, M_n$ and $L_n$ in Tables 1 and 2.\\

In Table \ref{estvec}, we record the performance of our estimator of $\sum_{i=1}^p u_{1i}^4$ for Model 4. The sample size is fixed to be 100, and for each dimension case, we run 500 replications and list the mean and variance. It can be seen that the estimator performs well.
\begin{table}[htbp!]
\caption{ Empirical sizes for testing $H_0:\bSigma_1=\bSigma_2=\bSigma^{(i)}$ for data generated from Model $i$ ($i=1,2,3,4, 5$) with N(0,1) entries. The sample size is 100 for both samples.}
\centering
\footnotesize
\begin{tabular}{lllllllllll}
\hline
Model               & $p$    & 40 & 60 & 80 & 100 & 120 & 150 & 240 & 300 &  \\ \hline
 		        &$T_n$    &0.042 &0.048  &0.038   &0.050  &0.042   &0.042   &0.052  &0.038  &  \\
	                &$M_n$   &0.044 & 0.044 &0.054   &0.042  &0.042   &0.042   &0.042  &0.038  &  \\
$\bSigma^{(1)}$	&$L_n$    &0.026 &0.030  &0.046  &0.042  & 0.038  &0.030   &0.066  &0.044  & \\
			& Cai &0.036 &0.048  &0.030   &0.058  &0.044   &0.040   &0.044  &0.046  & \\
 			&Chen  &0.084 &0.078  &0.072   &0.068  &0.070   &0.060   &0.048  &0.032  &\\  \\

			&$T_n$    &0.036 &0.030  &0.040   &0.052  &0.026   &0.063   &0.042  &0.048  &  \\
			&$M_n$   &0.028 &0.032  &0.052   &0.042  &0.040   &0.042   &0.044  &0.046  &  \\
$\bSigma^{(2)}$ 	&$L_n$    &0.030 &0.030  &0.042   &0.042  &0.036   &0.054   &0.044  &0.044  & \\
			& Cai &0.044 & 0.038 &0.048   &0.058  &0.044   &0.048   &0.044  &0.038  & \\
			&Chen  &0.056 &0.050  & 0.066  &0.056  &0.046   &0.053   &0.054  &0.056  &\\  \\

			&$T_n$    &0.024 &0.030  &0.046  &0.034  &0.040   &0.044   &0.042  &0.068  &  \\
			&$M_n$   &0.038&0.036 &0.044  &0.072   &0.034  &0.034   &0.038    &0.066  &  \\
$\bSigma^{(3)}$	&$L_n$    &0.040 &0.032  &0.036   &0.046  &0.050   &0.038   &0.050  &0.070  & \\
			& Cai &0.040 &0.046  &0.048   &0.036  &0.032   &0.054   &0.048  &0.028  & \\
			&Chen  &0.062 &0.060  &0.050   &0.056  &0.058   &0.050   &0.044  &0.038  &\\ \\

			&$T_n$    &0.056 &0.044  &0.048  &0.034  &0.048   &0.044   &0.030  &0.030  &  \\
			&$M_n$   &0.046 &0.040 &0.043  &0.040   &0.044  &0.044   &0.032    &0.030  &  \\
$\bSigma^{(4)}$&$L_n$    &0.034 &0.040  &0.038   &0.032  &0.038   &0.052   &0.048  &0.036  & \\
			& Cai &0.046 &0.056  &0.050   &0.040  &0.054   &0.054   &0.026  &0.040  & \\
			&Chen  &0.072 &0.064  &0.086   &0.066  &0.058   &0.058   &0.030  &0.042  &\\ \\
			
			&$T_n$    &0.028 &0.052  &0.040  &0.046  &0.060   &0.056   &0.046  &0.058  &  \\
			&$M_n$   &0.042 &0.028 &0.046  &0.036   &0.038  &0.042   &0.030    &0.052  &  \\
$\bSigma^{(5)}$&$L_n$    &0.026 &0.024  &0.042   &0.032  &0.042   &0.034   &0.038  &0.040  & \\
			& Cai &0.046 &0.054  &0.054   &0.044 &0.046   &0.038   &0.028  &0.042  & \\
			&Chen  &0.074 &0.062  &0.090   &0.054  &0.064   &0.052   &0.078  &0.056  &\\

\hline
\end{tabular}
\end{table}

\begin{table}[]
\caption{ Empirical sizes for testing $H_0:\bSigma_1=\bSigma_2=\bSigma^{(i)}$ for data generated from Model $i$ ($i=1,2,3,4,5$) with $t_{10}/\sqrt{5/4}$ entries. The sample size is 100 for both samples.}
\centering
\footnotesize
\begin{tabular}{lllllllllll}
\hline
Model               &p    & 40 & 60 & 80 & 100 & 120 & 150 & 240 & 300 &  \\ \hline
 		        &$T_n$    &0.052 &0.034  &0.054  &0.038  &0.034   &0.042   &0.030  &0.040  &  \\
	                &$M_n$   &0.048 & 0.028 &0.042   &0.036  &0.042   &0.038   &0.036  &0.040  &  \\
$\bSigma^{(1)}$	&$L_n$    &0.034 &0.026  &0.042  &0.036  & 0.042  &0.046   &0.06  &0.056  & \\
			& Cai &0.036 &0.040  &0.040   &0.022  &0.022   &0.032   &0.034  &0.020  & \\
 			&Chen  &0.098 &0.102  &0.112   &0.094  &0.092   &0.054   &0.048  &0.050  &\\  \\

			&$T_n$     &0.050 &0.026  &0.020   &0.042  &0.058   &0.046  &0.048  &0.038  &  \\
			&$M_n$   &0.048 &0.028  &0.026   &0.034  &0.042   &0.030   &0.034 &0.038  &  \\
$\bSigma^{(2)}$ 	&$L_n$    &0.054 &0.026  &0.042   &0.048  &0.046   &0.052   &0.052  &0.052  & \\
			& Cai &0.036 & 0.036 &0.034   &0.032  &0.046   &0.020   &0.034  &0.026  & \\
			&Chen  &0.086 &0.058  & 0.082  &0.080  &0.072   &0.052   &0.056  &0.050  &\\  \\

			&$T_n$    &0.038 &0.030  &0.034  &0.046  &0.042   &0.030   &0.038  &0.048  &  \\
			&$M_n$  &0.044&0.036 &0.044  &0.042   &0.034  &0.026   &0.040    &0.018  &  \\
$\bSigma^{(3)}$	&$L_n$    &0.028 &0.038  &0.036   &0.040  &0.050   &0.042   &0.044  &0.046  & \\
			& Cai &0.028 &0.032  &0.032   &0.020  &0.030   &0.032   &0.038  &0.028  & \\
			&Chen  &0.084 &0.066  &0.072   &0.054  &0.066   &0.050   &0.038  &0.040  &\\ \\

			&$T_n$   &0.040 &0.038  &0.046  &0.034  &0.046   &0.046   &0.038  &0.040  &  \\
			&$M_n$   &0.042 &0.034 &0.034  &0.054   &0.030  &0.036  &0.038    &0.040  &  \\
$\bSigma^{(4)}$	&$L_n$    &0.018 &0.018  &0.032   &0.026  &0.032   &0.020   &0.046  &0.044  & \\
			& Cai &0.038 &0.042  &0.040   &0.028  &0.022   &0.024   &0.032  &0.018  & \\
			&Chen  &0.108 &0.078  &0.066   &0.090  &0.056   &0.094   &0.058  &0.064  &\\ \\
			
						&$T_n$    &0.044 &0.042  &0.052  &0.060  &0.030   &0.060   &0.034  &0.042  &  \\
			&$M_n$   &0.022 &0.042 &0.056  &0.032 &0.024  &0.038   &0.036    &0.026  &  \\
$\bSigma^{(5)}$&$L_n$    &0.022 &0.034  &0.028   &0.030  &0.018   &0.038   &0.034  &0.038  & \\
			& Cai &0.026 &0.038  &0.030   &0.038 &0.034   &0.046   &0.028  &0.032  & \\
			&Chen  &0.100 &0.118  &0.104   &0.086  &0.084   &0.104  &0.070  &0.048  &\\

\hline
\end{tabular}
\end{table}

\subsection{Power discussion}
We consider the power of tests for comparing three pairs of covariances $\bSigma^{(1)}$ vs $\bSigma^{(2)}$, $\bSigma^{(3)}$ vs $\bSigma^{(4)}$ and  $\bSigma^{(2)}$ vs $\bSigma^{(5)}$.
 The empirical powers of the above three comparisons are, respectively, summarized in Tables \ref{Comp12}, \ref{Comp34} and \ref{Comp25}.
In Tables \ref{Comp12}-\ref{Comp25}, we find that $T_n$ always outperforms Cai's test and Chen's test. In Table \ref{Comp12}, all the three tests i.e., $T_n$, Chen and Cai's tests have competitive powers.
However, in Table \ref{Comp34}, both Cai and Chen's tests lose powers while $T_n$ has significant better powers and the powers increase as $p$ increases.
For the test comparing $\bSigma^{(2)}$ and $\bSigma^{(5)}$ in Table \ref{Comp25}, Cai's test loses powers, while $T_n$ and Chen's tests have satisfactory powers.
The performance of $T_n$ is more stable than Chen's test as $p$ increases, and $T_n$ outperforms Chen's test for large enough $p$.
In fact, we can infer from \eqref{43afb} that the limit of difference of two sample spiked eigenvalues increases as $p$ increases, so it is understandable that $M_n$ has good powers for large $p$ cases.

We observe $T_n$ has good powers whenever the differences between the two covariances are introduced by either
the non-spike eigenvalues (Table \ref{Comp12} and \ref{Comp34}) or the spike eigenvalues (Table \ref{Comp25}).  
Specifically, for the tests comparing $\bSigma^{(1)}$ and $\bSigma^{(2)}$ (Table \ref{Comp12}), the main differences between $\bSigma^{(1)}$ and $\bSigma^{(2)}$ are from non-spike eigenvalues.
Thus, $M_n$ has relatively low powers in this scenario, and the powers of $T_n$ in Table \ref{Comp12} are inherited from the difference of LSS excluding the largest population eigenvalue, i.e. the statistic $L_n^{(1)}$. This
phenomenon can be also seen by comparing the powers of $L_n$ and $L_n^{(1)}$.
Note that $L_n$ does not have good powers because the largest eigenvalue in $\bSigma^{(2)}$ is smaller than that of $\bSigma^{(1)}$, which offsets the effect of $L_n^{(1)}$.
In Table \ref{Comp34}, the powers of $T_n$ are also mainly contributed by $L_n^{(1)}$, but different from the results in Table \ref{Comp12}, $L_n$ and $L_n^{(1)}$ both have powers close to 1 for large $p$.
However, in Table \ref{Comp25}, $M_n$ has significant powers but $L_n^{(1)}$ loses power because $\bSigma^{(2)}$ and $\bSigma^{(5)}$ shares the same eigenvalues except the large difference between their largest spiked population eigenvalues. In this scenario, the powers of $T_n$ are mainly due to the contribution of $M_n$.


\begin{table}[]
\caption{Empirical powers for testing $H_0:\bSigma_1=\bSigma_2$ where $\bSigma_1=\bSigma^{(1)}$ and $\bSigma_2=\bSigma^{(2)}$ with two types of data entries:  $N(0,1)$ and  $t_{10}/\sqrt{5/4}$. The sample size is 100 for both samples.}
\centering
\footnotesize
\begin{tabular}{lllllllllll}
\hline
{Data Entries}  &$p$    & 40 & 60 & 80 & 100 & 120 & 150 & 240 & 300 &  \\ \hline
 		        &$T_n$    &1.000 &1.000  &1.000  &0.994  &0.992   &0.978   &0.894  &0.822  &  \\
	                &$M_n$   &0.248 & 0.266 &0.224   &0.218  &0.218   &0.246   &0.264  &0.220  &  \\
$N(0,1)$	       &$L_n$    &0.194 &0.188  &0.232  &0.256  & 0.286  &0.230   &0.274  &0.270  & \\
 		        &$L_n^{(1)}$    &1.000 &1.000  &1.000 &1.000  & 0.996  &0.978   &0.884  &0.844  & \\
			& Cai &0.796 &0.668  &0.584   &0.530  &0.434   &0.372   &0.260  &0.194  & \\
 			&Chen  &0.852 &0.786  &0.722   &0.606  &0.548   &0.448   &0.332  &0.240  &\\

\\

			&$T_n$    &1.000 &0.998  &0.996   &0.980  &0.942   &0.898   &0.728  &0.600  &  \\
			&$M_n$   &0.148 &0.164 &0.150   &0.160  &0.198   &0.136   &0.194 &0.136  &  \\
$t_{10}/\sqrt{5/4}$&$L_n$  &0.142 &0.176  &0.182   &0.204  &0.200   &0.204   &0.200  &0.216  & \\
			&$L_n^{(1)}$   &1.000 &0.996  &0.994 &0.984  & 0.966  &0.924   &0.748  &0.632  & \\
			& Cai  &0.448 & 0.326 &0.236   &0.190  &0.162   &0.132   &0.066  &0.062  & \\
			&Chen  &0.814 &0.774  & 0.710  &0.586  &0.522  &0.458   &0.346  &0.258  &\\

\hline
\end{tabular}\label{Comp12}
\end{table}

\begin{table}[]
\caption{Empirical powers for testing $H_0:\bSigma_1=\bSigma_2$ where $\bSigma_1=\bSigma^{(3)}$ and $\bSigma_2=\bSigma^{(4)}$ with two types of data entries:  $N(0,1)$ and  $t_{10}/\sqrt{5/4}$. The sample size is 100 for both samples.}
\centering
\footnotesize
\begin{tabular}{lllllllllll}
\hline
Data Entries  &$p$    & 40 & 60 & 80 & 100 & 120 & 150 & 240 & 300 &  \\ \hline
 		        &$T_n$    &0.792 &0.938  &0.976  &0.994  &1.000   &1.000   &1.000  &1.000  &  \\
	                &$M_n$   &0.176 & 0.176 &0.156   &0.150  &0.142   &0.132   &0.098  &0.084  &  \\
$N(0,1)$	       &$L_n$    &0.040 &0.062  &0.178 &0.318  & 0.506  &0.762  &0.994  &1.000  & \\
			&$L_n^{(1)}$   &0.818 &0.964  &0.986 &0.998  & 1.000  &1.000   &1.000  &1.000  & \\
			&Cai &0.062 &0.050  &0.092   &0.074  &0.068   &0.076   &0.070  &0.052  & \\
 			&Chen  &0.250 &0.236  &0.206   &0.172  &0.174   &0.150   &0.130  &0.160  &\\

\\

			&$T_n$    &0.636 &0.818  &0.930   &0.954  &0.982   &0.998   &1.000  &1.000  &  \\
			&$M_n$   &0.158 &0.146 &0.104   &0.114  &0.136   &0.116   &0.072 &0.074  &  \\
$t_{10}/\sqrt{5/4}$&$L_n$  &0.024 &0.054  &0.144   &0.266  &0.374   &0.632   &0.986  &0.996  & \\
			&$L_n^{(1)}$   &0.686 &0.862  &0.952 &0.960  & 0.992  &1.000   &1.000  &1.000  & \\
			& Cai &0.032 & 0.058 &0.026   &0.052  &0.040   &0.030   &0.032  &0.036  & \\
			&Chen  &0.248 &0.236  & 0.212  &0.178  &0.216   &0.184   &0.168  &0.122  &\\
\hline
\end{tabular}\label{Comp34}
\end{table}

\begin{table}[]
\caption{Empirical powers for testing $H_0:\bSigma_1=\bSigma_2$ where $\bSigma_1=\bSigma^{(2)}$ and $\bSigma_2=\bSigma^{(5)}$ with two types of data entries:  $N(0,1)$ and  $t_{10}/\sqrt{5/4}$. The sample size is 100 for both samples.}
\centering
\footnotesize
\begin{tabular}{lllllllllll}
\hline
Data Entries  &$p$    & 40 & 60 & 80 & 100 & 120 & 150 & 240 & 300 &  \\ \hline
 		        &$T_n$    &0.870 &0.858  &0.864  &0.842  &0.842   &0.820   &0.830  &0.844  &  \\
	                &$M_n$   &0.910 & 0.940 &0.946 &0.926  &0.908   &0.888   &0.900  &0.906  &  \\
$N(0,1)$	       &$L_n$    &0.882 &0.892  &0.878 &0.856  & 0.818  &0.772   &0.722  &0.704  & \\
			&$L_n^{(1)}$   &0.046 &0.048  &0.046 &0.046  & 0.058  &0.056   &0.036  &0.062  & \\
			&Cai &0.236 &0.140  &0.130   &0.110  &0.098   &0.096   &0.076  &0.058  & \\
 			&Chen  &0.908 &0.872  &0.878   &0.854  &0.778   &0.730   &0.630  &0.596  &\\

\\

			&$T_n$    &0.706 &0.644  &0.670   &0.702  &0.694   &0.672   &0.670  &0.634  &  \\
			&$M_n$   &0.822 &0.748 &0.790   &0.786  &0.804   &0.786   &0.782 &0.770  &  \\
$t_{10}/\sqrt{5/4}$&$L_n$  &0.756 &0.696  &0.688   &0.682  &0.668   &0.600   &0.546  &0.490  & \\
			&$L_n^{(1)}$   &0.054 &0.046  &0.050 &0.040  & 0.050  &0.054   &0.044  &0.042  & \\
			& Cai &0.096 & 0.062 &0.058   &0.058  &0.060   &0.050   &0.034  &0.034  & \\
			&Chen  &0.838 &0.784  & 0.790  &0.768  &0.778   &0.704   &0.634  &0.588  &\\
\hline
\end{tabular}\label{Comp25}
\end{table}

\begin{table}
\caption{Empirical mean and variance of the proposed estimators for $\sum_{i=1}^p u_{1i}^4$, with true value $0.5317$ in Model 4. The sample size is 100 and
the simulation replication is 500.}
\centering
\footnotesize
\begin{tabular}{lllllllllll}
\hline
 { Data Entries}  & $p$  &40 &60  &80   &100  &120   &150 &240 &300 &\\ \hline
$N(0,1)$	  & mean &0.5387 &0.5417 &0.5390 &0.5404 &0.5443 &0.5430 &0.5532 &0.5489&\\
                         &var &0.0029 &0.0035 &0.0044 &0.0041 &0.0044 &0.0057 &0.0073 &0.0092 \\

$t_{10}/\sqrt{5/4}$	  & mean &0.5378 &0.5414 &0.5394 &0.5368 &0.5425 &0.5442 &0.5441 &0.5567&\\
                    		&var &0.0039 &0.0042 &0.0044 &0.0048 &0.0053 &0.0055 &0.0074 &0.0110 \\

\hline
\end{tabular}\label{estvec}
\end{table}

\section{Proof of Theorem \ref{Th21}}

This section is to give the proof of Theorem \ref{Th21}. We begin with a list of results.

1. Two matrix formulas: \begin{eqnarray}\label{maeq1}
\bA^{-1}-\bB^{-1}=\bB^{-1}(\bB-\bA)\bA^{-1}
\end{eqnarray}
\begin{eqnarray}\label{maeq2}
\bA(\bI+\bB\bA)^{-1}=(\bI+\bA \bB)^{-1}\bA
\end{eqnarray}

2. Let $X = (X_1,\cdots,X_n)$, where $X_i$'s are i.i.d real random variables with mean zero and variance one. Let $\bA=(a_{ij})_{n\times n}$ and $\bB =(b _{ij})_{n\times n}$ be two real or complex matrices. Then we have an identity
\begin{equation}\begin{aligned}\label{543ah} E(X^\intercal\bA X- tr \bA)(X^\intercal\bB X-tr\bB) = (E|X_1|^4-3)\sum_{i=1}^n a_{ii}b_{ii} + tr\bA \bB^\intercal + tr\bA \bB.\end{aligned}\end{equation}

\begin{Lemma}\label{Lem4.2} (Theorem 35.12 of \cite{billingsley1995}) Suppose that for each $n$, $Y_{n1},Y_{n2},\cdots, Y_{nr_n}$ is a real martingale difference sequence with respect to an increasing $\sigma$-field $\{\mathcal{F}_{nj}\}$ having second moments. If as $n\rightarrow \infty$,\\
$$ (i) \hspace{1cm}\sum_{j=1}^{r_n}E(Y_{nj}^2|\mathcal{F}_{n,j-1})\stackrel{i.p.}\rightarrow\sigma^2,$$where $\sigma^2$ is a positive constant, and for each $\epsilon >0$,
$$ (ii) \hspace{1cm}\sum_{j=1}^{r_n}E\{Y_{nj}^2 { I({|Y_{nj}|>\epsilon)}\}}\rightarrow 0,$$
then $$ \sum_{j=1}^{r_n} Y_{nj} \stackrel{D}\rightarrow N(0,\sigma^2).$$
\end{Lemma}
4. Suppose that entries of $\bx$ is truncated at $\eta_n n^{1/4}$ and centralized, i.e. $x_{ij} = \frac{1}{\sqrt{n}}q_{ij}$, where $q_{ij}$ satisfying Assumption 1, are truncated at $\eta_n n^{1/4}$ and centralized. $\bM, \bM_1$ and $\bM_2$ are $p\times p$ non-random matrices (or independent of $\bx$). $\bw$ is a $p\times 1$ non-random vector with a bounded spectral norm.  We conclude the following simple results from Lemma 2.2 in \cite{bai2004}:
\begin{equation}\label{lar11}E|\bx^\intercal \bM \bx- \frac{1}{n}tr \bM|^d  \leq C ||\bM||^d n^{- d/2},\end{equation}
\begin{equation} \label{lar22} E|\bx^\intercal \bM_1 \bw\bw^\intercal \bM_2\bx- \frac{1}{n}\bw^\intercal \bM_2\bM_1\bw |^d \leq C||\bM_1||^d ||\bM_2||^d \eta_n^{2d-4}n^{-d/2-1},\end{equation}
\begin{equation} \label{lar33} E|\bx^\intercal \bM_1 \bw\bw^\intercal \bM_2\bx|^d \leq C||\bM_1||^d ||\bM_2||^d \eta_n^{2d-4}n^{-d/2-1}.
\end{equation}

\noindent \textbf{Proof of Theorem \ref{Th21}}. We below only prove (\ref{324hi}) and the proof of (\ref{ah72h}) is similar. The overall strategy of the proof is to decompose
$\bw_1^{\intercal}\bX(\bI-\bX^{\intercal}\frac{\bSigma_{1P}}{\psi_n(\alpha)}\bX)^{-1}\bX^{\intercal}\bw_1$ into summation of martingale differences and then apply Lemma \ref{Lem4.2}.
We assume that $X$ has already been truncated at $\eta_n n^{1/4}$ and centralized according to the argument in the Appendix.\\

\noindent \textbf{CLT of the random part}.
Throughout the rest of the paper, let $\bx_k$ be the $k$-th  ($k=1,\cdots, n$) column of $\bX$, and $\be_k = (0,\cdots, 0,1,0,\cdots,0)$ be an $n$-dimensional vector with  the $k$-th element being 1.
We use $C$ to denote constants which may change from line to line. 
Introduce notations
\begin{eqnarray}\label{noset}\begin{aligned}&\theta =\lim \psi_n(\alpha) =\psi(\alpha), \quad \bX_k = \bX- \bx_k \be_k^{\intercal}, \quad \bX_{jk} = \bX- \bx_k \be_k^{\intercal}-\bx_j \be_j^{\intercal}, \quad  \btSigma_1=\frac{\bSigma_{1P}}{\psi_n(\alpha)},\\ &\bA= \bI_n -\bX^\intercal \tilde{\bSigma}_1\bX,\quad \bA_k= \bI_n -\bX_k^\intercal \tilde{\bSigma}_1\bX_k, \ \bD_k =\bI_p-\btSigma_1\bX_k \bX_k^\intercal,\ \bD = \bI_p-\tilde{\bSigma}_1 \bX\bX^{\intercal},
\\ & \bD_{jk}=\bI_p-\tilde{\bSigma}_1\bX_{jk}\bX_{jk}^\intercal, \quad\bB_k = \bD_k^{-1}\bw_1\bw_1^{\intercal}(\bD_k^\intercal)^{-1},\quad \delta_k = \bx_k^{\intercal}\bB_k\bx_k - \frac{1}{n}tr\bB_k \\ &  \alpha_k = \frac{1}{1-\bx_k^\intercal \tilde{\bSigma}_1(\bD_k^\intercal)^{-1}\bx_k},\quad \alpha_{jk}= \frac{1}{1-\bx_k^\intercal \tilde{\bSigma}_1 (\bD_{jk}^\intercal)^{-1}\bx_k}\quad \bar{\alpha}_k=\frac{1}{1-\frac{1}{n}tr\tilde{\bSigma}_1(\bD_k^\intercal)^{-1}},\\&  \bar{\alpha}_{jk} =  \frac{1}{1-\frac{1}{n}tr \tilde{\bSigma}_1 (\bD_{jk}^\intercal)^{-1}},  \quad a_n = \frac{1}{1-\frac{1}{n}E tr\tilde{\bSigma}_1(\bD_1^\intercal)^{-1}},\quad a_{1n} = \frac{1}{1-E\frac{1}{n}tr \tilde{\bSigma}_1 (\bD_{12}^\intercal)^{-1}}\\& \gamma_k = \bx_k^{\intercal}\tilde{\bSigma}_1(\bD_k^\intercal)^{-1}\bx_k - \frac{1}{n}tr\tilde{\bSigma}_1(\bD_k^\intercal)^{-1},\quad \gamma_{1k} = \bx_k^{\intercal}\tilde{\bSigma}_1(\bD_{1k}^\intercal)^{-1}\bx_k - \frac{1}{n}tr\tilde{\bSigma}_1(\bD_{1k}^\intercal)^{-1}.
\end{aligned}\end{eqnarray}
With the help of \eqref{maeq2}, it is not difficult to conclude the following facts:\begin{eqnarray}\label{s1jag}
\be_k^\intercal \bX_k^\intercal = \be_k^\intercal \bA_k^{-1}\bX_k^\intercal =0,
\end{eqnarray}and
\begin{eqnarray}\label{s2jag}
\be_k^\intercal \bA^{-1} \bX^\intercal = \be_k^\intercal \bX^\intercal \bD^{-1}=\bx_k^\intercal \bD_k^{-1}\alpha_k.
\end{eqnarray}
Note that the quantities defined in \eqref{noset} such as $\alpha_k$, $\bar{\alpha}_k$ and $a_n$ are not always bounded and the matrices such as $\bA$, $\bD_{k} $ are not always invertible.  So we introduce events

 \begin{equation}\label{8ewo3}\begin{aligned}
&\mathcal{B}_{1}=\{|| \tilde{\bSigma}_1\bX\bX^\intercal|| \leq 1-\epsilon\}, \;    \mathcal{B}_{1k}=\{||  \tilde{\bSigma}_1\bX_k \bX_k^\intercal|| \leq 1-\epsilon\} , \quad \mathcal{B}_{1jk}=\{  \tilde{\bSigma}_1\bX_{jk}\bX_{jk}^\intercal|| \leq 1-\epsilon\}, \\
&\mathcal{B}_{2k}= \{ |\bx_k^\intercal \tilde{\bSigma}_1 ({\bD_k^\intercal})^{-1}\bx_k - (1 +\frac{1}{\theta \underline{m}(\theta)})| < \epsilon\},  \;   \mathcal{B}_{2jk}= \{ |\bx_k^\intercal \tilde{\bSigma}_1 ({\bD_{jk}^\intercal})^{-1}\bx_k - (1 +\frac{1}{\theta \underline{m}(\theta)})| <\epsilon\}, \\& \mathcal{B}_{3k}=\{|\frac{1}{n}tr\tilde{\bSigma}_1 ({\bD_k^\intercal})^{-1}-(1 +\frac{1}{\theta \underline{m}(\theta)})| < \epsilon\} , \; \mathcal{B}_{3jk}=\{|\frac{1}{n}tr\tilde{\bSigma}_1({\bD_{jk}^\intercal})^{-1}-(1 +\frac{1}{\theta \underline{m}(\theta)})| < \epsilon\},\end{aligned}\end{equation}where $\epsilon$ is a small positive constant. Note that $\mathcal{B}_1\subseteq \mathcal{B}_{1k} \subseteq \mathcal{B}_{1jk}$.
Denote \begin{equation}\label{defevnt43}
 \mathcal{B}= \mathcal{B}_1\bigcap (\bigcap_{i=2,3}\bigcap_{k=1}^n \mathcal{B}_{ik})\bigcap (\bigcap_{i=2,3}\bigcap_{1\leq j\neq k\leq n} \mathcal{B}_{ijk}).
  \end{equation}
Then we have following lemma and the proof is postponed to the Appendix.
\begin{Lemma}\label{Evbhighprob}
The event $\mathcal{B}$ holds with high probability (i.e.$P(\mathcal{B})=1-n^{-l}$ for any large constant $l$).
\end{Lemma}

This lemma ensures that it suffices to establish CLT of $\bw_1^{\intercal}\bX\bA^{-1}\bX^{\intercal}\bw_1 I(\mathcal{B})$. 
When the event $\mathcal{B}$ holds the terms $\alpha_k$, $\alpha_{jk}$, $\bar{\alpha}_k$ and $\bar{\alpha}_{jk}$ defined in (\ref{noset}) are bounded.  
We remark here that the more accurate definition of $a_n$ should be $$a_n = \frac{1}{1-\frac{1}{n}E tr\tilde{\bSigma}_1(\bD_1^\intercal)^{-1}I(\mathcal{B}_1)},$$ which is bounded for sufficient large $n$, see \eqref{rg42m} in the proof of Lemma \ref{Evbhighprob}.  The definition of $a_n$ in \eqref{noset} is just for notational simplicity. Another important fact of $a_n$  is  \begin{eqnarray}\label{anlim} \lim_{n\rightarrow \infty} a_n\rightarrow \frac{\psi(\alpha)}{\alpha}.
\end{eqnarray}
This is because we have $a_n \rightarrow -\theta \underline{m}(\theta)$ as $n\rightarrow \infty$, see \eqref{rg42m}. Recall that $\theta=\psi(\alpha)$. By the fact that $\psi$ is the inverse function of $\alpha: x \mapsto -1/\underline{m}(x)$, we have $\underline{m}(\theta)= -{1}/{\alpha}$. The above comment for $a_n$ also applies to $a_{1n}$, and the limit of $a_{1n}$ is also $ \theta \underline{m}(\theta) $. 

Let $E_0(\cdot)$ denote expectation, and $E_k (\cdot)$ denote the conditional expectation with respect to $\sigma$-field generated by $\bx_1,\cdots,\bx_k$.
We have
\begin{equation}\label{86jon}\begin{aligned} &\sqrt{n}\{\bw_1^{\intercal}\bX \bA^{-1}\bX^{\intercal}\bw_1 I(\mathcal{B})- E \bw_1^{\intercal}\bX \bA^{-1}\bX^{\intercal}\bw_1I(\mathcal{B}) \} \\& =
\sqrt{n}\sum_{k=1}^{n}(E_k-E_{k-1}) \{\bw_1^{\intercal}\bX \bA^{-1}\bX^{\intercal}\bw_1I(\mathcal{B})\}
\\&=\sqrt{n}\sum_{k=1}^{n} (E_k-E_{k-1}) \{\bw_1^{\intercal}\bX \bA^{-1}\bX^{\intercal}\bw_1I(\mathcal{B})  -\bw_1^{\intercal}\bX_k \bA_k^{-1}\bX_k^{\intercal}\bw_1I(\mathcal{B}_{1k}) \}
\\ &=\sqrt{n}\sum_{k=1}^{n} (E_k-E_{k-1}) \{\bw_1^{\intercal}\bX \bA^{-1}\bX^{\intercal}\bw_1I(\mathcal{B})  -\bw_1^{\intercal}\bX_k \bA_k^{-1}\bX_k^{\intercal}\bw_1I(\mathcal{B})\}+o_p(1)
 \\&=\sqrt{n}\sum_{k=1}^{n} (E_k-E_{k-1})\{\alpha_k \bx_k^{\intercal}B_k \bx_k I(\mathcal{B})\}+o_p(1),
\end{aligned}\end{equation}
where the last step uses the fact that by (\ref{s1jag}) and (\ref{s2jag})
\begin{eqnarray}\label{gn83g}
\begin{aligned}
&\bw_1^{\intercal}\bX \bA^{-1}\bX^{\intercal}\bw_1 -  \bw_1^{\intercal}\bX_k \bA_k^{-1}\bX_k^{\intercal}\bw_1\\
&=\bw_1^\intercal(\bX-\bX_k)\bA^{-1}\bX^\intercal \bw_1 + \bw_1^\intercal \bX_k(\bA^{-1}-\bA_k^{-1})\bX^\intercal \bw_1+ \bw_1^\intercal \bX_k \bA_k^{-1}(\bX^\intercal-\bX_k^\intercal) \bw_1\\
&=\bw_1^\intercal \bx_k \be_k^\intercal \bA^{-1} \bX^\intercal \bw_1+\bw_1^\intercal \bX_k \bA_k^{-1}(\bX^\intercal \btSigma_1 \bX- \bX_k^\intercal \btSigma_1 \bX_k)\bA^{-1}\bX^\intercal \bw_1
\\&=\bw_1^\intercal \bx_k\bx_k^\intercal \bD_k^{-1} \bw_1\alpha_k +\bw_1^\intercal \bX_k\bA_k^{-1}(\be_k \bx_k^\intercal\btSigma_1 \bX+\bX_k^\intercal \btSigma_1 \bx_k\be_k^\intercal)\bA^{-1}\bX^\intercal \bw_1
\\&=\bw_1^\intercal \bx_k\bx_k^\intercal \bD_k^{-1} \bw_1\alpha_k +\bw_1^\intercal \bX_k\bX_k^\intercal \btSigma_1  ({\bD_k^\intercal})^{-1}\bx_k\be_k^\intercal \bA^{-1}\bX^\intercal \bw_1
\\&=\bw_1^\intercal  ({\bD_k^\intercal})^{-1} \bx_k\bx_k^\intercal \bD_k^{-1} \bw_1\alpha_k=\alpha_k \bx_k^{\intercal}\bB_k \bx_k .
\end{aligned}
\end{eqnarray}

Note that $\alpha_k =\bar{\alpha}_k  + \bar{\alpha}_k^2 \gamma_k+\bar{\alpha}_k^2\gamma_k^2\alpha_k $.
It follows that
\begin{eqnarray}\label{jjn73}\begin{aligned}
\eqref{86jon}&=\sqrt{n}\sum_{k=1}^{n} (E_k-E_{k-1}) \big{[} (\bar{\alpha}_k  + \bar{\alpha}_k^2 \gamma_k+\bar{\alpha}_k^2 \gamma_k^2\alpha_k)(\delta_k+\frac{1}{n}trB_k)I(\mathcal{B})\big{]}\\&= \sqrt{n}\sum_{k=1}^{n}E_k [(\bar{\alpha}_k\delta_k + \frac{1}{n}\bar{\alpha}_k^2\gamma_k trB_k) I(\mathcal{B}_{1k}\mathcal{B}_{3k})] \\&\relphantom{EEEEEE}
+ \sqrt{n}\sum_{k=1}^{n} (E_k-E_{k-1})[( \bar{\alpha}_k^2 \gamma_k\delta_k +  \bar{\alpha}_k^2 \gamma_k^2\alpha_kx_k^{\intercal}B_k \bx_k)I(\mathcal{B})]+o_p(1),
 \end{aligned}\end{eqnarray}
 where in the second equality, we use $E_{k-1}\{\bar{\alpha}_k\delta_k  I(\mathcal{B})\} = E_{k-1}\{\bar{\alpha}_k\delta_k  I(\mathcal{B}_{1k}\mathcal{B}_{3k})\}+o_p(n^{-2})=o_p(n^{-2})$, and similarly, $E_{k-1}  \frac{1}{n}\bar{\alpha}_k^2\gamma_k trB_k  I(\mathcal{B})=E_{k-1}  \frac{1}{n}\bar{\alpha}_k^2\gamma_k trB_k  I(\mathcal{B}_{1k}\mathcal{B}_{3k})+o_p(n^{-2}) = o_p(n^{-2}) $.
 We below omit the indicator functions such as $I(\mathcal{B}), I(\mathcal{B}_{1k})$ for simplicity, but one should bear in mind that a suitable indicate function of events is needed whenever handling the inverses of random matrices.

 Using the Burkholder inequality, (\ref{lar11}) and (\ref{lar22}),  we have \begin{equation}\label{234yh} \begin{aligned} E|\sqrt{n}\sum_{k=1}^{n} (E_k-E_{k-1}) \gamma_k\delta_k|^{2}  \leq Cn^{2}(E|\gamma_k|^4)^{\frac{1}{2}}(E|\delta_k|^4)^{\frac{1}{2}} = o(1).  \end{aligned}\end{equation}
By similar arguments, together with the fact that $\bar{\alpha}_k$ and $\alpha_k$ are bounded, we have \begin{equation}\begin{aligned}\sqrt{n}\sum_{k=1}^{n} (E_k-E_{k-1})(\bar{\alpha}_k^2 \gamma_k^2\alpha_k\bx_k^{\intercal}\bB_k \bx_k) =o_p(1),\\
\sqrt{n}\sum_{k=1}^{n} E_k (\frac{1}{n}\bar{\alpha}_k^2\gamma_k tr\bB_k)=o_p(1).\end{aligned}\end{equation}
Therefore we only need to consider
$\sqrt{n}\sum_{k=1}^{n}E_k \bar{\alpha}_k\delta_k = \sqrt{n}\sum_{k=1}^{n}(E_k-E_{k-1}) (\bar{\alpha}_k \delta_k) = \sqrt{n} \sum_{k=1}^{n} (E_k-E_{k-1}) \big{[} (\bar{\alpha}_k-a_n)\delta_k \big{]}+\sqrt{n}a_n \sum_{k=1}^{n}E_k \delta_k$. Similarly to (\ref{234yh}), it is easy to get
\begin{equation*}E|\sqrt{n}\sum_{k=1}^{n} (E_k-E_{k-1}) [(\bar{\alpha}_k-a_n)\delta_k]|^{2 } = o(1).
\end{equation*}
Summarizing the above we conclude that
\begin{equation}\label{gnvb9}\sqrt{n}(\bw_1^{\intercal}\bX\bA^{-1}\bX^{\intercal}\bw_1- E \bw_1^{\intercal}\bX \bA^{-1}\bX^{\intercal}\bw_1) = \sqrt{n}\sum_{k=1}^{n}a_nE_k (\delta_k)+o_p(1).
\end{equation}

 Let $Y_{k}=E_k\delta_k = (E_k-E_{k-1})\delta_k$. By the fact that $a_n$ is bounded we obtain
 \begin{equation}\label{e7gab}
 \sum_{k=1}^n  a_n^2 E \left(nY_k^2 I{(|\sqrt{n}Y_k| \geq \epsilon)}\right) \leq  \frac{C}{\epsilon^2} \sum_{k=1}^n E|\sqrt{n}Y_k|^4  \leq  \frac{Cn^2}{\epsilon^2} \sum_{k=1}^n E|\delta_k|^4 =o(1),
  \end{equation}
 where in the last step we use \eqref{lar22}  and $\eta_n\to 0$.
 By Lemma \ref{Lem4.2} it suffices to verify
\begin{equation} \label{020720}\sum_{k=1}^{n}a_n^2E_{k-1}(nY_k^2)\stackrel{i.p.}\longrightarrow \sigma^2.\end{equation}
It follows from \eqref{543ah} that\begin{equation}\label{347bn}\begin{aligned}&n\sum_{k=1}^{n}E_{k-1}(Y_k^2) = \frac{1}{n}(\gamma_4-3)\sum_{k=1}^n\sum_{i=1}^{p}(E_k(\bB_k)_{ii})^2+\frac{2}{n}\sum_{k=1}^ntr(E_k\bB_k)^2.\end{aligned}\end{equation}
It suffices to find the limits of
\begin{equation}\label{aaa11}
\frac{1}{n}\sum_{k=1}^n\sum_{i=1}^{p}\Big(\be_i^\intercal E_k \bD_k^{-1}\bw_1\bw_1^{\intercal} ({\bD_k^\intercal})^{-1}\be_i \Big)^2,
\end{equation}
and\begin{equation}\label{bbbbb}\frac{2}{n}\sum_{k=1}^ntr\Big(E_k\bD_k^{-1}\bw_1\bw_1^{\intercal} ({\bD_k^\intercal})^{-1}\Big)^2.\end{equation}

First, we deal with  (\ref{aaa11}).  Write \begin{equation}\label{aaa33}\begin{split}&\be_i^\intercal(\bD_k^{-1}-E\bD_k^{-1})\bw_1\bw_1^\intercal (\bD_k^{\intercal})^{-1}\be_i\\&=\be_i^\intercal(\bD_k^{-1}-E\bD_k^{-1})\bw_1\bw_1^\intercal \big((\bD_k^{\intercal})^{-1}-E(\bD_k^{\intercal})^{-1}\big)\be_i +\be_i^\intercal(\bD_k^{-1}-E\bD_k^{-1})\bw_1\bw_1^\intercal E(\bD_k^{\intercal})^{-1}\be_i  \\&= \sum_{j_1,j_2 \neq k}\be_i^\intercal(E_{j_1}\bD_k^{-1}-E_{j_1-1}\bD_k^{-1})\bw_1 \bw_1^\intercal (E_{j_2}\bD_k^{-1}-E_{j_2-1}\bD_k^{-1})\be_i \\&\relphantom{=}+\sum_{j \neq k}\be_i^\intercal(E_{j}\bD_k^{-1}-E_{j-1}\bD_k^{-1})\bw_1 \bw_1^\intercal E(\bD_k^{\intercal})^{-1}\be_i .\end{split}\end{equation}
By the Burkholder inequality, \eqref{maeq1}, \eqref{lar33} and \eqref{aaa33}, we have
\begin{align}\label{lph0y}\begin{split}& \left[E \left| \sum_{i=1}^p E_k \be_i^\intercal(\bD_k^{-1}-E\bD_k^{-1})\bw_1\bw_1^\intercal \bD_k^{\intercal -1}\be_i\times \be_i^\intercal E_k\bD_k^{-1}\bw_1\bw_1^\intercal \bD_k^{\intercal -1} \be_i\right|\right]^2\\
&\leq \sum_{i=1}^{p}E |\be_i^\intercal(\bD_k^{-1}-E\bD_k^{-1})\bw_1\bw_1^\intercal \bD_k^{\intercal -1}\be_i|^2 \sum_{i=1}^{p}E|\be_i^\intercal \bD_k^{-1}\bw_1\bw_1^\intercal \bD_k^{\intercal -1}\be_i|^2 \\
& \leq C\sum_{i=1}^p \left[ E \left| \sum_{j_1\neq k}\alpha_{kj_1}\be_i^\intercal (E_{j_1}-E_{j_1-1})\bD_{kj_1}^{-1}\tilde\bSigma_1\bx_{j_1}\bx_{j_1}^\intercal \bD_{kj_1}^{-1}\bw_1 \right|^4\right]^{1/2}\\& \relphantom{\leq}\times\left[ E \left| \sum_{j_2\neq k}\alpha_{kj_2}\be_i^\intercal (E_{j_2}-E_{j_2-1})\bD_{kj_2}^{-1}\tilde\bSigma_1\bx_{j_2}\bx_{j_2}^\intercal \bD_{kj_2}^{-1}\bw_1 \right|^4\right]^{1/2} \\&\relphantom{\leq} + C\sum_{i=1}^p |\bw_1^\intercal E\bD_k^{\intercal -1}\be_i|^2 E \left| \sum_{j\neq k}\alpha_{kj}\be_i^\intercal (E_{j}-E_{j-1})\bD_{jk}^{-1}\tilde\bSigma_1\bx_{j}\bx_{j}^\intercal \bD_{jk}^{-1}\bw_1 \right|^2 \\& = O(n^{-1}),\end{split}\end{align}
where the second inequality uses \begin{equation}\begin{aligned}
&\sum_{i=1}^{n}E|\be_i^\intercal \bD_k^{-1}\bw_1\bw_1^\intercal \bD_k^{\intercal -1}\be_i|^2\leq E| \sum_{i=1}^{n} \bw_1^\intercal \bD_k^{\intercal -1}\be_i\be_i^\intercal \bD_k^{-1}\bw_1|^2 = E|\bw_1^\intercal \bD_k^{\intercal-1}\bD_k \bw_1|^2 = O(1).
\end{aligned}\end{equation}
\eqref{lph0y} implies \begin{equation}\begin{aligned}\frac{1}{n}\sum_{k=1}^n\sum_{i=1}^{p}(E_k\be_i^\intercal ( \bD_k^{-1}-E\bD_k^{-1})\bw_1\bw_1^{\intercal}\bD_k^{\intercal -1}\be_i)\times \be_i^\intercal E_k\bD_k^{-1}\bw_1\bw_1^\intercal \bD_k^{\intercal -1} \be_i\stackrel{i.p.}\rightarrow 0.\end{aligned}\end{equation}
Similarly, we have$$\frac{1}{n}\sum_{k=1}^n\sum_{i=1}^{p}\be_i^\intercal E\bD_k^{-1} \bw_1\bw_1^{\intercal}E_k(\bD_k^{\intercal -1}-E\bD_k^{\intercal -1})\be_i\times \be_i^\intercal E_k\bD_k^{-1}\bw_1\bw_1^\intercal \bD_k^{\intercal -1} \be_i\stackrel{i.p.}\rightarrow 0.$$
Therefore$$\frac{1}{n}\sum_{k=1}^n\sum_{i=1}^{p}\big (\be_i^\intercal E_k \bD_k^{-1}\bw_1\bw_1^\intercal \bD_k^{\intercal -1}\be_i -\be_i^\intercal E \bD_k^{-1}\bw_1\bw_1^\intercal E \bD_k^{\intercal -1}\be_i \big)\times \be_i^\intercal E_k\bD_k^{-1}\bw_1\bw_1^\intercal \bD_k^{\intercal -1} \be_i  \stackrel{i.p.}\rightarrow 0.$$

By similar arguments, we have$$\frac{1}{n}\sum_{k=1}^n\sum_{i=1}^{p}\be_i^\intercal E \bD_k^{-1}\bw_1\bw_1^\intercal E \bD_k^{\intercal -1}\be_i \times \be_i^\intercal (E_k\bD_k^{-1}\bw_1\bw_1^\intercal \bD_k^{\intercal -1}- E \bD_k^{-1}\bw_1\bw_1^\intercal E \bD_k^{\intercal -1})\be_i \stackrel{i.p.} \rightarrow 0.$$
Thus to find the limit of (\ref{aaa11}), it is equivalent to considering the limit of \begin{equation}\label{47sog}\frac{1}{n}\sum_{k=1}^n\sum_{i=1}^{p}(\be_i^\intercal E \bD_k^{-1}\bw_1\bw_1^\intercal E \bD_k^{\intercal -1}\be_i )^2.\end{equation}
Let \begin{eqnarray}\label{defT1}\bT= \bI - E\frac{1}{n}\sum_{k=2}^n\alpha_{1k}\tilde\bSigma_1. \end{eqnarray}
Using \eqref{32ndr} and the dominated convergence theorem,  it is easy to verify that
$$\lim_{n\rightarrow \infty} E\frac{1}{n}\sum_{k=2}^n\alpha_{1k}\rightarrow -\theta \underline{m}(\theta) =\frac{\psi(\alpha)}{\alpha}.$$ Since $\alpha$ has a positive distance to the support of $\bSigma_{1P}$, $\bT$ is invertible for large $n$.\\
Write
\begin{equation}\begin{aligned}\label{176af}
&E(\be_i^\intercal \bD_1^{-1}\bw_1)-\be_i^\intercal \bT^{-1}\bw_1\\
&=E\big[\be_i^\intercal \bT^{-1}(\sum_{j\geq2}\tilde\bSigma_1\bx_j\bx_j^\intercal-E\frac{1}{n}\sum_{j\geq2}\alpha_{1j}\tilde\bSigma_1)\bD_1^{-1}\bw_1\big]\\
&=\sum_{j\geq2}E\big[\alpha_{1j}\bx_j^\intercal \bD_{1j}^{-1}\bw_1\be_i^\intercal \bT^{-1}\tilde\bSigma_1\bx_j-\frac{E\alpha_{1j}}{n}\be_i^\intercal \bT^{-1}\tilde\bSigma_1\bD_1^{-1}\bw_1\big]\\
&=A_1+A_2+A_3,
\end{aligned}\end{equation}
where\begin{equation}\label{26ago}\begin{aligned}
&A_1= \sum_{j\geq2}E\big((\alpha_{1j}-\bar{\alpha}_{1j}) (\bx_j^\intercal \bD_{1j}^{-1}\bw_1e_i^\intercal \bT^{-1}\tilde\bSigma_1\bx_j-\frac{1}{n}\be_i^\intercal \bT^{-1}\tilde\bSigma_1\bD_{1j}^{-1}\bw_1)\big),\\
&A_2=\sum_{j\geq2}E \big(\frac{1}{n}\alpha_{1j}\be_i^\intercal \bT^{-1}\tilde\bSigma_1(\bD_{1j}^{-1}-\bD_1^{-1})\bw_1\big),\\
&A_3=\sum_{j\geq2}E \big(\frac{1}{n}\alpha_{1j}\be_i^\intercal \bT^{-1}\tilde\bSigma_1(\bD_{1}^{-1}-E\bD_1^{-1})\bw_1\big).
\end{aligned}\end{equation}

We prove $A_1 =O(n^{-1})$ first. Using  $ \alpha_{1j} -\bar{\alpha}_{1j}  = \bar{\alpha}_{1j}^2 \gamma_{1j}+\bar{\alpha}_{1j}^2\gamma_{1j}^2\alpha_{1j}$, we can write $A_1=A_{11}+A_{12}$, where
\begin{equation}\begin{aligned}
&A_{11}= \sum_{j\geq2}E\bar{\alpha}_{1j}^2\gamma_{1j}(\bx_j^\intercal \bD_{1j}^{-1}\bw_1e_i^\intercal \bT^{-1}\tilde\bSigma_1x_j-\frac{1}{n}\be_i^\intercal \bT^{-1}\tilde\bSigma_1\bD_{1j}^{-1}\bw_1),\\
&A_{12}= \sum_{j\geq2}E\alpha_{1j}\bar{\alpha}_{1j}^2\gamma_{1j}^2(\bx_j^\intercal \bD_{1j}^{-1}\bw_1e_i^\intercal \bT^{-1}\tilde\Sigma_1x_j-\frac{1}{n}\be_i^\intercal \bT^{-1}\tilde\Sigma_1\bD_{1j}^{-1}\bw_1).
\end{aligned}\end{equation}
\indent Using (\ref{543ah}), we obtain\begin{equation}\begin{aligned}\label{458ag}
&E[\bar{\alpha}_{1j}^2\gamma_{1j}(\bx_j^\intercal \bD_{1j}^{-1}\bw_1\be_i^\intercal \bT^{-1}\tilde\bSigma_1\bx_j-\frac{1}{n}\be_i^\intercal \bT^{-1}\tilde\bSigma_1\bD_{1j}^{-1}\bw_1)|\bX_j]\\
&\leq\frac{C}{n^2}\sum_{k=1}^{p}(\tilde\bSigma_1 \bD_{1j}^{-1})_{kk}( \bD_{1j}^{-1}\bw_1\be_i^\intercal \bT^{-1}\tilde\Sigma_1)_{kk}+C\frac{tr(\tilde\bSigma_1 \bD_{1j}^{-2}\bw_1e_i^\intercal \bT^{-1}\tilde\bSigma_1)}{n^2}\\
&+\frac{tr(\tilde\bSigma_1 \bD_{1j}^{-1}\tilde\bSigma_1^\intercal \bT^{-1}\be_i\bw_1^\intercal \bD_{1j}^{\intercal-1})}{n^2}.
\end{aligned}\end{equation}
The first summation is bounded by \begin{equation}\begin{aligned}
&(\sum_{k=1}^p|\be_k^\intercal \bD_{1j}^{-1}\bw_1|^2)^{1/2}(\sum_{k=1}^p|\be_i^\intercal \bT^{-1}\tilde\bSigma_1\be_k\be_k^\intercal\tilde\bSigma_1 \bD_{1j}^{-1}\be_k|^2)^{1/2}\\
&\leq ||\bD_{1j}^{-1}||^2 ||\tilde\bSigma_1||^2||\bT^{-1}|| <C.
\end{aligned}\end{equation}\\
Thus the first term in (\ref{458ag}) is $O(n^{-2})$. By similar but easier arguments, the second and third term also have bounds of the same order, so that we can conclude that\begin{equation} |A_{11}|=O(n^{-1}).\end{equation}
\indent For $A_{12}$, using \eqref{lar11} and \eqref{lar22}, we have
\begin{equation}\begin{aligned}|A_{12}|&\leq \sum_{j\geq2}(E|\alpha_{1j}\bar{\alpha}_{1j}^2\gamma_{1j}^2|^2)^{1/2} (E|\bx_j^\intercal \bD_{1j}^{-1}\bw_1\be_i^\intercal \bT^{-1}\tilde\bSigma_1\bx_j-\frac{1}{n}\be_i^\intercal \bT^{-1}\tilde\bSigma_1\bD_{1j}^{-1}\bw_1|^2)^{1/2}\\&= O(n^{-1}).
\end{aligned}\end{equation}
Thus \begin{equation}\label{A1qwe}|A_1| = O(n^{-1}). \end{equation}

Consider the terms $A_2$ and $A_3$ now in (\ref{26ago}). It follows from \eqref{maeq1}, \eqref{lar22}, \eqref{lar33} and the Burkholder inequality that
\begin{equation}\label{A2qwe}\begin{aligned}
|A_2|&= |\sum_{j\geq2}E (\frac{1}{n}\alpha_{1j}^2 \be_i^\intercal \bT^{-1}\tilde\bSigma_1 \bD_{1j}^{-1}\tilde\bSigma_1 \bx_j\bx_j^\intercal \bD_{1j}^{-1} \bw_1)|=O(n^{-1})\\
|A_3|&=|\sum_{j\geq2}E [\frac{1}{n}(\alpha_{1j}-a_n)\be_i^\intercal \bT^{-1}\tilde\bSigma_1(\bD_{1}^{-1}-E\bD_1^{-1})\bw_1]|\\
&=|\sum_{j\geq2}E [\frac{1}{n}(\alpha_{1j}-a_n)\be_i^\intercal \bT^{-1}\tilde\bSigma_1\big{(}\sum_{k=2}^n (E_k -E_{k-1})(\bD_{1}^{-1}-\bD_{1k}^{-1})\big{)}\bw_1]|\\
&\leq \frac{1}{n}\sum_{j\geq2}(E|\alpha_{1j}-a_n|^2)^{1/2}(E|\sum_{k=1}^n (E_k -E_{k-1})\be_i^\intercal \bT^{-1}\tilde\bSigma_1 \bD_{1k}^{-1}\tilde\bSigma_1 \bx_k\bx_k^\intercal \bD_{1k}^{-1}\bw_1|^2)^{1/2}\\
&=O(n^{-1}).
\end{aligned}\end{equation}
 From (\ref{176af}),(\ref{A1qwe}),(\ref{A2qwe}) and $\bw_1^\intercal \bU_2=0$ we have \begin{equation}\label{348cn}\begin{aligned} E(\be_i^\intercal \bD_1^{-1}\bw_1) &= \be_i^\intercal \bT^{-1}\bw_1 + O(n^{-1}) \\&= \be_i^\intercal \big(\bI +\bT^{-1}E\frac{1}{n}\sum_{k=2}^n\alpha_{1k}\tilde\Sigma_1\big) \bw_1+ O(n^{-1})  \\ &=w_{1i}+O(n^{-1}).\end{aligned}\end{equation}
Substituting these back into (\ref{aaa11}) we conclude that (\ref{47sog}) asymptotically equals $$(1+o(1))\sum_{i=1}^p w_{1i}^4.$$

For (\ref{bbbbb}),  it has the similar form to equation (4.7) in \cite{bai2007asymptotics}.
 Hence rewrite \eqref{bbbbb} as
\begin{equation}\label{236ho} \begin{aligned}
&\frac{2}{n}\sum_{k=1}^n E_{k-1}tr E_k (\bD_k^{-1}\bw_1\bw_1^\intercal \bD_k^{\intercal-1}) E_k (\bD_k^{-1}\bw_1\bw_1^\intercal \bD_k^{\intercal-1})\\
&=\frac{2}{n}\sum_{k=1}^n E_{k-1}(\bw_1^\intercal \bD_k^{\intercal-1} \breve{\bD}_k^{-1}\bw_1\bw_1^\intercal \breve {\bD}_k^{\intercal-1} \bD_k^{-1}\bw_1),
\end{aligned}\end{equation}
where $\breve{\bD}_k^{-1}$ is defined similarly as $\bD_k^{-1}$ by $(\bx_1,\cdots, \bx_{k-1},\breve{\bx}_{k+1},\cdots,\breve{\bx}_{n})$ and where $\breve{\bx}_{k+1}, \cdots, \breve{\bx}_{n}$ are i.i.d copies of $\bx_{k+1},\cdots,\bx_n$.
Following the argument similar to (4.7)-(4.22) of their work, we can obtain\begin{equation}\label{47agn} \begin{split}
E_{k-1}&(\bw_1^\intercal \bD_k^{\intercal-1} \breve{\bD}_k^{-1}\bw_1\bw_1^\intercal \breve {\bD}_k^{\intercal-1} \bD_k^{-1}\bw_1)\times [1 - \frac{k-1}{n} a_{1n}^2 \frac{1}{n}tr \tilde{\bSigma}_1 \bT^{-1} \tilde{\bSigma}_1 \bT^{-1}]\\
&= (\bw_1^\intercal \bT^{-2} \bw_1)^2[1+\frac{k-1}{n} a_{1n}^2 \frac{1}{n} E_{k-1} tr \bD_k^{-1}\tilde{\bSigma}_1 \breve{\bD}_k^{-1} \tilde{\bSigma}_1] +o_p(1).
\end{split}\end{equation}
Note that $\bw_1^\intercal \bT^{-2} \bw_1=1$ as in (\ref{348cn}).
Combining \eqref{maeq2} with (2.18) of \cite{bai2004} we have \begin{eqnarray}\label{m40b0}
E_{k-1}tr \bD_k^{-1}\tilde{\bSigma}_1 \breve{\bD}_k^{-1} \tilde{\bSigma}_1=\frac{tr\tilde{\bSigma}_1 \bT^{-1} \tilde{\bSigma}_1 \bT^{-1}+o_p(1)}{1-\frac{k-1}{n^2} \theta^2 (\underline{m}(\theta))^2 tr\tilde{\bSigma}_1 \bT^{-1} \tilde{\bSigma}_1 \bT^{-1}}
\end{eqnarray}
Recall that  $a_{1n} \rightarrow -\theta\underline{m}(\theta)$,  $\underline{m}(\theta) \rightarrow \underline{m}(\psi(\alpha)) = -\alpha^{-1}$, and $F^{\bSigma_{1P}}\rightarrow H$. Hence
\begin{equation}\label{43ahn}\begin{aligned}d:&= \lim \frac{a_{1n}^2}{n} tr   \tilde{\bSigma}_1 \bT^{-1} \tilde{\bSigma}_1 \bT^{-1}\\&=  \underline{ m}^2 (\theta) \int \frac{ct^2}{(1+t\underline{m}(\theta))^2} dH(t)\\&= \int \frac{ct^2}{(\alpha-t)^2}dH(t)\\&=1-\psi^{\prime}(\alpha).
\end{aligned}\end{equation}
By (\ref{47agn}) and (\ref{43ahn}), and similar argument to \eqref{348cn}, we obtain
\begin{equation}\label{378qw}\begin{aligned}
(\ref{236ho})&\rightarrow 2 (\bw_1^\intercal \bT^{-2} \bw_1)^2 (\int_0^1 \frac{1}{1-td}dt +\int_0^1 \frac{td}{(1-td)^2}dt )\\&=\frac{2}{1-d} = \frac{2}{\psi^{\prime}(\alpha)}.
\end{aligned}\end{equation}
Consequently, from (\ref{347bn})-(\ref{378qw}), and $a_n  \rightarrow \psi(\alpha)/\alpha$, we conclude that
 \begin{equation}\sum_{k=1}^{n}a_n^2E_{k-1}(nY_k^2)\rightarrow \frac{\psi^2(\alpha)}{\alpha^2} [(\gamma_4-3) \sum_{i=1}^p w_{1i}^4 +\frac{2}{\psi^{\prime}(\alpha)}] .\end{equation}

\textbf{Calculation of the mean.}
We next show that \begin{equation}\label{3o2c6} \sqrt{n} \bigg{(}Ew_1^{\intercal}\bX \bA^{-1}\bX^{\intercal}\bw_1-\frac{\psi_n(\alpha)}{\alpha}\bigg) \rightarrow 0\end{equation}
Let $\bX^0 = (\bx_1^0, \cdots, \bx_n^0)$  be a $p\times n$ matrix with entries consisting of i.i.d. Gaussian variables with mean 0 and variance $1/n$. Denote $\bA^{0} = \bI-(\bX^0)^\intercal\tilde{\bSigma}_1\bX^0.$ We finish (\ref{3o2c6}) by verifying that
\begin{equation}\label{36ahp}\begin{aligned}
\sqrt{n}E \Big(\bw_1^{\intercal}\bX\bA^{-1}\bX^{\intercal}\bw_1 -\bw_1^{\intercal}\bX^0 (\bA^0)^{-1}(\bX^0)^{\intercal}\bw_1\Big) \rightarrow 0,
\end{aligned}\end{equation}and
\begin{equation}\label{37heh} \sqrt{n}E\Big(\bw_1^{\intercal}\bX^0(\bA^0)^{-1}(\bX^0)^{\intercal}\bw_1-\frac{\psi_n(\alpha)}{\alpha}\Big)\rightarrow 0.
\end{equation}

Let \begin{eqnarray} \begin{aligned}
&\bZ_{k}^{1}=\sum_{i=1}^k \bx_i \be_i^\intercal+\sum_{i=k+1}^n \bx_i^0 \be_i^{\intercal},\quad  \bZ_{k}^{0}=\sum_{i=1}^{k-1} \bx_i \be_i^\intercal+\sum_{i=k}^n \bx_i^0 \be_i^{\intercal}, \\&\bZ_{k}=\sum_{i=1}^{k-1} \bx_i \be_i^\intercal+\sum_{i=k+1}^n \bx_i^0 \be_i^{\intercal}, \quad \hat{\bA}_k^1=\bI-\frac{1}{n}(\bZ_k^1)^{\intercal}\tilde{\bSigma}_1 \bZ_k^1,\\
& \hat{\bA}_k^0=\bI-\frac{1}{n}(\bZ_k^0)^{\intercal}\tilde{\bSigma}_1 \bZ_k^0,\quad  \hat{\bA}_k=\bI-\frac{1}{n}(\bZ_k)^{\intercal}\tilde{\bSigma}_1 \bZ_k.\quad
\hat{\bD}_k^1 = \bI-\tilde{\bSigma}_1\bZ_k^1(\bZ_k^{1})^\intercal, \\
& \quad  \hat{\bD}_k^0 = \bI-\tilde{\bSigma}_1\bZ_k^0(\bZ_k^0)^\intercal, \quad \hat{\bD}_k = \bI-\tilde{\bSigma}_1\bZ_k\bZ_k^\intercal, \quad  \alpha_k^{z}=\frac{1}{1-\bx_k^\intercal\tilde{\bSigma}_1 (\hat{\bD}_k^{\intercal})^{-1}\bx_k},\\&\alpha_k^{0z}=\frac{1}{1-\bx_k^{0\intercal}\tilde{\bSigma}_1(\hat{\bD}_k^{\intercal})^{-1}\bx_k^0},\quad
\bar{\alpha}_k^{z}=\frac{1}{1-\frac{1}{n}tr\tilde{\bSigma}_1(\hat{\bD}_k^{\intercal})^{-1}}.
\end{aligned} \end{eqnarray}

We prove \eqref{36ahp} first. To this end introduce the following notation
\begin{eqnarray}\label{270ag}\begin{aligned}
&\sqrt{n}E \Big[\bw_1^{\intercal}\bX\bA^{-1}\bX^{\intercal}\bw_1 -\bw_1^{\intercal}\bX^0 (\bA^0)^{-1}(\bX^0)^{\intercal}\bw_1\Big] \\
&=\sqrt{n}\sum_{k=1}^n E\Big[\bw_1^{\intercal}\bZ_k^1 (\hat{\bA}_k^1)^{-1}(\bZ_k^1)^{\intercal}\bw_1-\bw_1^{\intercal}\bZ_k^0 (\hat{\bA}_k^0)^{-1}(\bZ_k^0)^{\intercal}\bw_1\Big]\\
&=\sqrt{n}\sum_{k=1}^n E\Big[\bw_1^{\intercal}\bZ_k^1 (\hat{\bA}_k^1)^{-1}(\bZ_k^1)^{\intercal}\bw_1-\bw_1^{\intercal}\bZ_k \hat{\bA}_k^{-1}(\bZ_k)^{\intercal}\bw_1\Big] \\
&-\sqrt{n}\sum_{k=1}^n E\Big[\bw_1^{\intercal}\bZ_k^1 (\hat{\bA}_k^0)^{-1}(\bZ_k^1)^{\intercal}\bw_1-\bw_1^{\intercal}\bZ_k \hat{\bA}_k^{-1}(\bZ_k)^{\intercal}\bw_1\Big] \\
&=\sqrt{n} \sum_{k=1}^n E \Big[\alpha_k^z \bx_k^\intercal (\hat{\bD}_k)^{-1}\bw_1\bw_1^\intercal (\hat{\bD}_k^{\intercal})^{-1}\bx_k\Big]-\sqrt{n} \sum_{k=1}^n E \Big[\alpha_k^{0z} (\bx_k^{0})^\intercal (\hat{\bD}_k)^{-1}\bw_1\bw_1^\intercal (\hat{\bD}_k^{\intercal})^{-1}\bx_k^0\Big].
\end{aligned}\end{eqnarray}
For the first summation above, we write \begin{equation}\begin{split}
\sqrt{n} \sum_{k=1}^n E\Big[\alpha_k^z \bx_k^\intercal (\hat{D}_k)^{-1}\bw_1\bw_1^\intercal (\hat{\bD}_k)^{\intercal-1}\bx_k\Big]&=\sqrt{n} \sum_{k=1}^n E\Big[(\alpha_k^z-\bar{\alpha}_k^{z}) \bx_k^\intercal (\hat{\bD}_k)^{-1}\bw_1\bw_1^\intercal (\hat{\bD}_k)^{\intercal-1}\bx_k\Big]  \\&+\sqrt{n} \sum_{k=1}^n E\Big[\bar{\alpha}_k^{z}\bx_k^\intercal (\hat{\bD}_k)^{-1}\bw_1\bw_1^\intercal (\hat{\bD}_k^{\intercal})^{-1}\bx_k\Big].
\end{split}\end{equation}
By the same strategy used in estimation of $A_1$ in (\ref{26ago}), one can get $$\sqrt{n} \sum_{k=1}^n E \Big[(\alpha_k^z-\bar{\alpha}_k^{z}) \bx_k^\intercal (\hat{\bD}_k)^{-1}\bw_1\bw_1^\intercal (\hat{\bD}_k^{\intercal})^{-1}\bx_k \Big]= O(n^{-1/2}).$$
Thus, we have \begin{equation}\label{347hr}\sqrt{n} \sum_{k=1}^n E\Big[\alpha_k^z \bx_k^\intercal (\hat{\bD}_k)^{-1}\bw_1\bw_1^\intercal (\hat{\bD}_k^{\intercal})^{-1}\bx_k\Big] = \frac{1}{\sqrt{n}} \sum_{k=1}^n E\Big[\bar{\alpha}_k^{z}tr\big((\hat{\bD}_k)^{-1}\bw_1\bw_1^\intercal (\hat{\bD}_k^{\intercal})^{-1}\big)\Big] + O(n^{-1/2}).\end{equation}
Similarly, one can get \begin{equation}\label{484ph}\sqrt{n} \sum_{k=1}^n E\Big[\alpha_k^{0z} (\bx_k^{0})^\intercal (\hat{\bD}_k)^{-1}\bw_1\bw_1^\intercal (\hat{\bD}_k^{\intercal})^{-1}\bx_k^0 \Big] =\frac{1}{\sqrt{n}} \sum_{k=1}^n E\Big[\bar{\alpha}_k^{z}tr\big((\hat{\bD}_k)^{-1}\bw_1\bw_1^\intercal (\hat{\bD}_k^{\intercal})^{-1}\big)\Big]+  O(n^{-1/2}).\end{equation}
 (\ref{36ahp}) follows from (\ref{270ag}), (\ref{347hr}) and (\ref{484ph}).\\

Next we consider (\ref{37heh}).  By $\bw_1^\intercal \bU_2 = 0$, we know that \begin{equation}\label{43baq} E \bw_1^{\intercal}\bX^0 (\bA^0)^{-1}(\bX^0)^{\intercal}\bw_1=\frac{1}{n}E tr(\bA^0)^{-1}.   \end{equation}
Denote by $B_\epsilon(\theta)$ the ball with center $\theta$ and radius  $\epsilon$  in the complex plane.  By the strategy used in section 4 of \cite{bai2004}, one can prove that  $$\sup_{z\in B_\epsilon(\theta)}\sqrt{n} |E m_{F^{({\bX^{0})^\intercal} \bSigma_{1P}\bX^0}} (z)-\underline{m}_n(z)| \rightarrow 0.$$
Since $\psi_n(\alpha) \rightarrow \theta$ we have $$\sqrt{n}\Big(E m_{F^{({\bX^{0})^\intercal} \bSigma_{1P}\bX^0}} \{\psi_n(\alpha)\}-\underline{m}_n\{\psi_n(\alpha)\}\Big) = o(1).$$
This implies \begin{eqnarray}\label{8b906}\sqrt{n}\Big({\frac{1}{n}E tr(\bA^0)^{-1}-\frac{\psi_n(\alpha)}{\alpha}}\Big)\rightarrow 0.\end{eqnarray}
We conclude \eqref{37heh} from \eqref{8b906} and \eqref{43baq}.
\qed

\section{Proof of Theorems \ref{Th22} and \ref{Th23}}
\textbf{Proof of Theorem \ref{Th22}}.
First we give an outline of the proof, which is similar to the proof of Theorem 2.2 in \cite{cai2019}. As $n \rightarrow \infty$,  $\lambda_k\stackrel{a.s.} \rightarrow \psi(\alpha_k), 1 \leq k \leq K$,  which lie outside the support of $F^{c_n,H_n}$ by Theorem 4.1 in \cite{bai2012}. 
Moreover, no eigenvalue of  $\bX^\intercal \bSigma_{1P} \bX$ appears in a small interval that lies outside the support of $F^{c_n,H_n}$ and that contains $\psi(\alpha_k)$ almost surely by \cite{bai1998no}.
Therefore, $(\lambda_k \bI-\bX^\intercal \bSigma_{1P}  \bX)$ is invertible with probability 1. Conditional on this event, $\lambda_k$ solves the determinant equation
\begin{equation}
\label{kmv46} det (\bLambda_S^{-1}-\bU_1\bX(\lambda_k \bI-\bX^\intercal \bSigma_{1P} \bX)^{-1}\bX^\intercal \bU_1^\intercal) = 0.
\end{equation}
To prove (\ref{21thm}) we rewrite the normalized largest spike in terms of random quadratic forms (see (\ref{26hae}) below) from the determinant equation.  Then its limiting distribution is governed by the fluctuation of the random quadratic forms. The idea of proving (\ref{22thm}) is similar. 

We prove (\ref{21thm}) first. Define $\bB(x) = xI - \bX^{\intercal}\bSigma_{1P}\bX$ and  $\chi_i =  {(\lambda_i-\theta_i)}/{\theta_i}$. We only prove the central limit theorem for $\chi_1$, and the others can be similarly proved.

Using (\ref{maeq1}) we write \begin{eqnarray}\label{4bndg}\bU_1^{\intercal}\bX\bB^{-1}(\lambda_1)\bX^{\intercal}\bU_1 =\bU_1^{\intercal}\bX\bB^{-1}(\theta_1)\bX^{\intercal}\bU_1-\chi_1\theta_1\bU_1^{\intercal}\bX\bB^{-1}(\lambda_1)\bB^{-1}(\theta_1)\bX^{\intercal}\bU_1.\end{eqnarray}
Denote by $\delta_{ij}$ the Kronecker delta function; i.e. $\delta_{ii}=1$ and $\delta_{ij}=0$ if $i\neq j$.  By Theorem \ref{Th21} we have
\begin{eqnarray}\label{436hr}
&P(\max_{1\leq i,j\leq K}|\frac{\delta_{ij}}{\alpha_1}-\bu_i^\intercal \bX\bB^{-1}(\theta_1)\bX^\intercal \bu_j )|\geq \frac{\epsilon}{\sqrt{n}})\\
&\leq K^2P\Big(\sqrt{n}|\frac{\delta_{ij}}{\alpha_1}-\bu_i^\intercal \bX\bB^{-1}(\theta_1)\bX^\intercal \bu_j )|\geq \epsilon\Big)=O(1).
\end{eqnarray}

It follows that
\begin{equation}\label{46bdf} \begin{aligned} & \bLambda_S^{-1}- \bU_1^{\intercal}\bX\bB^{-1}(\theta_1)\bX^{\intercal}\bU_1
 \\ &={\left[\begin{array}{cccc}{\hat{S}_{n}} & {O_{p}\left(\frac{1}{\sqrt{n}}\right)} & {\cdots} & {O_{p}\left(\frac{1}{\sqrt{n}}\right)} \\ {O_{p}\left(\frac{1}{\sqrt{n}}\right)} & {O_{p}(1)} & {\cdots} & {O_{p}\left(\frac{1}{\sqrt{n}}\right)} \\ {\cdots} & {\cdots} & {O_{p}(1)} & {O_{p}\left(\frac{1}{\sqrt{n}}\right)} \\ {O_{p}\left(\frac{1}{\sqrt{n}}\right)} & {\cdots} & {O_{p}\left(\frac{1}{\sqrt{n}}\right)} & {O_{p}(1)}\end{array}\right]}
\end{aligned} \end{equation}
where  $\hat{S}_n  =\frac{1}{\alpha_1}-\bu_1^{\intercal}\bX\bB^{-1}(\theta_1)\bX^{\intercal}\bu_1$.
Moreover, we have \begin{equation}\label{shh23}\begin{aligned}
&\chi_1\theta_1\bU_1^{\intercal}\bX\bB^{-1}(\lambda_1)\bB^{-1}(\theta_1)\bX^{\intercal}\bU_1\\
&= \left[\begin{array}{cccc} {\chi_1\theta_1 R_1} &  {O_p(\chi_1)} & {\cdots} & {O_p(\chi_1)}  \\  {O_p(\chi_1)} & {O_p(\chi_1)} & {\cdots}  &  {O_p(\chi_1)} \\  {\cdots}&{ \cdots} &{O_p(\chi_1)} & {O_p(\chi_1)}  \\  {O_p(\chi_1)} & {\cdots}& {O_p(\chi_1)}& {O_p(\chi_1)}  \end{array}\right] ,
\end{aligned}\end{equation}
where $R_1 = \bu_1^{\intercal}\bX\bB^{-1}(\lambda_1)\bB^{-1}(\theta_1)\bX^{\intercal}\bu_1.$  By   \eqref{kmv46}, \eqref{4bndg}, \eqref{46bdf}, \eqref{shh23}, the fact that $\chi_1 \stackrel{a.s.} \rightarrow 0,$ and the Leibniz formula for a determinant, one can show that  $$\hat{S}_n +\chi_1\theta_1 R_1 = o_p(\frac{1}{\sqrt{n}}),$$ which implies \begin{equation}\label{26hae} \sqrt{n}\chi_1 = \frac{-\sqrt{n}\hat{S}_n}{\theta_1 R_1}.
 \end{equation}


\indent By Theorem \ref{Th21}, $ -\sqrt{n}\hat{S}_n /\theta_1 $ converges in distribution to a normal distribution with mean $0$ and variance $\{\alpha_1^2 \psi^2(\alpha_1)\}^{-1}[(\gamma_4-3) \sum_{i=1}^p u_{1i}^4 + \frac{2}{\psi^\prime(\alpha)}]$.
To handle $R_1$ in (\ref{26hae}) by (\ref{maeq1}) we further expand it as
  \begin{equation}\label{624dg}R_1= \bu_1^{\intercal}\bX\bB^{-2}(\theta_1)\bX^{\intercal}\bu_1+ (\theta_1-\lambda_1)\bu_1^{\intercal}\bX\bB^{-1}(\lambda_1)\bB^{-2}(\theta_1)\bX^{\intercal}\bu_1.
\end{equation}
It is easy to obtain
$$(\theta_1-\lambda_1)\bu_1^{\intercal}\bX\bB^{-1}(\lambda_1)\bB^{-2}(\theta_1)\bX^{\intercal}\bu_1 \stackrel{i.p.} \rightarrow 0,$$
Next, we handle the first term of $R_1$, that equals $\theta_1^{-2}\bu_1^\intercal \bX \bA^{-2}(\theta_1) \bX^\intercal \bu_1$. Here $\bA(\theta_1)$ is obtained from $\bA$ in (\ref{noset}) with $\psi_n(\alpha)$ replaced by $\theta_1=\psi_n(\alpha_1)$.
Write \begin{equation}\begin{aligned}
&\sqrt{n} (\bu_1^\intercal \bX \bA^{-2}(\theta_1)  \bX^\intercal \bu_1 - E \bu_1^\intercal \bX \bA^{-2}(\theta_1)  \bX^\intercal \bu_1 ) \\
&=\sqrt{n}\sum_{k=1}^n(E_k-E_{k-1}) (I_1 +I_2 +I_3 + I_4)
\end{aligned}\end{equation}
where \begin{equation}\begin{aligned}
&I_1 = -\alpha_k \bx_k^\intercal \bD_k^{-1}\bu_1\bu_1^\intercal  ({\bD_k^\intercal})^{-1}\bx_k,\quad & I_2=\alpha_k \bx_k^\intercal \bD_k^{-1}\bu_1\bu_1^\intercal  ({\bD_k^\intercal})^{-2}\bx_k,\\
&I_3=\alpha_k \bx_k^\intercal \bD_k^{-2}\bu_1\bu_1^\intercal  ({\bD_k^\intercal})^{-1}\bx_k, \quad  &I_4=-\alpha_k \bx_k^\intercal \bD_k^{-2}\tilde{\bSigma}_1 \bx_k\times I_1.
\end{aligned}\end{equation}
Here we should have put an indicator of the event as in the proof of Theorem \ref{Th21} to ensure the existence of the expectation involving the inverse of the matrices of interest but we ignore it for simplicity.
By the above decomposition as in the proof of Theorem \ref{Th21} one can verify that
\begin{equation}\label{82hgh}E|\sqrt{n} (\bu_1^\intercal \bX \bA^{-2}(\theta_1)  \bX^\intercal \bu_1 - E \bu_1^\intercal \bX \bA^{-2}(\theta_1) \bX^\intercal \bu_1 )|^2= O(1).\end{equation}
Moreover by the interpolation method using in the calculating the asymptotic mean in the proof of Theorem \ref{Th21} we conclude that  \begin{eqnarray}\label{4w3km}
E\bu_1^\intercal \bX \bA^{-2}(\theta_1)  \bX^\intercal \bu_1 - E \bu_1^\intercal \bX^0 \{\bA^0(\theta_1)\}^{-2} \bX^{0\intercal} \bu_1= o(1).
\end{eqnarray}
Note that \begin{eqnarray}\label{556ja}
\begin{aligned}E \bu_1^\intercal \bX^0 \{\bA^0(\theta_1)\}^{-2} \bX^{0\intercal} \bu_1&=\frac{1}{n}Etr \{\bA^0(\theta_1)\}^{-2}= E\int \frac{\theta_1^2}{(t-\theta_1)^2}dF^{\bX^{0\intercal} \bSigma_{1P} \bX^0}(t)\\&
\rightarrow \int \frac{\psi(\alpha_1)^2}{\{t-\psi(\alpha_1)\}^2}dF^{c,H}(t)= \psi(\alpha_1)^2\underline{m}^\prime\{\psi(\alpha_1)\}.
\end{aligned}
\end{eqnarray}
By the fact that $\psi(\alpha)$ is the the inverse function of $\alpha: x \mapsto -1/\underline{m}(x)$ we have
\begin{eqnarray}\label{gnm67}
\underline{m}^\prime\{\psi(\alpha_1)\}= \frac{1}{\alpha_1^2\psi^\prime(\alpha_1)}.
\end{eqnarray}
It follows from (\ref{624dg}),(\ref{82hgh}),(\ref{4w3km}) (\ref{556ja}) and (\ref{gnm67}) that
\begin{equation}\label{393ki} R_1 \stackrel{i.p.} \rightarrow \frac{1}{\alpha_1^2\psi^\prime(\alpha_1)}.
\end{equation}
One application of Slutsky's theorem implies that $\sqrt{n}\chi_1$ converges in distribution to a normal distribution with mean 0 and variance $\sigma_1^2$.\\

\indent To finish the proof of Theorem \ref{Th22}, we prove \eqref{22thm} next. Without loss of generality,  we just deal with the joint distribution of $\chi_1$ and $\chi_2$.\\

For $i=1,2$, let \begin{equation}\begin{aligned}
&\tilde{\bSigma}_1(\theta_i) = \frac{\bSigma_{1P}}{\theta_i},\quad  \bD_k^{(i)}=\bI-\tilde{\bSigma}_1(\theta_i)\bX_k\bX_k^\intercal, \quad \bB_k^{(i)}= ( \bD_k^{(i)})^{-1}\bu_i\bu_i^{\intercal}(\bD_k^{(i)\intercal})^{-1},\\
&a_n(\theta_i)=\frac{1}{1-\frac{1}{n}Etr\tilde{\bSigma}_1(\theta_i)(\bD_k^{(i)\intercal})^{-1}} ,\quad \delta_k (\theta_i) =\bx_k^{\intercal}\bB_k^{(i)} \bx_k - \frac{1}{n}tr\bB_k^{(i)}.
\end{aligned}\end{equation}
For any constants $c_1$ and $c_2$, as in (\ref{gnvb9}), write
 \begin{equation}\begin{split}
&c_1\sqrt{n}\Big\{\bu_1^{\intercal}\bX\bB^{-1}(\theta_1)\bX^{\intercal}\bu_1-E\bu_1^{\intercal}\bX\bB^{-1}(\theta_1)\bX^{\intercal}\bu_1\Big\}+\\&\relphantom{EEEEEE} c_2\sqrt{n}\Big\{\bu_2^{\intercal}\bX\bB^{-1}(\theta_2)\bX^{\intercal}\bu_2-E\bu_2^{\intercal}\bX\bB^{-1}(\theta_2)\bX^{\intercal}\bu_2\Big\}\\
&= \sqrt{n}\sum_{k=1}^n E_k\big\{c_1 \theta_1^{-1}a_n(\theta_1)\delta_k(\theta_1) + c_2\theta_2^{-1} a_n(\theta_2)\delta_k(\theta_2)\big\}+o_p(1).
\end{split}\end{equation}

Condition (ii) of Lemma \ref{Lem4.2} can be verified as in (\ref{e7gab}). For condition (i), we write
\begin{equation}\label{m9xc7}\begin{aligned}
&n\sum_{k=1}^n E_{k-1}[E_k\big\{c_1 \theta_1^{-1}a_n(\theta_1)\delta_k(\theta_1) + c_2 \theta_2^{-1}a_n(\theta_2)\delta_k(\theta_2)\big\}]^2\\
&=n\sum_{i=1}^2 c_i^2 \theta_i^{-2}a_n^2(\theta_i) \sum_{k=1}^n E_{k-1}[E_k\{\delta_k(\theta_i)\}]^2\\
&+n \frac{c_1c_2 a_n(\theta_1)a_n(\theta_2)}{\theta_1\theta_2}\sum_{k=1}^n E_{k-1}[E_k\{\delta_k(\theta_1)\}E_k\{\delta_k(\theta_2)\}].
\end{aligned}\end{equation}
The first term corresponds to marginal variances which have been derived before.  The second term above representing the covariance of $\lambda_1$ and $\lambda_2$ is slightly different from (\ref{347bn}).  By (\ref{543ah}) we need to find the limit of
\begin{equation}\label{81ago}\begin{aligned}
\frac{1}{n}(\gamma_4-3)\sum_{k=1}^n\sum_{i=1}^{p}E_k(\bB_k^{(1)})_{ii} E_k(\bB_k^{(2)})_{ii}+\frac{2}{n}\sum_{k=1}^ntr( E_k\bB_k^{(1)}E_k\bB_k^{(2)}).
\end{aligned}\end{equation}
Following the arguments similar to those from (\ref{347bn})-(\ref{348cn}), we conclude that the first term of (\ref{81ago}) equals $(\gamma_4-3)\sum_{k=1}^p u_{1k}^2u_{2k}^2+o_p(1)$.
We claim that the second term of (\ref{81ago}) is negligible. Indeed the second term can be rewritten as
 \begin{equation} \label{040720} \begin{aligned}
&\frac{2}{n}\sum_{k=1}^n E_{k-1}tr E_k ((\bD_k^{(1)})^{-1}\bu_1\bu_1^\intercal(\bD_k^{(1)\intercal})^{-1} ) E_k ((\bD_k^{(2)})^{-1}\bu_2\bu_2^\intercal(\bD_k^{(2)\intercal})^{-1} )\\
&=\frac{2}{n}\sum_{k=1}^n E_{k-1}\Big(\bu_1^\intercal (\bD_k^{(1)\intercal})^{-1}(\breve{\bD}_k^{(2)})^{-1} \bu_2\bu_2^\intercal (\breve{\bD}_k^{(2)\intercal})^{-1} (\bD_k^{(1)})^{-1}\bu_1\Big),
\end{aligned}\end{equation}
where  $(\breve{\bD}_k^{(i)})^{-1}$ is defined similarly as $(\bD_k^{(i)})^{-1}$ by $(\bx_1,\cdots, \bx_{k-1},\breve{\bx}_{k+1},\cdots,\breve{\bx}_{n})$ and $\breve{\bx}_{k+1}, \cdots, \breve{\bx}_{n}$ are i.i.d copies of $\bx_{k+1},\cdots,\bx_n$.
 Define $\bT(\theta_i), i=1,2$ to be analogues of $\bT$ in \eqref{defT1} with $\psi_n$ replaced by $\theta_i, i=1,2$.  The argument for (\ref{040720}) is parallel to \eqref{47agn}-\eqref{378qw}, so we do not provide all details and list only the differences below. The term  $\bu_1^\intercal \bT(\theta_1)^{-1}\bT(\theta_2)^{-1}\bu_2$ replaces $\bw_1^\intercal \bT^{-2} \bw_1$ in  \eqref{378qw}.
Since $\bu_1^\intercal \bu_2 = 0$, and $\bu_1^\intercal \bU_2=\bu_2^\intercal \bU_2=0$, we can get $\bu_1^\intercal \bT(\theta_1)^{-1}\bT(\theta_2)^{-1}\bu_2 =0$ by using \eqref{maeq1}.  The term replacing $n^{-1}a_{1n}^2 tr\tilde{\bSigma}_1 \bT^{-1} \tilde{\bSigma}_1 \bT^{-1}$ in \eqref{47agn} and \eqref{m40b0} is $n^{-1}a_{1n}(\theta_1)a_{1n}(\theta_2)tr\tilde{\bSigma}_1(\theta_1) (\bT(\theta_1))^{-1} \tilde{\bSigma}_1(\theta_2) (\bT(\theta_2))^{-1}$, that equals $1-(\psi(\alpha_2)-\psi(\alpha_1))/(\alpha_2-\alpha_1)$ asymptotically.  These imply that the second term is $o_p(1)$ via \eqref{47agn}.
Together with the fact that for $i = 1$ or $2$,  $a_n(\theta_i)\rightarrow \theta_i/ \alpha_i$, the limit of the second term of (\ref{m9xc7}) in probability is \begin{equation}c_1c_2\frac{\gamma_4-3}{\alpha_1\alpha_2}\sum_{k=1}^p u_{1k}^2 u_{2k}^2. \end{equation}
Therefore, the asymptotic covariance of $\sqrt{n}\chi_1$ and $\sqrt{n}\chi_2$ is given by $\sigma_{12}$ defined in assumption 5.  \\

\noindent \textbf{Proof of Theorem \ref{Th23}}.  We use the first $m_1$ spiked  eigenvalues to illustrate the idea. Denote  $(\lambda_{j}-\theta_1)/{\theta_1}$ by $\chi_{1j}$ for $j=1,\cdots, m_1$.
Partition the matrix $\bU_1 = (\bU_{11},\cdots, \bU_{1\mathcal{L}})$. Similar to \eqref{46bdf} and \eqref{shh23}, we write

\begin{equation}\label{72hoa} \begin{aligned} & \bLambda_S^{-1}- \bU_1^{\intercal}\bX\bB^{-1}(\lambda_j)\bX^{\intercal}\bU_1   \\ &= \left[\begin{array}{cccc}  {\tilde{\bS}_n}  &   {O_p(\frac{1}{\sqrt{n}})+O_p(\chi_{1j})}  &  \cdots &  {O_p(\frac{1}{\sqrt{n}})+O_p(\chi_{1j})} \\  {O_p(\frac{1}{\sqrt{n}})+O_p(\chi_{1j})} & {O_p(1)}  &  \cdots &  {\cdots}   \\    { \cdots} & {\cdots} &{O_p(1)}&  {O_p(\frac{1}{\sqrt{n}})+O_p(\chi_{1j})}\\ {O_p(\frac{1}{\sqrt{n}})+O_p(\chi_{1j})} & {\cdots} &   {O_p(\frac{1}{\sqrt{n}})+O_p(\chi_{1j})} &  {O_p(\frac{1}{\sqrt{n}})+O_p(\chi_{1j})} \end{array}\right]
\end{aligned}\end{equation}
where \begin{equation}\begin{aligned}
\tilde{S}_n&=\frac{\bI}{\alpha_1}-\bU_{11}^\intercal \bX \bB^{-1}(\lambda_j)\bX^\intercal \bU_{11}\\
&=\frac{\bI}{\alpha_1}-\bU_{11}^\intercal \bX \bB^{-1}(\theta_1)\bX^\intercal \bU_{11}+\chi_{1j}\theta_1 \bU_{11}^\intercal \bX \bB^{-2}(\theta_1)\bX^\intercal \bU_{11}+\\&
\relphantom{EEEEEE}
\chi_{1j}(\theta_1-\lambda_j)\bU_{11}^\intercal \bX\bB^{-1}(\lambda_1) \bB^{-1}(\theta_1)\bX^\intercal \bU_{11}\\
&=\frac{\bI}{\alpha_1}-\bU_{11}^\intercal \bX \bB^{-1}(\theta_1)\bX^\intercal \bU_{11}+ \chi_{1j}\big(\frac{\psi(\alpha_1)}{\alpha_1^2\psi^\prime(\alpha_1)}\bI + o_p(1)\big).
\end{aligned}\end{equation}
 Using arguments similar to the proof of Theorem \ref{Th22} we have \begin{equation}\sqrt{n}\frac{\alpha_1^2\psi^\prime(\alpha_1)}{\psi(\alpha_1)}\Big(\frac{\bI}{\alpha_1}-\bU_{11}^\intercal \bX \bB^{-1}(\theta_1)\bX^\intercal \bU_{11}\Big) \stackrel{D}\rightarrow \bG_1,\end{equation} where $\bG_1$ is a Gaussian random matrix with mean 0 and covariance structure characterized by \eqref{cocar}.
Therefore, if we multiply $n^{1/4}$ on the first $m_1$ rows and multiply $n^{1/4}$ on the first $m_1$ columns of \eqref{72hoa}, by using the Skorokhod strong representation similar to arguments in pages 464-465 of \cite{bai2008central}  we see that $\sqrt{n}\chi_{11},\cdots,\sqrt{n}\chi_{1m_1}$ converge weakly to the joint distribution of eigenvalues of $\bG_1$. Thus we conclude Theorem \ref{Th23}.

\section{Proof of Theorem \ref{Th2.4}}

 Set $\theta_n=\psi_n(\alpha)$, where $\alpha$ is one of $\alpha_1,\cdots,\alpha_K$. Let $\varphi_n = \sqrt{n}(\bw_1^\intercal \bX\bA^{-1}(\theta_n) \bX^\intercal \bw_1 - \theta_n/\alpha)$, and $M_n(z) = p(m_{F^{\bS_n}}(z)-m_{n}(z))$ where $m_{n}(z)$ is the Stieltjes transformation of $F_{c_n,H_n}(x)$. Define the contour $\mathcal{C}_n$  as in page 561 of \cite{bai2004}. We assume that the $M_n(z)$ has been truncated on the contour as in \cite{bai2004}. 
 To prove Theorem \ref{Th2.4} it suffices to establish the following lemma.
\begin{Lemma}  Under the assumptions of Theorem \ref{Th2.4}, 
we have for $z \in \mathcal{C}_n$,  $$ (\varphi_n, M_n(z)) \rightarrow (\varphi, M(z)),$$ where $\varphi$ is normal random variable with mean $0$ and variance $(\gamma_4-3) \sum_{i=1}^p  w_{1i}^4 + 2/\psi^\prime(\alpha) $,   independent of Gaussian process $M(z)$ with mean
\begin{equation}\begin{aligned}
EM(z) = &\frac{c\int \underline{m}^3(z)t^2dH(t)/(1+t\underline{m}(z))^3}{1-c\int \underline{m}^2(z)t^2dH(t)/(1+t\underline{m}(z))^2}dz \\ & + (\gamma_4 - 3)\int \frac{c\underline{m}^3(z)h_2(z)}{1-c\int \underline{m}^2(z)t^2dH(t)/(1+t\underline{m}(z))^2}dz,
\end{aligned}\end{equation}
 and covariance\begin{equation}\begin{aligned} Cov(M(z_1),M(z_2)) =& \frac{2\underline{m}^\prime(z_1)\underline{m}^\prime(z_2)}{(\underline{m}(z_1)-\underline{m}(z_2)^2}-\frac{1}{(z_1-z_2)^2} \\
 &+ c(\gamma_4-3)\frac{d^2}{dz_1 dz_2}[\underline{m}(z_1)\underline{m}(z_2)h_1(z_1,z_2)].
 \end{aligned}\end{equation}
\end{Lemma}

\begin{proof}
Define $\bQ(z) = \bS_n-z\bI, \quad \bQ_k(z) = \bQ(z)-\bGamma \bx_k \bx_k^\intercal \bGamma^\intercal$, $\bC_k = \bGamma^\intercal \bQ_k^{-1}(z) \bGamma$,
$$
\epsilon_k(z) = \bx_k^\intercal \bC_k \bx_k-\frac{1}{n}tr \bC_k,\  b_n(z)=\frac{1}{1+ \frac{1}{n}Etr \bC_k},
 \beta_k(z)=\frac{1}{1+ \bx_k^\intercal  \bC_k \bx_k},\  \bar{\beta}_k=\frac{1}{1+ \frac{1}{n}tr \bC_k}.
$$
For any constants $c_1$ and $c_2$, we consider the distribution of  $c_1 \varphi_n + c_2 M_n(z)$. Its tightness is from that of  ${M}_n(z)$ which was proved in \cite{bai2004}. We truncate the entries of the matrix $\bX$ in the linear spectral statistic at $\eta_n n^{1/2}$ as in \cite{bai2004} while we truncate the entries of the matrix $\bX$ at $\eta_n n^{1/4}$ for the largest spike eigenvalues which is in Appendix. Hence it is equivalent to considering 
\begin{equation}
c_1 \sqrt{n}(\bw_1^{\intercal}\dot{\bX} \dot{\bA}^{-1}\dot{\bX}^{\intercal}\bw_1- \frac{\theta_n}{\alpha} )+c_2p(m_{F^{\ddot{\bS}_n}}(z)- m_n(z))
\end{equation}
where we use `.' to mean that the entries of $\bX$ is truncated at $\eta_n n^{1/4}$ and `..' to mean that the entries of $\bX$ is truncated at $\eta_n n^{1/2}$.
We below use $\ddot{\beta}_k(z)$ to represent $\beta_k(z)$ but the underlying random variables are truncated at $\eta_n n^{1/2}$, and $\ddot{\epsilon}_k(z),\dot{\bar{\alpha}}_k, \dot{\delta}_k$, etc, are similarly defined. Section 2 of \cite{bai2004}
establishes that
 $$p(m_{F^{\ddot{\bS}_n}}(z)-Em_{F^{\ddot {\bS}_n}}(z))
  = -\sum_{k=1}^n E_k \frac{d}{dz}\ddot{\beta}_k(z) \ddot{\epsilon}_k(z) + o_p(1).$$
 Thus, referring to (\ref{gnvb9}), we have
\begin{equation}\label{040702}\begin{aligned} &c_1\sqrt{n}(\bw_1^{\intercal}\dot{\bX} \dot{A}^{-1}\dot{\bX}^{\intercal}\bw_1- E \bw_1^{\intercal}\dot{\bX} \dot{A}^{-1}\dot{\bX}^{\intercal}\bw_1 )+c_2p(m_{F^{\ddot{\bS}_n}}(z)-Em_{F^{\ddot{\bS}_n}}(z))\\
&=\sum_{k=1}^{n} E_k (c_1 \sqrt{n}\dot{\bar{\alpha}}_k \dot{\delta}_k - c_2 \frac{d}{dz}\ddot{\bar{\beta}}_k(z) \ddot{\epsilon}_k(z))+o_p(1).
\end{aligned}\end{equation}
Note that both Theorem \ref{Th21} and the main result in \cite{bai2004} use martingale to prove central limit theorem as in (\ref{040702}). Hence by Lemma 4.1
it suffices to consider the limit of the cross product
\begin{equation}\label{723zc}  \sum_{k=1}^n E_{k-1}(E_k \sqrt{n}\dot{\bar{\alpha}}_k \dot{\delta}_k) (E_k \ddot{\bar{\beta}}_k(z) \ddot{\epsilon}_k(z)) 
.\end{equation}
Using (2.1) in \cite{bai2004}, and (\ref{lar11}) we conclude that
\begin{equation}\label{346mv} (\ref{723zc}) = \sqrt{n}\dot{a}_n \ddot{b}_n(z)\sum_{k=1}^n   E_{k-1}[(E_k \dot{\delta}_k) (E_k \ddot{\epsilon}_k(z))]+o_p(1).
\end{equation}
From (\ref{543ah}), we have
\begin{eqnarray}\label{fgn21}|E_{k-1}[(E_k \dot{\delta}_k) (E_k \ddot{\epsilon}_k(z))]|\leq\frac{C}{n^2} \Big[|\sum_{i=1}^p (E_k\dot{\bB}_k)_{ii}(E_k\ddot{\bC}_k)_{ii}|+ |tr(E_k\dot{\bB}_kE_k\ddot{\bC}_k)|\Big].\end{eqnarray}
By the fact that $(\dot{\bB}_{k})_{ii}\geq 0$ , and $|(E_k\ddot{\bC}_k)_{ii}|\leq \rVert E_k \ddot{\bC}_k \rVert $, we get \begin{eqnarray}\label{fgn22}|\sum_{i=1}^p (E_k\dot{\bB}_k)_{ii}(E_k\ddot{\bC}_k)_{ii}|\leq  \rVert E_k \ddot{\bC}_k \rVert |trE_k \dot{\bB}_k|. \end{eqnarray}
A simple application of Von Neumann's trace inequality yields that \begin{eqnarray} \label{fgn23}|tr(E_k\dot{\bB}_kE_k\ddot{\bC}_k)|\leq \rVert E_k \ddot{\bC}_k \rVert |trE_k \dot{\bB}_k|.\end{eqnarray}
It follows from (\ref{fgn21}), (\ref{fgn22}) and (\ref{fgn23}) that  \begin{equation}\label{833oa}\begin{aligned}
&E|\sqrt{n}\dot{a}_n \ddot{b}_n(z)\sum_{k=1}^n   E_{k-1}(E_k \dot{\delta}_k) (E_k \ddot{\epsilon}_k(z)))|^2\\
&\leq Cn^{-3} |\dot{a}_n \ddot{b}_n(z)|^2E\Big{|}\sum_{k=1}^n \rVert E_k \ddot{\bC}_k \rVert |trE_k \dot{\bB}_k|\Big{|}^2\\
&\leq Cn^{-2}|\dot{a}_n \ddot{b}_n(z)|^2\sum_{k=1}^n E|trE_k \dot{\bB}_k|^2=O(n^{-1}).
\end{aligned}\end{equation}
\noindent Then (\ref{346mv}) and (\ref{833oa}) imply $(\ref{723zc}) = o_p(1)$. Lemma 2.3 in \cite{bai2004} further ensures
 $$\sum_{k=1}^n E_{k-1}(E_k \sqrt{n}\dot{\bar{\alpha}}_k \dot{\delta}_k) (E_k \frac{d}{dz}\ddot{\bar{\beta}}_k(z) \ddot{\epsilon}_k(z)))=o_p(1).$$
 It follows that$ \sqrt{n}(\bw_1^{\intercal}\dot{\bX} \dot{\bA}^{-1}\dot{\bX}^{\intercal}\bw_1- E \bw_1^{\intercal}\dot{\bX} \dot{\bA}^{-1}\dot{\bX}^{\intercal}\bw_1)$ and  $ p(m_{F^{\ddot{\bS}_n}(z)}-Em_{F^{\ddot{\bS}_n}(z)})$ are asymptotically independent from Lemma \ref{Lem4.2} and Cram\'er-Wold's device. Consequently, $\varphi_n$ and $M_n(z)$ are asymptotically independent. The marginal distribution is from our Theorem \ref{Th22} and Theorem 1.4 of \cite{pan2008central}.
\end{proof}

\noindent \textbf{Proof of  Theorem \ref{Th2.4}}: Without loss of generality, we consider the first spiked eigenvalue $\lambda_1$. Recall that $\chi_1 = (\lambda_1-\theta_1)/\theta_1$, and  $R_1$ in (\ref{624dg}). Asymptotic independence of $\sqrt{n}\chi_1$ and $L_p(\varphi)$ follows easily from the above lemma, (\ref{26hae}) and (\ref{393ki}).

\section{Proof of Theorem \ref{Th2.5}}
\begin{proof}
 We take the estimation of $\sum_{k=1}^p u_{1k}^4$ as an illustration.  Let $\bs_1, \bs_2$ be two deterministic $p\times 1$ vectors, and $\bu_i, i = 1,\cdots, p$ be the population eigenvectors of $\bSigma$. 
 We need to clarify that although results in \cite{mestre2008} are given under absolutely continuous random
entries with a bounded 8th moment, this assumption can be relaxed to random entries with a bounded 4th moment.  Indeed, note that the conclusion (7) therein is a well known result, see \cite{bai2010}, and the conclusion (9) therein is also true by appropriately modifying the proof of \cite{bai2007asymptotics} (just change one unit vector there to $\bs_2$) assuming that the bounded 4th moment.

By replacing $c$ and $H$ with $c_n$ and $H_n$ in both \eqref{88am9} and \eqref{23ato}, we get two $n$-dependent solutions $m_n(z)$ and $\underline{m}_n(z)$ respectively. Some notations are introduced first:
\begin{equation}\label{jo38b}\begin{aligned}
&\eta_1 = \bs_1^ \intercal \bu_1 \bu_1^\intercal \bs_2,\ f_n(z) = \frac{z}{1-c-czm_n(z)}, \quad  s_n = \bs_1^ \intercal(-z\underline{m}_n(z)\bSigma-zI)^{-1} \bs_2 ,  \\
& g_n(z)=\sum_{i=1}^p \frac{\bs_1^ \intercal \bu_i \bu_i^\intercal \bs_2}{ \alpha_i-f_n(z)}f_n^\prime(z) = s_n \frac{1-c+cz^2m_n^\prime(z)}{1-c-czm_n(z)},\\
&\hat{s}_n (z)= \bs_1^\intercal (\bS_n-zI)^{-1}\bs_2,\ \hat{m}_n (z)= m_{F^{\bS_n}}(z),\ \hat{g}_n(z) = \hat{s}_n(z)\frac{1-c+cz\hat{m}_n'(z)}{1-c-cz\hat{m}_n(z)}.
\end{aligned}\end{equation}
Let $\Upsilon$ be a negatively oriented contour described by the boundary of a rectangle $$\{z \in \mathbb{C}, E_1 \leq Re(z) \leq E_2, |Im(z)| \leq b\},$$ where $b$ is positive, and $[E_1,E_2]$ encloses $\psi(\alpha_1)$ and no other components of the support of $F^{c_n,H_n}$.
 Note that since $\alpha_1$ is the largest distant spike, the associated support of $F^{c_n,H_n}$ for this spike is separated from other components,  and the length of the support decreases to 0 when n goes to infinity. So it is legitimate to choose $\Upsilon$ as above according to our assumptions.

From  Proposition 3 in \cite{mestre2008}
  $$\eta_1 =\frac{1}{2\pi i}\oint_\Upsilon  g_n(z) dz. $$
Furthermore Proposition 4 in \cite{mestre2008} yields
\begin{equation}\label{050701}
\eta_1-\hat{\eta}_1\stackrel {a.s.}\longrightarrow 0
\end{equation}
where $\hat{\eta}_1:=\frac{1}{2\pi i}\oint_\Upsilon  \hat{g}_n(z) dz$.
We below prove a stronger result in terms of the moment convergence (see \ref{jh93g} below) for our purpose.

Let $\Xi $ be the event such that the contour $\Upsilon$  contains the largest sample eigenvalue but excludes all the remaining sample eigenvalues. 
The event $\Xi$ holds with high probability as pointed out in the proof of Theorem \ref{Th22}.
The next aim is to show that \begin{equation}\label{03jpn}E\Big |(\hat{\eta}_1-\eta_1)I(\Xi) \Big|^4 \leq Cn^{-2},
\end{equation}
where the constant $C$ is independent of the choice of the vectors $\bs_1$ and $\bs_2$ as long as they are bounded in terms of the Euclidean norm.
We first show that for any $z \in \Upsilon$
\begin{eqnarray}\label{jh93g}E ( |\hat{g}_n(z)-g_n(z)|^4  I(\Xi) ) \leq C n^{-2}.
\end{eqnarray}
When there is no confusion we below omit the variable z from notations in (\ref{jo38b}).
 Set
 $$g_z(u,v,w) = u \frac{1-c+cz^2w}{1-c-czv},
 $$
  which is an analytic function of u, v and $w$ in a region excluding $v\neq (1-c)/cz.$ Due to the structure of the contour $\Upsilon$
  by the exact separation arguments in \cite{bai1998no},  $\hat{s}_n, \hat{m}_n, \hat{m}_n^\prime $ are uniformly bounded for $z$ on the contour $\Upsilon$ given the event $\Xi $.
  Note that $|\underline{m}(z)|$ is bounded below away from zero on a bounded set by (5.1) in \cite{bai2004} (one may verify that $\underline{m}_n(z)$ has a similar property) and the end point $E_1$ involved in the contour $\Upsilon$ is positive. Hence from the relationship between $\underline{m_n}(z)$ and $m_n(z)$ we conclude that the absolute value of the  denominator of $g(\hat{s}_n, \hat{m}_n, \hat{m}_n^\prime)$ is bounded below away from zero on the contour.
 Hence a straightforward calculation indicates that
\begin{equation}\begin{aligned}
& |g_z(\hat{s}_n, \hat{m}_n, \hat{m}_n^\prime)-g_z(s_n,m_n,m_n^\prime)| I(\Xi)\\
&\leq C(|\hat{s}_n -s_n| + |\hat{m}_n-m_n| + |\hat{m}_n^\prime-m_n^\prime|)I(\Xi),
\end{aligned}\end{equation}
where the constant $C$ in the above inequality is independent of $z$. 
We next consider the 4th moment of the above three terms term by term.

   By modifying the proof of Theorem 2 in \cite{bai2004} we conclude that
   $E|\hat{s}_n-s_n |^4I(\Xi)\leq C_1n^{-2}$. 
   The proof of Lemma 1.1 in \cite{bai2004} shows that $E|\hat{m}_n-m_n|^4I(\Xi) \leq C_2 n^{-2}$. 
  Note that although \cite{bai2007asymptotics} and \cite{bai2004} established the above conclusions on the two horizontal lines of the contour $\Upsilon$ they still hold on the two vertical lines of the contour because $\min\limits_{i}|\lambda_i-z|\geq C>0$ on the two vertical lines given that the respective distances from $\varphi(\alpha_1)$ to the end points $E_1$ and $E_2$ involved in the contour can be positive. To handle the last term by using a martingale method it is not difficult to derive
   \begin{equation} \label{734oj} E|\hat{m}_n^\prime -E\hat{m}_n^\prime|^4 \leq C n^{-2}.
   \end{equation}
    Moreover note that $\sqrt{n}(E\hat{m}_n -m_n)$ converges uniformly to zero on the contour by \cite{bai2004}. Applying Vitali's convergence theorem, on the contour,
\begin{equation}\label{34oan} \sqrt{n}(E\hat{m}_n^\prime-m_n^\prime) \rightarrow 0  \;\; \mbox{uniformly for z}.\end{equation}
(\ref{734oj}), together with (\ref{34oan}), implies $E|\hat{m}_n^\prime - m_n^\prime|^4 \leq C n^{-2}$. Consequently, we have (\ref{jh93g}).
It follows from Jensen's inequality that
\begin{equation}
E| \oint_\Upsilon (\hat{g}_n-g_n) I(\Xi) dz|^4  \leq C\oint_\Upsilon E\big{(}|\hat{g}_n-g_n|^4 I(\Xi)\big{)} |dz|\leq Cn^{-2}.
\end{equation}
Therefore (\ref{03jpn}) is true. \\

Now we are at a position to show that our estimator for $\sum_{k=1}^p u_{1k}^4$ is weakly consistent.  
If we take $\bs_1=\bs_2=\be_1$ in the definition $\eta_1$ of (\ref{03jpn}) then one could find the estimator $\hat{v}_{1j}$ of $u_{1j}^2$ from (\ref{050701}) for $j = 1,\cdots, p$. 
It turns out that 
\begin{equation}\label{070701}
\hat{v}_{1j} = \sum_{k=1}^p \theta_1(k) \hat{u}_{kj}^2,
\end{equation}
and recall that $\theta_1(k)$ is defined in (\ref{ceeg1}). One may refer to \cite{mestre2008} for details of calculating the explicit expression of $\hat{v}_{1j}$.

\indent Using (\ref{03jpn}), we have
\begin{equation} \begin{aligned} &P( \max_{j}| \hat{v}_{1j}-u_{1j}^2 | > \epsilon) \leq P( \max_{j}| \hat{v}_{1j}-u_{1j}^2 | > \epsilon,\ \Xi) + P(\Xi^c) \\ &\leq \sum_{i=1}^p P(| \hat{v}_{1i}-u_{1i}^2 | > \epsilon,\  \Xi)+ P(\Xi^c)=o(1).\end{aligned}\end{equation}
It follows that
 $$\sum_{j=1}^p |\hat{v}_{1j}^2-u_{1j}^4| \leq \max_{j}| \hat{v}_{1j}-u_{1j}^2| \sum_{j=1}^p ( |\hat{v}_{1j}| + u_{1j}^2) =o_p(1),
  $$
where we use the fact $\sum_{j=1}^p |\hat{v}_{1j}| =O_p(1)$ to be proved below. 

 We finally verify the fact $\sum_{j=1}^p |\hat{v}_{1j}| =O_p(1)$. We consider the case $p/n < 1$ first. By analyzing equation \eqref{an9k0} we conclude have the following interlacing relationship \begin{equation}\label{interlac}\lambda_1 > \nu_1 > \lambda_2 > \nu_2 > \cdots > \lambda_p >\nu_p >0.\end{equation}
 It follows from \eqref{ceeg1} that  $\theta_1(1) > 0$ and $\theta_1(k) < 0$ for $k\geq 2$.
Thus $|\hat{v}_{1j}| \leq \theta_1(1)\hat{u}_{1j}^2 - \sum_{k \geq 2}^p \theta_1(k)\hat{u}_{kj}^2$ by (\ref{070701}).
Note that $\sum_{k=1}^p \theta_1(k) = 1$ by the definition of $\theta_1(k)$. Furthermore, we claim that $\theta_{1}(1) $ is bounded with high probability.  Indeed, from \eqref{interlac} we have $\frac{\lambda_j}{\lambda_1-\lambda_j}<\frac{\nu_{j-1}}{\lambda_1-\nu_{j-1}}$ for $ j = 2, \cdots, p$. This further implies that $\theta_1(1) < 1 + \lambda_2/(\lambda_1-\lambda_2)$ by the definition of $\theta_1(k)$. Since with high probability, $\lambda_1$ lies in a small interval containing $\psi(\alpha_1)$ while $\lambda_2$ lies outside the interval, we get $\theta_1(1)$ is bounded with high probability, as claimed. Thus $\sum_{j=1}^p |\hat{v}_{1j}| \leq \theta_1(1)-\sum_{k\geq 2}^p \theta_1(k) = 2\theta_1(1)-1=O_p(1)$. This is also true for $p/n \geq 1$ through similar arguments. Therefore we conclude that  $\sum_{k=1}^p u_{1k}^4$ is consistently estimated by $\sum_{j=1}^p (\sum_{k=1}^p \theta_1(k)\hat{u}_{kj}^2)^2$.

\end{proof}

\section{Appendix}
This appendix provides justifications for truncation and centralization on the entries of $\bX$, and the proof of Lemma \ref{Evbhighprob}, Lemma \ref{lem31}.\\


\subsection{Truncation and Centralization}

Recall $x_{ij}=q_{ij}/\sqrt{n}$. 
By Assumption 1 we can select $\eta_n \rightarrow 0$ such that
 \begin{equation}\label{348c1}\eta_n^{-4}\int_{|q_{11}| > \eta_n n^{1/4}} |q_{11}|^4\rightarrow 0.
\end{equation}
 Let $\hat{x}_{ij}=\frac{1}{\sqrt{n}}q_{ij}I(|q_{ij}|\leq\eta_n n^{1/4})-E\frac{1}{\sqrt{n}}q_{ij}I(|q_{ij}|\leq\eta_n n^{1/4}), \tilde{x}_{ij}=x_{ij}-\hat{x}_{ij}, \hat{\bX}=(\hat{x}_{ij}),\tilde{\bX}=(\tilde{x}_{ij}).$ Let $\hat{\bX}_k=\hat{\bX}-\hat{\bx}_k \be_k^\intercal, \hat{\bX}_{jk}=\hat{\bX}-\hat{\bx}_k \be_k^\intercal-\hat{\bx}_j \be_j^\intercal.$ Recall $\tilde{\bSigma}_1, \bD, \bD_k, \bD_{jk}, \alpha_k, \alpha_{jk}$ defined in \eqref{noset} and  similarly define their respective analogues $\hat{\bD}_k, \hat{\bD}, \hat{\bD}_{jk}, \hat{\alpha}_k, \hat{\alpha}_{jk}$ corresponding to the truncated version.

 The first aim is to prove that
\begin{equation}\label{xf5la}\begin{aligned}
\sqrt{n}\Big(\bw_1^{\intercal}\bX(\bI_p-\bX^{\intercal}\tilde{\bSigma}_1\bX)^{-1}\bX^{\intercal}\bw_1-\bw_1^{\intercal}\hat{\bX}(\bI_p-\hat{\bX}^{\intercal}\tilde{\bSigma}_1 \hat{\bX})^{-1}\hat{\bX}^{\intercal}\bw_1\Big) \stackrel{i.p.}\rightarrow 0.\end{aligned}\end{equation}
With the help of \eqref{maeq1} and \eqref{maeq2}, we write
\begin{equation}\label{45non}\begin{aligned}
&\sqrt{n}(\bw_1^{\intercal}\bX(\bI_n-\bX^{\intercal}\tilde{\bSigma}_1\bX)^{-1}\bX^{\intercal}\bw_1-\bw_1^{\intercal}\hat{\bX}(\bI_n-\hat{\bX}^{\intercal}\tilde{\bSigma}_1\hat{\bX})^{-1}\hat{\bX}^{\intercal}\bw_1)\\
&=\sqrt{n}(\bw_1^{\intercal}\bX\bX^{\intercal}\bD^{-1}\bw_1-\bw_1^{\intercal}\hat{\bX}\hat{\bX}^{\intercal}\hat{\bD}^{-1}\bw_1)\\
&=\sqrt{n}\bw_1^\intercal(\bX\bX^{\intercal}-\hat{\bX}\hat{\bX}^{\intercal})\bD^{-1}\bw_1+\sqrt{n}\bw_1^\intercal\hat{\bX}\hat{\bX}^{\intercal}(\bD^{-1}-\hat{\bD}^{-1})\bw_1\\
&=\sqrt{n}\bw_1^\intercal (\tilde{\bX}\bX^\intercal +\hat{\bX}\tilde{\bX}^\intercal )\bD^{-1}\bw_1+\sqrt{n}\bw_1^\intercal\hat{\bX}\hat{\bX}^{\intercal}\hat{\bD}^{-1}\tilde{\bSigma}_1(\bX\bX^{\intercal}-\hat{\bX}\hat{\bX}^{\intercal})\bD^{-1}\bw_1\\
&=\sqrt{n}\sum_{i=1}^n \alpha_i\bw_1^\intercal \tilde{\bx}_i \bx_i^\intercal \bD_i^{-1}\bw_1+\sqrt{n}\sum_{i=1}^n \bw_1^\intercal \hat{\bx}_i \tilde{\bx}_i^\intercal \bD^{-1}\bw_1
\\&\relphantom{=}+\sqrt{n}\sum_{i=1}^n \alpha_i \bw_1^\intercal\hat{\bX}\hat{\bX}^{\intercal}\hat{\bD}^{-1}\tilde{\bSigma}_1\tilde{\bx}_i \bx_i^\intercal \bD_i^{-1}\bw_1 + \sqrt{n}\sum_{i=1}^n \bw_1^\intercal\hat{\bX}\hat{\bX}^{\intercal}\hat{\bD}^{-1}\tilde{\bSigma}_1\hat{\bx}_i \tilde{\bx}_i^\intercal \bD^{-1}\bw_1.
\end{aligned}\end{equation}
For the first and second term above they can be handled by the same way as the $\omega_1$ and $\omega_2$ in (7.3) of \cite{bai2007asymptotics}. For the third and fourth term, they are still similar but more complicated. We below only show the third term is $o_p(1)$ for illustration.

\indent By extracting the $i$-th column of $\bX$ respectively from $\hat{\bX}\hat{\bX}^{\intercal}$ and $\hat{\bD}^{-1}$ we split the third term above into summations of the four terms: $J_1, J_2, J_3, J_4$, where
\begin{equation}\begin{aligned}
&J_1 =\sqrt{n}\sum_{i=1}^n \alpha_i \bw_1^\intercal\hat{\bX}_i\hat{\bX}_i^{\intercal}\hat{\bD}_i^{-1}\tilde{\bSigma}_1\tilde{\bx}_i \bx_i^\intercal \bD_i^{-1}\bw_1,\\
&J_2=\sqrt{n}\sum_{i=1}^n \alpha_i \bw_1^\intercal\hat{\bx}_i\hat{\bx}_i^{\intercal}\hat{\bD}_i^{-1}\tilde{\bSigma}_1\tilde{\bx}_i \bx_i^\intercal \bD_i^{-1}\bw_1,\\
&J_3=\sqrt{n}\sum_{i=1}^n \alpha_i\hat{\alpha}_i \bw_1^\intercal\hat{\bX}_i\hat{\bX}_i^{\intercal}\hat{\bD}_i^{-1}\tilde{\bSigma}_1\hat{\bx}_i \hat{\bx}_i^\intercal \hat{\bD}_i^{-1}\tilde{\bSigma}_1\tilde{\bx}_i \bx_i^\intercal \bD_i^{-1}\bw_1,\\
&J_4=\sqrt{n}\sum_{i=1}^n \alpha_i \hat{\alpha}_i \bw_1^\intercal\hat{\bx}_i\hat{\bx}_i^{\intercal}\hat{\bD}_i^{-1}\tilde{\bSigma}_1\hat{\bx}_i \hat{\bx}_i^\intercal \hat{\bD}_i^{-1}\tilde{\bSigma}_1\tilde{\bx}_i \bx_i^\intercal \bD_i^{-1}\bw_1.
\end{aligned}\end{equation}
We only show $J_1=o_p(1)$ below and others can be handled similarly. We can define a event similar to \eqref{defevnt43} which holds with high probability such that $\alpha_i$ and the spectral norm of $\hat{\bD}_i^{-1}, \bD_i^{-1}$ are bounded on this event.  Therefore, 
we may ignore $\alpha_i$ inside $J_1$ when calculating the corresponding moment such as \eqref{00vv6} below for simplicity. Moreover, an indicator function of such an event should be added inside the expectation to handle the inverse of the matrices of interest, but for notational simplicity, we omit it as in the earlier section. A direct calculation indicates that
\begin{equation}\label{00vv6}\begin{split}E| \sqrt{n}&\sum_{i=1}^n  \bw_1^\intercal\hat{\bX_i}\hat{\bX_i}^{\intercal}\hat{\bD}_i^{-1}\tilde{\bSigma}_1\tilde{\bx}_i \bx_i^\intercal \bD_i^{-1}\bw_1|^2  \leq  n\sum_{i=1}^n E|\bw_1^\intercal\hat{\bX_i}\hat{\bX_i}^{\intercal}\hat{\bD}_i^{-1}\tilde{\bSigma}_1\tilde{\bx}_i \bx_i^\intercal \bD_i^{-1}\bw_1|^2\\
&+  n\sum_{i_1\neq i_2}E \Big(\bw_1^\intercal\hat{\bX}_{i_1}\hat{\bX}_{i_1}^{\intercal}\hat{\bD}_{i_1}^{-1}\tilde{\bSigma}_1\tilde{\bx}_{i_1} \bx_{i_1}^\intercal \bD_{i_1}^{-1}\bw_1\times \bw_1^\intercal\hat{\bX}_{i_2}\hat{\bX}_{i_2}^{\intercal}\hat{\bD}_{i_2}^{-1}\tilde{\bSigma}_1\tilde{\bx}_{i_2} \bx_{i_2}^\intercal \bD_{i_2}^{-1}\bw_1\Big).
\end{split}\end{equation}
It is easy to see the first term above converges to zero by the fact that
$$E| \bx_i^\intercal \bD_i^{-1}\bw_1\bw_1^\intercal\hat{\bX_i}\hat{\bX_i}^{\intercal}\hat{\bD}_i^{-1}\tilde{\bSigma}_1\tilde{\bx}_i |^2=o(n^{-2}). $$
For the second term,  by further extracting $\bx_{i_1}$ from $\hat{\bX}_{i_2}$ and $\hat{\bD}_{i_2}$ we have
\begin{equation}\begin{aligned}
&\bw_1^\intercal\hat{\bX}_{i_2}\hat{\bX}_{i_2}^{\intercal}\hat{\bD}_{i_2}^{-1}\tilde{\bSigma}_1\tilde{\bx}_{i_2} \bx_{i_2}^\intercal \bD_{i_2}^{-1}\bw_1 = \bw_1^\intercal (\hat{\bX}_{i_1i_2}\hat{\bX}_{i_1i_2}^\intercal + \hat{\bx}_{i_1}\hat{\bx}_{i_1}^\intercal)\\
&\relphantom{EE}\times (\hat{\bD}_{i_1i_2}^{-1}+\hat{\bD}_{i_1i_2}^{-1}\tilde{\bSigma}_1 \hat{\bx}_{i_1}\hat{\bx}_{i_1}^\intercal \hat{\bD}_{i_2}^{-1})\tilde{\bSigma}_1\tilde{\bx}_{i_2} \bx_{i_2}^\intercal (\bD_{i_1i_2}^{-1} +\bD_{i_1i_2}^{-1} \tilde{\bSigma}_1 \bx_{i_1}\bx_{i_1}^\intercal \bD_{i_2}^{-1})\bw_1.
\end{aligned}
\end{equation}
Plugging the above expansion into the second term of (\ref{00vv6}) we have eight terms. We below consider only one of eight terms
\begin{equation}\Delta_1= n\sum_{i_1\neq i_2}E \bw_1^\intercal\hat{\bX}_{i_1}\hat{\bX}_{i_1}^{\intercal}\hat{\bD}_{i_1}^{-1}\tilde{\bSigma}_1\tilde{\bx}_{i_1} \bx_{i_1}^\intercal \bD_{i_1}^{-1}\bw_1\times \bw_1^\intercal\hat{\bX}_{i_1i_2}\hat{\bX}_{i_1i_2}^{\intercal}\hat{\bD}_{i_1i_2}^{-1}\tilde{\bSigma}_1\tilde{\bx}_{i_2} \bx_{i_2}^\intercal \bD_{i_1i_2}^{-1}\bw_1,
\end{equation}
and the other seven terms can be estimated similarly.

\indent By calculating expectation with respect to $\bx_{i_1}$, we find  \begin{equation}\label{bc8u2}\Delta_1= n\sum_{i_1\neq i_2}E\tilde{x}_{11}^2 E\bw_1^\intercal\hat{\bX}_{i_1}\hat{\bX}_{i_1}^{\intercal}\hat{\bD}_{i_1}^{-1}\tilde{\bSigma}_1 \bD_{i_1}^{-1}\bw_1\times \bw_1^\intercal\hat{\bX}_{i_1i_2}\hat{\bX}_{i_1i_2}^{\intercal}\hat{\bD}_{i_1i_2}^{-1}\tilde{\bSigma}_1\tilde{\bx}_{i_2} \bx_{i_2}^\intercal \bD_{i_1i_2}^{-1}\bw_1.
\end{equation}
Since we have $E\tilde{x}_{11}^2 = o(n^{-3/2})$ by \eqref{348c1} we just need to show that the summation term in \eqref{bc8u2} is $o(n^{1/2})$. To prove this, we extract  $\bx_{i_2}$ from $\bw_1^\intercal\hat{\bX}_{i_1}\hat{\bX}_{i_1}^{\intercal}\hat{\bD}_{i_1}^{-1}\tilde{\bSigma}_1 \bD_{i_1}^{-1}\bw_1$, and decompose $\Delta_1$ into eight terms:
\begin{eqnarray*}\begin{split}
&\Delta_{11}=\sum_{i_1\neq i_2}E \bw_1^\intercal\hat{\bX}_{i_1i_2}\hat{\bX}_{i_1i_2}^{\intercal}\hat{\bD}_{i_1i_2}^{-1}\tilde{\bSigma}_1 \bD_{i_1i_2}^{-1}\bw_1\times \bw_1^\intercal\hat{\bX}_{i_1i_2}\hat{\bX}_{i_1i_2}^{\intercal}\hat{\bD}_{i_1i_2}^{-1}\tilde{\bSigma}_1\tilde{\bx}_{i_2} \bx_{i_2}^\intercal \bD_{i_1i_2}^{-1}\bw_1,\\
&\Delta_{12}=\sum_{i_1\neq i_2}E \bw_1^\intercal\hat{\bx}_{i_2}\hat{\bx}_{i_2}^{\intercal}\hat{\bD}_{i_1i_2}^{-1}\tilde{\bSigma}_1 \bD_{i_1i_2}^{-1}\bw_1\times \bw_1^\intercal\hat{\bX}_{i_1i_2}\hat{\bX}_{i_1i_2}^{\intercal}\hat{\bD}_{i_1i_2}^{-1}\tilde{\bSigma}_1\tilde{\bx}_{i_2} \bx_{i_2}^\intercal \bD_{i_1i_2}^{-1}\bw_1,\\
&\Delta_{13}=\sum_{i_1\neq i_2}E \hat{\alpha}_{i_1i_2}\bw_1^\intercal\hat{\bX}_{i_1i_2}\hat{\bX}_{i_1i_2}^{\intercal}\hat{\bD}_{i_1i_2}^{-1}\tilde{\bSigma}_1\hat{\bx}_{i_2} \bx_{i_2}^\intercal \hat{\bD}_{i_1i_2}^{-1}\tilde{\bSigma}_1 \bD_{i_1i_2}^{-1}\bw_1 \bw_1^\intercal\hat{\bX}_{i_1i_2}\hat{\bX}_{i_1i_2}^{\intercal}\hat{\bD}_{i_1i_2}^{-1}\tilde{\bSigma}_1\tilde{\bx}_{i_2} \bx_{i_2}^\intercal \bD_{i_1i_2}^{-1}\bw_1,\\
&\Delta_{14}=\sum_{i_1\neq i_2}E \hat{\alpha}_{i_1i_2}\bw_1^\intercal\hat{\bx}_{i_2}\hat{\bx}_{i_2}^{\intercal}\hat{\bD}_{i_1i_2}^{-1}\tilde{\bSigma}_1\hat{\bx}_{i_2} \bx_{i_2}^\intercal \hat{\bD}_{i_1i_2}^{-1}\tilde{\bSigma}_1 \bD_{i_1i_2}^{-1}\bw_1\times \bw_1^\intercal\hat{\bX}_{i_1i_2}\hat{\bX}_{i_1i_2}^{\intercal}\hat{\bD}_{i_1i_2}^{-1}\tilde{\bSigma}_1\tilde{\bx}_{i_2} \bx_{i_2}^\intercal \bD_{i_1i_2}^{-1}\bw_1,\\
&\Delta_{15}=\sum_{i_1\neq i_2}E \alpha_{i_1i_2}\bw_1^\intercal\hat{\bX}_{i_1i_2}\hat{\bX}_{i_1i_2}^{\intercal}\hat{\bD}_{i_1i_2}^{-1}\tilde{\bSigma}_1 \bD_{i_1i_2}^{-1}\tilde{\bSigma}_1 \bx_{i_2} \bx_{i_2}^\intercal \bD_{i_1i_2}^{-1}\bw_1\\ &\relphantom{EEEEEE} \times \bw_1^\intercal\hat{\bX}_{i_1i_2}\hat{\bX}_{i_1i_2}^{\intercal}\hat{\bD}_{i_1i_2}^{-1}\tilde{\bSigma}_1\tilde{\bx}_{i_2} \bx_{i_2}^\intercal \bD_{i_1i_2}^{-1}\bw_1,\\
&\Delta_{16}=\sum_{i_1\neq i_2}E \alpha_{i_1i_2}\bw_1^\intercal\hat{\bx}_{i_2}\hat{\bx}_{i_2}^{\intercal}\hat{\bD}_{i_1i_2}^{-1}\tilde{\bSigma}_1 \bD_{i_1i_2}^{-1}\tilde{\bSigma}_1 \bx_{i_2} \bx_{i_2}^\intercal \bD_{i_1i_2}^{-1}\bw_1\times \bw_1^\intercal\hat{\bX}_{i_1i_2}\hat{\bX}_{i_1i_2}^{\intercal}\hat{\bD}_{i_1i_2}^{-1}\tilde{\bSigma}_1\tilde{\bx}_{i_2} \bx_{i_2}^\intercal \bD_{i_1i_2}^{-1}\bw_1,\\
&\Delta_{17}=\sum_{i_1\neq i_2}E \hat{\alpha}_{i_1i_2}\alpha_{i_1i_2}\bw_1^\intercal\hat{\bX}_{i_1i_2}\hat{\bX}_{i_1i_2}^{\intercal}\hat{\bD}_{i_1i_2}^{-1}\tilde{\bSigma}_1\hat{\bx}_{i_2} \hat{\bx}_{i_2}^\intercal \hat{\bD}_{i_1i_2}^{-1}\tilde{\bSigma}_1 \bD_{i_1i_2}^{-1}\tilde{\bSigma}_1 \bx_{i_2} \bx_{i_2}^\intercal \bD_{i_1i_2}^{-1}\bw_1\\&\relphantom{EEEEEE}  \times \bw_1^\intercal\hat{\bX}_{i_1i_2}\hat{\bX}_{i_1i_2}^{\intercal}\hat{\bD}_{i_1i_2}^{-1}\tilde{\bSigma}_1\tilde{\bx}_{i_2} \bx_{i_2}^\intercal \bD_{i_1i_2}^{-1}\bw_1,\\
&\Delta_{18}=\sum_{i_1\neq i_2}E \hat{\alpha}_{i_1i_2}\alpha_{i_1i_2}\bw_1^\intercal\hat{\bx}_{i_2}\hat{\bx}_{i_2}^{\intercal}\hat{\bD}_{i_1i_2}^{-1}\tilde{\bSigma}_1\hat{\bx}_{i_2} \hat{\bx}_{i_2}^\intercal \hat{\bD}_{i_1i_2}^{-1}\tilde{\bSigma}_1 \bD_{i_1i_2}^{-1}\tilde{\bSigma}_1 \bx_{i_2} \bx_{i_2}^\intercal \bD_{i_1i_2}^{-1}\bw_1\\&\relphantom{EEEEEE}  \times \bw_1^\intercal\hat{\bX}_{i_1i_2}\hat{\bX}_{i_1i_2}^{\intercal}\hat{\bD}_{i_1i_2}^{-1}\tilde{\bSigma}_1\tilde{\bx}_{i_2} \bx_{i_2}^\intercal \bD_{i_1i_2}^{-1}\bw_1.
\end{split}\end{eqnarray*} Next we consider above eight terms one by one.

\indent For $\Delta_{11}$, by calculating expectation with respect to $\bx_{i_2}$, we have
$$|\Delta_{11}| =  |\sum_{i_1\neq i_2}E\tilde{x}_{11}^2 E\bw_1^\intercal\hat{\bX}_{i_1i_2}\hat{\bX}_{i_1i_2}^{\intercal}\hat{\bD}_{i_1i_2}^{-1}\tilde{\bSigma}_1 \bD_{i_1i_2}^{-1}\bw_1\times \bw_1^\intercal\hat{\bX}_{i_1i_2}\hat{\bX}_{i_1i_2}^{\intercal}\hat{\bD}_{i_1i_2}^{-1}\tilde{\bSigma}_1 \bD_{i_1i_2}^{-1}\bw_1|=o(n^{1/2}).$$
We conclude that $|\Delta_{12}|=o(1)$ from  
\begin{equation}\begin{aligned}
&E|\bx_{i_2}^\intercal \bD_{i_1i_2}^{-1}\bw_1\bw_1^\intercal\hat{\bx}_{i_2}|^2=O(n^{-2}),\\
&E|\hat{\bx}_{i_2}^{\intercal}\hat{\bD}_{i_1i_2}^{-1}\tilde{\bSigma}_1 \bD_{i_1i_2}^{-1}\bw_1\times \bw_1^\intercal\hat{\bX}_{i_1i_2}\hat{\bX}_{i_1i_2}^{\intercal}\hat{\bD}_{i_1i_2}^{-1}\tilde{\bSigma}_1\tilde{\bx}_{i_2}|^2=o(n^{-2}).
\end{aligned}\end{equation}
For $\Delta_{13}$, we find
   \begin{equation}\label{4236q}\begin{split}|\Delta_{13}|&\leq \sum_{i_1\neq i_2}(E|\bx_{i_2}^\intercal \hat{\bD}_{i_1i_2}^{-1}\tilde{\bSigma}_1 \bD_{i_1i_2}^{-1}\bw_1\times \bw_1^\intercal\hat{\bX}_{i_1i_2}\hat{\bX}_{i_1i_2}^{\intercal}\hat{\bD}_{i_1i_2}^{-1}\tilde{\bSigma}_1\tilde{\bx}_{i_2}|^2)^{1/2}\\
&\times (E| \bx_{i_2}^\intercal \bD_{i_1i_2}^{-1}\bw_1\bw_1^\intercal\hat{\bX}_{i_1i_2}\hat{\bX}_{i_1i_2}^{\intercal}\hat{\bD}_{i_1i_2}^{-1}\tilde{\bSigma}_1\hat{\bx}_{i_2}|^2)^{1/2}\\&=o(1).
\end{split}\end{equation}
 It follows from
\begin{equation}\label{kfc9p}\begin{aligned}
&E|\hat{\bx}_{i_2}^{\intercal}\hat{\bD}_{i_1i_2}^{-1}\tilde{\bSigma}_1\hat{\bx}_{i_2}|^4 = O(1),\\
&E| \bx_{i_2}^\intercal \hat{\bD}_{i_1i_2}^{-1}\tilde{\bSigma}_1 \bD_{i_1i_2}^{-1}\bw_1\times \bw_1^\intercal\hat{\bX}_{i_1i_2}\hat{\bX}_{i_1i_2}^{\intercal}\hat{\bD}_{i_1i_2}^{-1}\tilde{\bSigma}_1\tilde{\bx}_{i_2}|^2=o(n^{-2}),\\
&E| \bx_{i_2}^\intercal \bD_{i_1i_2}^{-1}\bw_1\bw_1^\intercal\hat{\bx}_{i_2}|^4=o(n^{-3}).
\end{aligned}\end{equation}
that $|\Delta_{14}|=o(n^{1/4})$. By the same argument as \eqref{4236q} $\Delta_{15}$ is of order $o(1)$.  For $\Delta_{16}$, one can verify that
\begin{equation}
E|\hat{\bx}_{i_2}^{\intercal}\hat{\bD}_{i_1i_2}^{-1}\tilde{\bSigma}_1 \bD_{i_1i_2}^{-1}\tilde{\bSigma}_1 \bx_{i_2}|^4 = O(1).
\end{equation}
This, together with \eqref{kfc9p}, implies
 \begin{equation}\begin{aligned}
|\Delta_{16}|&\leq \sum_{i_1\neq i_2}(E|\hat{\bx}_{i_2}^{\intercal}\hat{\bD}_{i_1i_2}^{-1}\tilde{\bSigma}_1 \bD_{i_1i_2}^{-1}\tilde{\bSigma}_1 \bx_{i_2} |^4)^{1/4}\times (E| \bx_{i_2}^\intercal \bD_{i_1i_2}^{-1}\bw_1 \bw_1^\intercal\hat{\bX}_{i_1i_2}\hat{\bX}_{i_1i_2}^{\intercal}\hat{\bD}_{i_1i_2}^{-1}\tilde{\bSigma}_1\tilde{\bx}_{i_2} |^2)^{1/2}\\&\relphantom{HHHH} \times (E |\bx_{i_2}^\intercal \bD_{i_1i_2}^{-1}\bw_1w_1^\intercal\hat{\bx}_{i_2}|^4)^{1/4}\\
&=o(n^{1/4}).
\end{aligned}\end{equation}
For $\Delta_{17}$ and $\Delta_{18}$, we have
\begin{equation}\begin{aligned}
|\Delta_{17}|&\leq \sum_{i_1\neq i_2}(E|\hat{\bx}_{i_2}^\intercal \hat{\bD}_{i_1i_2}^{-1}\tilde{\bSigma}_1 \bD_{i_1i_2}^{-1}\tilde{\bSigma}_1 \bx_{i_2} |^4)^{1/4}\times (E|\bx_{i_2}^\intercal \bD_{i_1i_2}^{-1}\bw_1  \bw_1^\intercal\hat{\bX}_{i_1i_2}\hat{\bX}_{i_1i_2}^{\intercal}\hat{\bD}_{i_1i_2}^{-1}\tilde{\bSigma}_1\tilde{\bx}_{i_2}|^2)^{1/2}\\&\relphantom{HHHH} \times (E | \bx_{i_2}^\intercal \bD_{i_1i_2}^{-1}\bw_1\bw_1^\intercal\hat{\bX}_{i_1i_2}\hat{\bX}_{i_1i_2}^{\intercal}\hat{\bD}_{i_1i_2}^{-1}\tilde{\bSigma}_1\hat{\bx}_{i_2}|^4)^{1/4}\\
&=o(n^{1/4}),\\
|\Delta_{18}|&\leq\sum_{i_1\neq i_2}(E|\hat{\bx}_{i_2}^{\intercal}\hat{\bD}_{i_1i_2}^{-1}\tilde{\bSigma}_1\hat{\bx}_{i_2} |^4)^{1/4}\times (E| \hat{\bx}_{i_2}^\intercal \hat{\bD}_{i_1i_2}^{-1}\tilde{\bSigma}_1 \bD_{i_1i_2}^{-1}\tilde{\bSigma}_1 \bx_{i_2} |^4)^{1/4}\\&\relphantom{HHHH} \times (E | \bx_{i_2}^\intercal \bD_{i_1i_2}^{-1}\bw_1   \bw_1^\intercal\tilde{\bx}_{i_2}|^4)^{1/4}\times (E|\bx_{i_2}^\intercal H\bD_{i_1i_2}^{-1}\bw_1\bw_1^\intercal\hat{\bx}_{i_2}|^4)^{1/4}\\&=o(n^{1/4}).
\end{aligned}\end{equation}
Combining above arguments we conclude that $\Delta_1 \rightarrow 0$ and therefore the third term in \eqref{45non} is $o_p(1)$. 

Let $\breve{x}_{ij} = \frac{\hat{x}_{ij}}{\sqrt{n}\sigma_n}$, with $\sigma_n^2=E|\hat{x}_{ij}|^2$, $\breve{\bX}=(\breve{x}_{ij})$. We have \begin{equation}\begin{aligned}&\sqrt{n}(\bw_1^{\intercal}\hat{\bX}(\bI-\hat{\bX}^{\intercal}\tilde{\bSigma}_1\hat{\bX})^{-1}\hat{\bX}^{\intercal}\bw_1-\bw_1^{\intercal}\breve{\bX}(\bI-\breve{\bX}^{\intercal}\tilde{\bSigma}_1\breve{\bX})^{-1}\breve{\bX}^{\intercal}\bw_1)\\&=\sqrt{n}(n\sigma_n^2-1)\bw_1^{\intercal}\hat{\bX}(n\sigma_n^2 \bI-\hat{\bX}^{\intercal}\tilde{\bSigma}_1\hat{\bX})^{-1}(\bI-\hat{\bX}^{\intercal}\tilde{\bSigma}_1\hat{\bX})^{-1}\hat{\bX}^{\intercal}\bw_1=o_p(1),
\end{aligned}\end{equation}
where we use $1-n\sigma_n^2 = Eq_{11}^2 I ({|q_{11}|>\eta_n n^{1/4}})+E^2q_{11}I(|q_{11}| >\eta_n n^{1/4})=o(n^{-1/2})$.\\
\indent Up to now, we conclude that truncation at $\eta_n n^{1/4}$ and centralization do not influence the limiting distribution of the statistics in Theorem \ref{Th21}.\\

\subsection{Proof of Lemma \ref{lem31}}
\begin{proof}
For simplicity, we omit the index of population, thus denote the population covariance by $\bSigma$ and the sample covariance by $\bS=\frac{1}{n}\bY\bY^\intercal =\bSigma^{\frac{1}{2}} \bX\bX^\intercal \bSigma^{\frac{1}{2}}$. To conclude \eqref{aah92} it suffices to show the following two facts:
\begin{eqnarray}\label{apt21}  &\max_{1\leq i\leq p}|\bS_{ii}-\bSigma_{ii}|\stackrel{D}\rightarrow 0,\\
&\frac{\sum_{i=1}^p|\bS_{ii}+\bSigma_{ii}|}{p}=O_p(1). \label{apt22}
\end{eqnarray}
The proof of (\ref{apt22}) is trivial, since $$\frac{\sum_{i=1}^p|\bS_{ii}+\bSigma_{ii}|}{p}=\frac{tr(\bS)+tr(\bSigma)}{p}=O_p(1). $$

To derive (\ref{apt21}) we first truncate and centralize $\bX$. Select $\eta_n \rightarrow 0$ such that $\eta_n^{-4}\int_{|q_{11}| > \eta_n n^{1/2}} |q_{11}|^4\rightarrow 0$. Let $\hat{\bX}' = (\hat{x}_{ij}')_{p\times n}$ where $\hat{x}_{ij}' = \frac{1}{\sqrt{n}}q_{ij}I(|q_{ij}| \leq \eta_n n^{1/2})$, and $\hat{\bS}' =\bSigma^{1/2}\hat{\bX}' (\hat{\bX}')^\intercal \bSigma^{1/2} $.  Then define the truncated and centralized matrix $\breve{\bX}'= (\breve{x}_{ij}')_{p\times n}$, where $\breve{x}_{ij}' = (\hat{x}_{ij}' -E\hat{x}_{ij}' )/(\sqrt{n}\sigma_n')$, with $\sigma_n'=E^{1/2}|\hat{x}_{ij}' -E\hat{x}_{ij}'|^2$, and  $\breve{\bS}' =\bSigma^{1/2}\breve{\bX}' (\breve{\bX}')^{\intercal}\bSigma^{1/2} $.
We have $$P(\bS \neq \hat{\bS}') \leq pn P(|q_{11}|\geq \eta_n\sqrt{n})  =o(1). $$
It follows that for any $\epsilon > 0$,
$$P(\max_{1\leq i \leq p} |\bS_{ii}-\hat{\bS}'_{ii}| >\epsilon)=o(1) $$
For any deterministic unit vector $\bs_1$, we have \begin{eqnarray*}
\begin{aligned}
&E|\bs_1^\intercal  \hat{\bS}'\bs_1-\bs_1^\intercal \breve{\bS}' \bs_1| \\&\leq(E|\bs_1^\intercal \bSigma^{1/2}( \hat{\bX'}-\breve{\bX'})(\hat{\bX'})^\intercal \bSigma^{1/2} \bs_1| + E|\bs_1^\intercal \bSigma^{1/2} \breve{\bX'} ( \hat{\bX'}-\breve{\bX'})^\intercal \bSigma^{1/2} \bs_1|)\\
&\leq C (\frac{|\sqrt{n}\sigma_n' -1|}{\sqrt{n}\sigma_n'} + \frac{|E\hat{x}'_{11}|}{\sigma_n'}) = o(n^{-1})
\end{aligned}
\end{eqnarray*}
where the last step uses two facts that $|n\sigma_n'^2-1| \leq 2\int_{\{|q_{11}|\geq \eta_n n^{1/2}\}}|q_{11}|^2= o(n^{-1})$, and that $|E\hat{x}'_{11}| \leq \int_{\eta_n n^{1/2}}^\infty P(|q_{11}|\geq t)dt= o(n^{-3/2})$.
From above estimates, we obtain for any $\epsilon > 0$,
$$P(\max_{1\leq i \leq p} |\hat{\bS}'_{ii}-\breve{\bS}'_{ii}| >\epsilon)\leq \sum_{i=1}^p \frac{E|\be_i^\intercal  \hat{\bS}' \be_i-\be_i^\intercal \breve{\bS}' \be_i|}{\epsilon}=o(1). $$
Therefore we below assume the underlying variables are truncated at $\eta_n n^{1/2}$ and centralized.

\indent Then for any deterministic unit vector $\bs_1$, using Lemma 2.1 in \cite{bai2004}, and the Burkh\"{o}lder's inequality,  we have
\begin{eqnarray}
\begin{aligned}
&E|\bs_1^\intercal \bSigma^{1/2}\bX\bX^\intercal \bSigma^{1/2}\bs_1-\bs_1^\intercal \bSigma \bs_1|^4\\&\leq E|\sum_{k=1}^n (E_k-E_{k-1})\bs_1^\intercal \bSigma^{1/2}\bX\bX^\intercal \bSigma^{1/2}\bs_1|^4
\\&\leq  E|\sum_{k=1}^n (E_k-E_{k-1}) (\bs_1^\intercal \bSigma^{1/2}\bx_k\bx_k^\intercal \bSigma^{1/2}\bs_1)|^4
\\&=n \sum_{k=1}^n E|\bs_1^\intercal \bSigma^{1/2}\bx_k\bx_k^\intercal \bSigma^{1/2}\bs_1|^4 = o(n^{-1}).
\end{aligned}
\end{eqnarray}
This implies (\ref{apt21}) and hence \eqref{aah92}.

Next we show that \eqref{aah93} holds. The law of large numbers ensures that
 \begin{equation}\label{gac5j}
\frac{1}{pn}\sum_{i=1}^n (\by_i^\intercal \bSigma_1\by_i - tr \bSigma_1)^2 - \mathfrak{M} \stackrel{i.p.} \rightarrow 0.
\end{equation}
Note that
\begin{equation}\label{34lv3}\begin{aligned}
&\left|\frac{1}{pn}\sum_{i=1}^n (\by_i^\intercal \bSigma_1 \by_i - tr \bS_1)^2 -\frac{1}{pn}\sum (\by_i^\intercal \bSigma_1 \by_i - tr \bSigma_1)^2\right| \\
&\leq \frac{2}{pn}\sum_{i=1}^n \left|\left(\by_i^\intercal \bSigma_1 \by_i -tr \bSigma_1\right)\left(tr \bS_1-tr\bSigma_1 \right)\right|  +\frac{1}{p}|tr \bS_1 -tr \bSigma_1|^2.
\end{aligned}\end{equation}
The second term is $o_p(1)$ by  the law of large numbers. 
Also from the law of large numbers, \begin{equation}\label{890bm}
\frac{1}{n}\sum_{i=1}^n \frac{1}{\sqrt{p}}|\by_i^\intercal \bSigma_1 \by_i -tr \bSigma_1| - E \frac{1}{\sqrt{p}}|\by_i^\intercal \bSigma_1 \by_i -tr \bSigma_1|
\stackrel{i.p.}\rightarrow0.
\end{equation}
Moreover (\ref{543ah}) implies that $E \frac{1}{\sqrt{p}}|\by_i^\intercal \bSigma_1 \by_i -tr \bSigma_1|$ is bounded and  from Lemma \ref{Lem4.2} $tr \bS_1 -tr \bSigma_1$ is asymptotically normal.
These imply that the first term in \eqref{34lv3} is also $o_p(1)$. Thus we have
\begin{equation}\label{402n9}
\left|\frac{1}{pn}\sum_{i=1}^n (\by_i^\intercal \bSigma_1 \by_i - tr \bS_1)^2 -\frac{1}{pn}\sum_{i=1}^n (\by_i^\intercal \bSigma_1 \by_i - tr \bSigma_1)^2\right| \stackrel{i.p.}\rightarrow 0.
\end{equation}
From \eqref{gac5j} and \eqref{402n9}, we conclude \eqref{aah93}.
\end{proof}

\subsection{Proof of Lemma \ref{Evbhighprob}}
\indent First we show that $\mathcal{B}_1$ holds with high probability. By Lemma 3.1 in \cite{bai2012}, $\psi(\alpha)$ lies outside the support of $F^{c,H}$ and $F^{c_n,H_n}$. Thus we can select a small interval $[a,b]$ containing $\psi(\alpha)$ and $\psi_n(\alpha)$ for sufficiently large $n$, such that this interval lies outside the support of $F^{c,H}$ and $F^{c_n,H_n}$. By carefully checking on the proofs of \cite{bai1998no} one can show that under the truncation at $\eta_n n^{1/4}$ no eigenvalues of $\bX^\intercal \bSigma_{1P} \bX$ lies in $[a,b]$ with high probability.  Therefore, $\mathcal{B}_1$ holds with high probability. Actually, we note that proof of Lemma C.3 in \cite{jiang2019} also uses this conclusion.  Following a similar argument $\mathcal{B}_{1k}$ and $\mathcal{B}_{1jk}$ hold with high probability.

We next show that the event $\mathcal{B}_{2k}$ and  $\mathcal{B}_{3k}$ hold with high probability. For any small positive constant $\epsilon$, and large positive constant $l$, using \eqref{lar11}, we find
 \begin{equation}\label{9b3la}
P\left(|\bx_k^\intercal \tilde{\bSigma}_1 ({\bD_k^\intercal})^{-1}\bx_k I(\mathcal{B}_{1k})-\frac{1}{n}tr\tilde{\bSigma}_1 ({\bD_k^\intercal})^{-1}I(\mathcal{B}_{1k})|> \epsilon \right)\leq n^{-l}.
\end{equation}
By Burkholder's inequality
\begin{equation*}\begin{aligned}&E|tr \frac{1}{n}\tilde{\bSigma}_1(\bD_1^{\intercal})^{-1}I(\mathcal{B}_{11})-E  tr
\frac{1}{n}\tilde{\bSigma}_1(\bD_1^{\intercal})^{-1}I(\mathcal{B}_{11})|^p\\&
\leq \frac{1}{n^p}E|\sum_{j=2}^{n}(E_j-E_{j-1})[tr\tilde{\bSigma}_1(\bD_1^{\intercal})^{-1}I(\mathcal{B}_{11})- tr\tilde{\bSigma}_1 (\bD_{1j}^{\intercal})^{-1}I(\mathcal{B}_{11j})|^p \\
& \leq \frac{C_p}{n^p}E(\sum_{j=2}^n|(E_j-E_{j-1})\bx_j^{\intercal} \tilde{\bSigma}_1 (\bD_1^\intercal)^{-1}\tilde{\bSigma}_1 (\bD_{1j}^\intercal)^{-1}\bx_j I(\mathcal{B}_{11}\mathcal{B}_{11j})|^2)^{\frac{p}{2}}\\
&\leq \frac{C_p}{n^{p/2+1}}\sum_{j=2}^{n}E|(E_j-E_{j-1})\bx_j^{\intercal} \tilde{\bSigma}_1  (\bD_1^\intercal)^{-1}\tilde{\bSigma}_1 (\bD_{1j}^\intercal)^{-1}\bx_j I(\mathcal{B}_{11}\mathcal{B}_{11j})|^{p}\\&\leq \frac{C_p}{n^{p/2}}E|\bx_j^{\intercal} \tilde{\bSigma}_1 (\bD_{1}^\intercal)^{-1}\tilde{\bSigma}_1  (\bD_{1j}^\intercal)^{-1}\bx_j I(\mathcal{B}_{11}\mathcal{B}_{11j})|^{p}\\&\leq\frac{C_p}{n^{p/2}} E|\bx_j^{\intercal}\bx_j|^p =O(n^{-\frac{p}{2}}),
\end{aligned}
\end{equation*}
where we use the fact that $\mathcal{B}_{11}\subseteq\mathcal{B}_{11j}$. 
This implies that
\begin{equation}\label{2gm91}
P\left(|\frac{1}{n}tr\tilde{\bSigma}_1 ({\bD_k^\intercal})^{-1}I(\mathcal{B}_{1k})-E\frac{1}{n}tr\tilde{\bSigma}_1 ({\bD_k^\intercal})^{-1}I(\mathcal{B}_{1k})|> \epsilon \right)\leq n^{-l}.
\end{equation}
We claim that 
\begin{eqnarray}\label{rg42m}
E\Big(\frac{1}{n}tr\tilde{\bSigma}_1 ({\bD_k^\intercal})^{-1}I(\mathcal{B}_{1k})\Big)\rightarrow 1 +\frac{1}{\theta \underline{m}(\theta)}, \; \text{as} \; n\rightarrow \infty,  \end{eqnarray}
(to be proved later).
We conclude from \eqref{9b3la},\eqref{2gm91} and \eqref{rg42m} that
\begin{equation}
P\left(\left|\bx_k^\intercal \tilde{\bSigma}_1 ({\bD_k^\intercal})^{-1}\bx_k I(\mathcal{B}_{1k})- ( 1 +\frac{1}{\theta \underline{m}(\theta)})\right| > \epsilon \right) =o(n^{-l}),
\end{equation}
and \begin{equation}
P\left(\left|\frac{1}{n}tr\tilde{\bSigma}_1 ({\bD_k^\intercal})^{-1}I(\mathcal{B}_{1k})- ( 1 +\frac{1}{\theta \underline{m}(\theta)})\right| > \epsilon \right) =o(n^{-l}).
\end{equation}
Thus the events $\mathcal{B}_{2k}$ and $\mathcal{B}_{3k}$ hold with high probability. Using the similar arguments we conclude that $\mathcal{B}_{2jk}$ and $\mathcal{B}_{3jk}$ hold with high probability. To conclude, the event $\mathcal{B}$ holds with high probability.

Finally we show \eqref{rg42m}.  Lemma 3.3 in \cite{bai1999exact}, together with the fact $\bx_k^\intercal \tilde{\bSigma}_1 ({\bD_k^\intercal})^{-1}\bx_k=\bx_k^\intercal \btSigma_1^{1/2}(\bI-\btSigma_1^{1/2}\bX_k\bX_k^*\btSigma_1^{1/2})^{-1}\btSigma_1^{1/2}\bx_k $, ensures that \begin{equation}\label{54ncq}
\bx_k^\intercal \tilde{\bSigma}_1 ({\bD_k^\intercal})^{-1}\bx_k I(\mathcal{B}_{1k}) \stackrel{a.s.}\rightarrow 1 +\frac{1}{\theta \underline{m}(\theta)}, \; \text{as} \; n\rightarrow \infty.
\end{equation}
We need to be careful here that although Lemma 3.3 in \cite{bai1999exact} is derived under the truncation of entries of $\bX$ at a constant, we can still get the same conclusion using our truncation.
From \eqref{9b3la}, we see that   \begin{equation}\label{7n2hp}|\bx_k^\intercal \tilde{\bSigma}_1 ({\bD_k^\intercal})^{-1}\bx_k I(\mathcal{B}_{1k})-\frac{1}{n}tr\tilde{\bSigma}_1 ({\bD_k^\intercal})^{-1}I(\mathcal{B}_{1k})|\stackrel{a.s.}\rightarrow 0. \end{equation}
From \eqref{54ncq} and \eqref{7n2hp}, it follows that \begin{eqnarray}\label{32ndr}\frac{1}{n}tr\tilde{\bSigma}_1 ({\bD_k^\intercal})^{-1}I(\mathcal{B}_{1k})\stackrel{a.s.}\rightarrow 1 +\frac{1}{\theta \underline{m}(\theta)}.\end{eqnarray}
The dominated convergence theorem implies \eqref{rg42m}.

\end{document}